\documentclass[11pt,twoside]{article}
\usepackage{simplewick}
\usepackage{times}
\usepackage{amsmath,amssymb,amsthm}
\usepackage{color}
\usepackage{hep}
\usepackage{mathrsfs}
\usepackage{leftidx}
\usepackage{graphicx}
\usepackage{float}
\usepackage{enumerate}
\usepackage{titletoc}
\usepackage{epstopdf}
\usepackage{ulem}
\usepackage{mathrsfs}
\usepackage{appendix}
\allowdisplaybreaks
\DeclareSymbolFont{EulerExtension}{U}{euex}{m}{n}
\DeclareMathSymbol{\euintop}{\mathop} {EulerExtension}{"52}
\DeclareMathSymbol{\euointop}{\mathop} {EulerExtension}{"48}

\def\al{\alpha}
\def\be{\beta}
\def\de{\delta}
\def\De{\Delta}
\def\ep{\epsilon}
\def\et{\eta}
\def\ga{\gamma}
\def\Ga{\Gamma}
\def\ka{\kappa}
\def\la{\lambda}
\def\La{\Lambda}
\def\na{\nabla}
\def\om{\omega}
\def\Om{\Omega}
\def\si{\sigma}
\def\Si{\Sigma}
\def\th{\theta}
\def\ze{\zeta}
\def\ve0{\varepsilon_0}
\def\vp{\varphi}
\def\vt{\vartheta}

\def\Ct{\tilde{C}}

\def\chic{\check{\chi}}

\def\d{\textrm{d}}
\def\dse{\displaystyle}

\def\ds{\slashed{\textrm{d}}}

\def\divg{\textrm{div}_g}
\def\divgs{{\textrm{div}}_{\slashed{g}}}
\def\D{\mathscr{D}}

\def\Des{\slashed{\Delta}}

\def\Et{\tilde{E}}

\def\f{\frac}

\def\geqslant{\geq}
\def\gs{\slashed{g}}

\def\leqslant{\leq}
\def\lot{\text{l.o.t.}}
\def\Lb{\underline{L}}

\def\Lbt{\tilde{\underline{L}}}
\def\Lr{\mathring{L}}
\def\Lt{\tilde{L}}

\def\Lies{\slashed{\mathcal{L}}}

\def\nas{\slashed{\nabla}}

\def\O{\mathcal{O}}
\def\p{\partial}
\def\pis{\slashed{\pi}}

\def\R{\mathscr{R}}
\def\t{\tilde}

\def\trg{{\textrm{tr}}_{g}}
\def\trgs{{\textrm{tr}}_{\slashed{g}}}

\def\Tr{\mathring{T}}

\def\widetilde{\tilde}

\begin{document}
\footskip=0pt
\footnotesep=2pt
\let\oldsection\section
\renewcommand\section{\setcounter{equation}{0}\oldsection}
\renewcommand\thesection{\arabic{section}}
\renewcommand\theequation{\thesection.\arabic{equation}}
\newtheorem{claim}{\noindent Claim}[section]
\newtheorem{theorem}{\noindent Theorem}[section]
\newtheorem{lemma}{\noindent Lemma}[section]
\newtheorem{proposition}{\noindent Proposition}[section]
\newtheorem{definition}{\noindent Definition}[section]
\newtheorem{remark}{\noindent Remark}[section]
\newtheorem{corollary}{\noindent Corollary}[section]
\newtheorem{example}{\noindent Example}[section]

\title{On the critical exponent $p_c$ of the 3D quasilinear wave equation $-\big(1+(\p_t\phi)^p\big)\p_t^2\phi+\De\phi=0$ with
short pulse initial data. I, global existence}
\author{Bingbing Ding$^{1,*}$, \quad Yu  Lu$^{1,*}$, \quad Huicheng Yin$^{1,}$
\footnote{Ding Bingbing (bbding@njnu.edu.cn, 13851929236@163.com), Lu Yu (15850531017@163.com) and Yin Huicheng (huicheng@nju.edu.cn, 05407@njnu.edu.cn)
are supported by the NSFC (No.11731007, No.12071223, No.11601236, No.11971237).}\vspace{0.5cm}\\
\small School of Mathematical Sciences and Mathematical Institute, Nanjing Normal University,\\
\small Nanjing, 210023, China.}
\date{}

\maketitle

\centerline {\bf Abstract} \vskip 0.3 true cm
For the 3D quasilinear wave equation $-\big(1+(\p_t\phi)^p\big)\p_t^2\phi+\De\phi=0$
with the short pulse initial data $(\phi,\p_t\phi)(1,x)=\big(\de^{2-\ve0}\phi_0(\frac{r-1}{\de},\om),
\de^{1-\ve0}\phi_1(\frac{r-1}{\de},\om)\big)$,
where $p\in\mathbb N$, $p\ge 2$, $0<\ve0<1$, $r=|x|$, $\om=
\f{x}{r}\in\mathbb S^2$, and $\de>0$ is sufficiently small, under the  outgoing constraint condition
$(\p_t+\p_r)^k\phi(1,x)=O(\de^{2-\ve0})$ for $k=1,2$,
we will establish the global existence of smooth large data solution $\phi$ when $p>p_c$ with $p_c=\f{1}{1-\ve0}$
being the critical exponent. In the forthcoming paper, when $1\le p\le p_c$, we show the
formation of the outgoing shock before the time $t=2$ under the suitable assumptions of $(\phi_0,\phi_1)$.

\vskip 0.2 true cm {\bf Keywords:} Short pulse initial data, critical exponent, incoming, outgoing,
inverse foliation density,

\qquad \qquad \quad Goursat problem
\vskip 0.2 true cm {\bf Mathematical Subject Classification:} 35L05, 35L72
\vskip 0.4 true cm

\tableofcontents

\section{Introduction}\label{Section 1}

In this paper, we are concerned with  the following 3D quasilinear wave equation
 \begin{equation}\label{main equation}
  -\big(1+(\p_t\phi)^p\big)\p_t^2\phi+\De\phi=0,
 \end{equation}
where $p\in\mathbb N$, $p\ge 2$, $x=(x^1,x^2,x^3)\in\mathbb R^3$, $t\ge 1$, $\p=(\p_t, \p_{x^1}, \p_{x^2}, \p_{x^3})
=(\p_t, \p_1, \p_2, \p_3)$  and $\De=\p_1^2+\p_2^2+\p_3^2$.

Let \eqref{main equation} equip with the small initial data
\begin{align}\label{Y-1}
\phi(1,x)=\de \vp_0(x), \quad \p_t\phi(1,x)=\de \vp_1(x),
\end{align}
where $(\vp_0(x),\vp_1(x))\in C_0^{\infty}(\mathbb R^3)$, $\de>0$ is sufficiently small. Then
\eqref{main equation}-\eqref{Y-1} has a global smooth solution $\phi$ (see \cite{K-P} or Chapter 6 of
\cite{H}).
On the other hand, if $p=1$ in \eqref{main equation}, then the smooth solution $\phi$
of \eqref{main equation}-\eqref{Y-1} will blow up in finite time as long as $(\vp_0(x),\vp_1(x))\not\equiv 0$
(see \cite{A1}, \cite{A2} and \cite{Ding2}).

For $p=2$, let \eqref{main equation} equip with  the short pulse initial data
 \begin{align}\label{Y-2}
  (\phi,\p_t\phi)(1,x)=\big(\de^{\f32}\phi_0(\frac{r-1}{\de},\om),
   \de^{\f12}\phi_1(\frac{r-1}{\de},\om)\big),
   \end{align}
where $r=|x|=\sqrt{(x^1)^2+(x^2)^2+(x^3)^2}$, $\om=(\om_1, \om_2,\om_3)=\frac{x}{r}\in\mathbb S^2$, $\phi_0(s,\om)$
and $\phi_1(s,\om)\in C_0^{\infty}\left((-1,0)\times\mathbb{S}^2\right)$, moreover, such the incoming constraint condition
is posed
 \begin{equation}\label{HC-1}
  (\p_t-\p_r)^k\phi(1,x)=O(\de^{\f32}),\ k=1,2.
 \end{equation}
Then under the suitable assumption of $(\phi_0,\phi_1)$, it is shown in \cite{MY17} that the smooth solution $\phi$ of equation
$-\big(1+(\p_t\phi)^2\big)\p_t^2\phi+\De\phi=0$ will blow up
and further the incoming shock will be formed before the time $t=2$. Here we point out that
the ``short pulse initial data'' (a class of large initial data) are firstly introduced
by D. Christodoulou in
\cite{C09}. For the short pulse data and by the short pulse method, the authors in
monumental papers \cite{C09} and \cite{K-R} showed that the black holes can be formed
in vacuum spacetime and the blowup mechanism is due to the condensation of the gravitational waves for the 3D Einstein general
relativity equations.

Motivated by \cite{MY17},
we now consider \eqref{main equation} with the general short pulse initial data
 \begin{align}\label{initial data}
  (\phi,\p_t\phi)(1,x)=\big(\de^{2-\ve0}\phi_0(\frac{r-1}{\de},\om),
    \de^{1-\ve0}\phi_1(\frac{r-1}{\de},\om)\big),
 \end{align}
where $0<\ve0<1$, and $(\phi_0(s,\om),\ \phi_1(s,\om))\in C_0^{\infty}\left((-1,0)\times\mathbb{S}^2\right)$.
Meanwhile, the following outgoing constraint condition is posed
 \begin{equation}\label{condition on data}
  (\p_t+\p_r)^k\phi(1,x)=O(\de^{2-\ve0}),\ k=1,2.
 \end{equation}

Note that in order to guarantee the strict hyperbolicity of \eqref{main equation},
the smallness of $\p\phi$ should be needed (in particular, when $p$ is odd). However, by the following expression of
solution $v$ to the 3D linear wave equation
$\square v=0$ with $(v, \p_tv)(0,x)=(v_0, v_1)(x)$,
 \begin{equation}\label{YHC-1}
  \begin{split}
   v(t,x)&=\f{1}{4\pi}\f{\p}{\p t}\bigg(t\int_{\Bbb S^2}v_0(x+t\omega)\d\omega\bigg)+
   \f{t}{4\pi}\int_{\Bbb S^2}v_1(x+t\omega)\d\omega\\
         &=\f{1}{4\pi}\int_{\Bbb S^2}v_0(x+t\omega)\d\omega+\f{t^2}{4\pi}\int_{\Bbb S^2}\nabla_xv_0(x+t\omega)\cdot\omega \d\omega
         +\f{t}{4\pi}\int_{\Bbb S^2}v_1(x+t\omega)\d\omega,\\
  \end{split}
 \end{equation}
when $\nabla^2v_0$ or $\nabla v_1$ are large, then $\p v$ is generally large. Therefore, only in terms of the
short pulse initial data \eqref{initial data}, it is not enough to keep the smallness of solution $\p\phi$ to
\eqref{main equation} when $\delta>0$ is small. This means that such an outgoing constraint condition  \eqref{condition on data} is required
in order to guarantee the well-posedness of \eqref{main equation} with \eqref{initial data}.
On the other hand, the condition \eqref{condition on data} actually implies the better smallness of  $\phi$ along the
outgoing directional derivative
$\p_t+\p_r$ up to  orders 2 on the initial time $t=1$. It is pointed out that when \eqref{initial data} is given,
$\p^{\be}\phi(1,x)=O(\de^{2-\ve0-|\be|})$ and further $(\p_t+\p_r)^k\phi(1,x)=O(\de^{2-\ve0-k})$ hold from \eqref{main equation},
which means the over-determination of \eqref{condition on data}
for arbitrary choice of $(\phi_0,\phi_1)$ in \eqref{initial data}. Thus the choice of $(\phi_0,\phi_1)$
will be somewhat restricted (see Appendix below).

The main result in the paper is
\begin{theorem}[]\label{main theorem}
Under the condition \eqref{condition on data}, when $p>p_c=\f{1}{1-\ve0}$,
for small $\de>0$, the equation \eqref{main equation} with \eqref{initial data} admits a
global smooth solution
 \begin{equation*}
  \phi\in C^{\infty}\left([1,+\infty)\times\mathbb{R}^3\right)\ \text{with}\ |\p\phi|\leq C\de^{1-\ve0}t^{-1},
 \end{equation*}
where $C>0$ is a uniform constant independent of $\de$ and $\ve0$.
\end{theorem}

\begin{remark}
In our subsequent companion paper \cite{Lu1}, when $1\le p\le p_c$,  we will show that the
smooth solution $\phi$ of \eqref{main equation} with \eqref{initial data} will blow up
and further form the outgoing shock before the  time $t=2$ under the following assumption of $(\phi_0,\phi_1)$:
there exists a point $(s_0,\om_0)\in (-1,0)\times\Bbb S^2$ such that
$$\phi_1^{p-1}(s_0, \om_0)\p_s\phi_1(s_0, \om_0)>\frac{2}{p}\quad\text{for $1\le p<p_c$}$$
or
$$
\phi_1^{p-1}(s_0, \om_0)\p_s\phi_1(s_0, \om_0)>\frac{(p-1)2^p}{p(2^{p-1}-1)}\quad\text{for $p=p_c$}.
$$
\end{remark}

\begin{remark}
Due to the special forms  of \eqref{initial data} and equation \eqref{main equation}, and
the restrictions of $k=1,2$, then \eqref{condition on data} is completely  equivalent to
 \begin{equation}\label{YHC-00}
  (\p_t+c\p_r)^k\Om^{\al}\p^{\be}\phi(1,x)=O(\de^{2-\ve0-|\be|}),\ 0\leq k\leq 2,
 \end{equation}
where $\Om\in\{x^i\p_j-x^j\p_i:1\leq i<j\leq 3\}$ stands for the derivatives on $\Bbb S^2$ and $c=\big(1+(\p_t\phi)^p\big)^{-\f12}$ is the wave speed.
\end{remark}

\begin{remark}
In \cite{MPY}, the authors study the Cauchy problem of the following 3-D semilinear wave equation systems with the short pulse initial data
\begin{equation}\label{MPYc}
  \left\{
   \begin{aligned}
&\Box\vp^I=\dse\sum_{0\le\alpha,\beta\le 3; 1\le J,K\le N}A_{JK}^{\alpha\beta,I}\p_{\alpha}\vp^J
\p_{\beta}\vp^K, \quad I=1, ..., N,\\
&(\vp^I, \p_t\vp^I)(1,x)=(\de^{\f12}\vp_0^I(\f{r-1}{\de},\om), \de^{-\f12}\vp_1^I(\f{r-1}{\de},\om)),
 \end{aligned}
  \right.
 \end{equation}
where $A_{JK}^{\alpha\beta,I}$ are constants, $\p_0=\p_t$, $(\vp_0^I(s,\om),
\vp_1^I(s,\om))\in C_0^{\infty}\left((-1,0)\times\mathbb{S}^2\right)$,
the quadratic nonlinear forms satisfy the null conditions
$\dse\sum_{0\le\alpha,\beta\le 3}A_{JK}^{\alpha\beta,I}\xi_{\alpha}
\xi_{\beta}\equiv0$ for $\xi_0=-1$, any $(\xi_1,\xi_2,\xi_3)\in\mathbb S^2$
and $1\le I, J, K\le N$.
Moreover,  it is assumed that
\begin{equation}\label{Y-3}
	|(\p_t+\p_r)^k\Omega^{\al}\p^q\vp^I(1,x)|\le C_{k\al q}^I \delta^{1/2-|q|},\quad k\leq N_0,
	\end{equation}
where $C_{k\al q}^I$ are constants, and
$N_0\ge 40$ is a sufficiently large integer. Then the global existence of smooth solution $\vp=(\vp^1, ..., \vp^N)$
to \eqref{MPYc} is  established in \cite{MPY}. Due to the largeness
of the integer $N_0$ in \eqref{Y-3}, it follows from the proof of \cite{MPY} that \eqref{MPYc} essentially becomes
the small value solution problem inside the cone $\{r\le t-\delta\}$.
Recently,
by the similar idea of \cite{MPY}, the authors in \cite{Wang} proved the global existence of the relativistic membrane equation
\begin{equation}\label{HCC-1}
\p_t(\displaystyle\frac{\p_t\vp}{\sqrt{1-(\p_t\vp)^{2}+|\nabla\vp|^2}})
-\displaystyle\sum_{i=1}^n\partial_i(\frac{\p_i \vp}{\sqrt{1-(\p_t\vp)^2+|\nabla\vp|^2}})
=0\quad (n=2,3)
\end{equation}
with the short pulse initial data $(\de^{\f32}\vp_0(\f{r-1}{\de},\om), \de^{\f12}\vp_1(\f{r-1}{\de},\om))$
and the constraint condition $(\p_t+\p_r)^k\Omega^{\al}$ $\p^q\vp(1,x)=O(\delta^{3/2-|q|})$ for $k\le N_0$
($N_0$ is large enough).
By the largeness of  $N_0$  in \eqref{Y-3} and the methods
for treating the semilinear wave equation (one can choose $u=t-r$ as the optical function),
the authors in  \cite{MPY} and \cite{Wang} can show the global
existence of smooth solutions by the energy method.
In this paper, we need to treat the large value of $\phi$ in the whole time-space $\Bbb R_+\times\Bbb R^3$
because of $k=1,2$ in \eqref{condition on data},
and the methods in \cite{MPY} or \cite{Wang} seem  not to be available for us.
\end{remark}

\begin{remark}
For $k=3$ or larger number $k$, \eqref{condition on data} with the power
$\delta^{2-\ve0}$ is seriously over-determined
and difficult to be realized for the choice of $(\phi_0,\phi_1)$.
In Appendix, we will chose $(\phi_0,\phi_1)$ such that
\eqref{condition on data} holds for $k=1,2$.
\end{remark}

We now comment on the proof of Theorem \ref{main theorem}.
For the general 4D quasilinear wave equation
$\displaystyle\sum_{\al,\beta=0}^4g^{\al\beta}(\phi,\p\phi)\p_{\al\beta}^2\phi=0$ and the 2D
quasilinear wave  equation
$\displaystyle\sum_{\al,\beta=0}^2g^{\al\beta}(\p\phi)\p_{\al\beta}^2\phi=0$ with
the short pulse initial data $(\delta^{2-\varepsilon_0}\phi_0(\f{r-1}{\delta},\omega),
\delta^{1-\varepsilon_0}\phi_1(\f{r-1}{\delta},\omega))$, under the corresponding null conditions
the authors in \cite{Ding4}-\cite{Ding3}
have established the global existence of the smooth solutions $\phi$ for  the suitable scope of $\ve0$ with $0<\ve0<\varepsilon^*$ and  $\varepsilon^*<1$ under the assumptions like \eqref{condition on data}. As in \cite{Ding4}-\cite{Ding3},
strongly motivated by the geometric methods of D. Christodoulou \cite{Ch07},
we will construct the solution $\phi$ of \eqref{main equation} near the outermost outgoing conic surface $C_0=\{(t, x):t\geq 1+2\delta,t=r\}$.
Introduce the inverse foliation density $\mu=-\displaystyle\frac{1}{(1+(\p_t\phi)^p)\p_tu}$,
where the  optical function $u$ satisfies
$-\big(1+(\p_t\phi)^p\big)(\p_tu)^2+\displaystyle\sum_{i=1}^3(\p_iu)^2=0$ with the initial data $u(1+2\de,x)=1+2\de-r$.
Under the bootstrap assumptions on $\p\phi$ with the suitable time-decay rates and the precise
smallness powers of $\de$, then $\mu$ satisfies the equation  $L\mu=O(\de^{(1-\ve0)p-1}t^{-p})$,
where $L$ is a vectorfield approximating $\p_t+\p_r$.
By $\mu(1+2\de,x)\sim 1$ and integration along the integral curves of $L,\ \mu\sim 1$ is derived for small $\de>0$.
The positivity of $\mu$ tells us that the outgoing characteristic conic surfaces never intersect as long as the smooth solution $\phi$
with suitable time-decay rate exists. From this, the global weighted energy estimates of $\phi$ near $C_0$ can be derived
and further the  bootstrap assumptions are closed.
In addition, we can establish that the outgoing characteristic conic surface of equation \eqref{main equation}
starting from the domain $\{t=1+2\de,1-2\de\leq r\leq 1+\de\}$
are almost straight and further contain the surface $\Ct_{2\de}=\{(t,x):t\geq 1+2\de,t-r=2\de\}$.
On the other hand, the crucial estimate  $|\p^\al\phi|\lesssim\de^{2-\ve0}t^{-1}$
on $\Ct_{2\de}$ is obtained, which contains the higher order smallness factor $\de^{2-\ve0}$ rather than $\de^{2-\ve0-|\al|}$.
Based on such ``good" smallness of $\phi$ on $\Ct_{2\de}$, we can solve the
global Goursat problem of \eqref{main equation} in the conic domain $B_{2\de}=\{(t,x):t\geq 1+2\de,t-r\geq 2\de\}$.
Therefore, the proof of Theorem \ref{main theorem} is completed. It is emphasized that
compared with  \cite{Ding4}-\cite{Ding3}, except the differences of space dimensions and time-decay rate
of solution $\phi$, one of our main ingredients
is to obtain the optimal power estimates of $\delta$ for all the related quantities
so that the critical exponent $p_c$ of equation \eqref{main equation} can be determined.

Our paper is organized as follows.
In Section \ref{2},
at first, we state the local existence of the solution $\phi$
to equation \eqref{main equation} with \eqref{initial data}-\eqref{condition on data}  for $1\leq t\leq 1+2\delta$
and some key estimates of $\phi(1+2\delta, x)$.
In Subsection \ref{Section 2},
we give the preliminary knowledge on the Lorentzian geometry, especially, recall the
definitions of optical function, inverse foliation density $\mu$, deformation tensor, null frame and some norms of
smooth functions. In addition, the equation of $\mu$ is derived, and some basic calculations are given for the covariant
derivatives of the null frame and for the deformation tensors.
In Section \ref{Section 3}, the basic bootstrap assumptions near $C_0$
are given, meanwhile, we also give some estimates on several quantities which will be extensively used in subsequent sections.
In Section \ref{Section 4}, the global energy estimates for the linearized covariant wave equation $\mu\Box_g\Psi=\Phi$
are derived and some higher order weighted energies and fluxes are defined.
In Section \ref{Section 5}, under the bootstrap assumptions, the higher order $L^{\infty}$ and
$L^2$ estimates of $\p\phi$ near $C_0$ are established.
In Section \ref{Section 6}, the top order $L^2$ estimates for the derivatives of $\chi$ and $\mu$ are established, where $\chi$ is
the second fundamental form of the related  metric $g$.
In Section \ref{Section 7}, we first derive the commuted wave equation and then treat the estimates for the
resulting error terms.
In Section \ref{Section 8}, based on all the estimates in the previous sections, the bootstrap arguments are closed
and further the global existence of solution $\phi$ to equation \eqref{main equation} near $C_0$ is established.
On the other hand, the global existence of solution $\phi$ in $B_{2\de}$ is obtained
and then Theorem \ref{main theorem} is shown.
In Appendix, we prove the existence of $(\phi_0,\phi_1)$ such that the condition \eqref{condition on data}
is satisfied.
\vskip 0.2 true cm
Throughout the whole paper, without special mentions, the following notations are used:
\begin{itemize}
  \item Greek letters $\{\al,\be,\ga,\cdots\}$ corresponding to the spacetime coordinates are chosen in \{0,1,2,3\}; Latin letters $\{i,j,k,\cdots\}$ corresponding to the spatial coordinates are chosen in \{1,2,3\}; Capital letters $\{A,B,C,\cdots\}$ corresponding to the sphere coordinates are chosen in \{1,2\}.
  \item We use the Einstein summation convention to sum over the repeated upper and lower indices.
  \item The convention $f\lesssim h$ means that there exists a generic positive constant $C$ independent of the parameter $\de>0$ and the variables $(t,x)$ such that $f\leq Ch$.
  \item If $\xi$ is a $(0,2)$-type spacetime tensor, $\La$ is a one-form, $U$ and $V$ are vectorfields, then
  the contraction of $\xi$ with respect to $U$ and $V$ is defined as $\xi_{UV}=\xi_{\al\be}U^{\al}V^{\be}$,
  and the contraction of $\La$ with respect to $U$ is defined as $\La_{U}=\La_{\al}U^{\al}$.
  \item The restriction of quantity $\zeta$ (including the metric $g$, $(m,n)$-type
  spacetime tensor field) on the sphere is represented by $\slashed{\zeta}$. But if $\zeta$ is already defined on the sphere, it is still represented by $\zeta$.
  \item $\mathcal L_V\xi$ stands for the Lie derivative of $\xi$ with respect to vector $V$,
and $\slashed{\mathcal L}_V\xi$ is the restriction of  $\mathcal L_V\xi$
on the sphere.
\end{itemize}
Finally, such notations are introduced:
\begin{align*}
&C_0=\{(t,x):t\geq 1+2\delta,t=r\},\\
&B_{2\delta}=\{(t,x):t\geq 1+\delta,t-r\geq 2\delta\},\\
&\tilde C_{2\delta}=\{(t,x):t\geq 1+2\delta,t-r=2\delta\},\\
&c=\big(1+(\p_t\phi)^p\big)^{-\frac{1}{2}},\\
&t_0=1+2\de,\\
&\t L=\p_t+\p_r,\\
&\t {\underline L}=\p_t-\p_r,\\
&\Om_i={\epsilon_{ij}}^kx^j\p_k,\\
&\Om\in \{\Om_{i}: 1\le i\le 3\},\\
&S=t\p_t+r\p_r=\f{t-r}{2}\t {\underline L}+\f{t+r}{2}\t L,\\
&H_i=t\p_i+x^i\p_t=\om^i\big(\f{r-t}{2}\t{\underline L}+\f{t+r}{2}\t L\big)+\f{t\om^j}{r}{\epsilon_{ji}}^k\Om_k,\\
&\Sigma_t=\{(t',x):t'=t,x\in\Bbb R^3\},
\end{align*}
where ${\epsilon_{ij}}^k=\epsilon_{ijk}=-1$ when $ijk$ is $123$'s odd permutation, and ${\epsilon_{ij}}^k=\epsilon_{ijk}=1$
when $ijk$ is $123$'s even permutation.

\vskip 0.2 true cm

\section{\bf Some preliminaries}\label{2}
\vskip 0.2 true cm

\subsection{Local existence}\label{Section 1}
In this subsection,  for the equation \eqref{main equation} with \eqref{initial data}-\eqref{condition on data},
we list the local existence of the smooth solution $\phi$ and some crucial properties
for $1\leq t\leq t_0$. Since the proof is rather analogous to that of Theorem 3.1 in \cite{Ding4}
(for the general 4D quasilinear wave equation with the first null condition) or
that of Theorem 2.1 in  \cite{Ding3} (for the 2D quasilinear wave equations with the first and second null conditions),
we omit the details here.

\begin{theorem}\label{Theorem local existence}
	For sufficiently small $\de>0$,
	the equation \eqref{main equation} with \eqref{initial data}-\eqref{condition on data} admits a local smooth solution $\phi\in C^{\infty}\big([1,t_0]\times\mathbb{R}^3\big)$. Moreover, for $q\in\mathbb{N}_0^4$, $\kappa\in\mathbb{N}_0^3$, $k\in\mathbb{N}_0$ and $l\in\mathbb{N}_0$, it holds that
	\begin{itemize}
		\item[(1)] \begin{align*}
		|\t L^k\p^{q}\Om^{\kappa}\phi(t_0,x)|&\lesssim\de^{2-|q|-\ve0},\ r\in[1-2\de,1+2\de],\\
		|\t\Lb^l\p^{q}\Om^{\kappa}\phi(t_0,x)|&\lesssim\de^{2-|q|-\ve0},\ r\in[1-3\de,1+\de].
		\end{align*}
		\item[(2)] \begin{equation*}
		|\p^{q}\Om^{\kappa}\phi(t_0,x)|\lesssim
		\left\{
		\begin{aligned}
		\de^{2-\ve0}&,\ as\ |q|\leq 2,\\
		\de^{4-|q|-\ve0}&,\ as\ |q|> 2,
		\end{aligned}
		\right.
		\ r\in[1-3\de,1+\de].
		\end{equation*}
		\item[(3)] \begin{equation*}
		|\t L^k\t \Lb^l\Om^{\kappa}\phi(t_0,x)|\lesssim\de^{2-\ve0},\ r\in[1-2\de,1+\de].
		\end{equation*}
	\end{itemize}
\end{theorem}

\subsection{The Lorentzian geometry and some related definitions}\label{Section 2}

In this subsection, we give some preliminaries on the related Lorentzian geometry and definitions, which will be
utilized as the basic tools later on.

\vskip 0.2 true cm

{\bf 2.2.1. Metric and Christoffel symbols}
\vskip 0.1 true cm

By the form of \eqref{main equation}, it is natural to introduce the following inverse spacetime metric
 \begin{equation}\label{g inverse}
  g^{-1}=(g^{\al\be})=\mathrm{diag}(-\frac{1}{c^2},1,1,1)
 \end{equation}
and the corresponding spacetime metric
 \begin{equation}\label{g}
  g=(g_{\al\be})=\mathrm{diag}(-c^2,1,1,1).
 \end{equation}
In this case, \eqref{main equation} can be rewritten as
 \begin{equation}\label{wave equation}
  -\frac{1}{c^2}\p_t^2\phi+\De\phi=0
 \end{equation}
or
 \begin{equation}\label{quasilinear wave equation}
  g^{\al\be}\p_{\al\be}^2\phi=0.
 \end{equation}

In the Cartesian coordinates, the Christoffel symbols of $g$ are defined by
 \begin{equation}\label{Definition 1st Christoffel symbol}
  \Ga_{\al\be\ga}=\frac{1}{2}(\p_{\al}g_{\be\ga}+\p_{\ga}g_{\al\be}-\p_{\be}g_{\al\ga})
 \end{equation}
and
 \begin{equation}\label{Definition 2nd Christoffel symbol}
  \Ga_{\al\be}^{\ga}=g^{\ga\la}\Ga_{\al\la\be}.
 \end{equation}
Meanwhile, set
 \begin{equation}\label{Definition 3rd Christoffel symbol}
  \Ga^{\ga}=g^{\al\be}\Ga_{\al\be}^{\ga}.
 \end{equation}

\begin{definition}\label{Definition G function}
Define $G$ function as
 \begin{equation*}
  G_{\al\be}^{\ga}=\frac{\p g_{\al\be}}{\p\vp_{\ga}},
 \end{equation*}
where and below $\vp_{\ga}=\p_{\ga}\phi$.
\end{definition}
Due to $g_{00}=-c^2=-(1+\vp_0^p)^{-1}$, then the only non-vanishing component of the $G$ function is
 \begin{equation*}
  G_{00}^0=pc^4\vp_0^{p-1}.
 \end{equation*}

\vskip 0.2 true cm

{\bf 2.2.2. Optical function, inverse foliation density and null frames}
\vskip 0.2 true cm

As in \cite{Ch07}, one can introduce the following \textit{optical function}.
\begin{definition}\label{Definition optical function}
A $C^1$ function $u(t,x)$ is called the optical function of problem \eqref{quasilinear wave equation} if $u(t,x)$ satisfies the eikonal equation
 \begin{equation}\label{eikonal equation}
  g^{\al\be}\p_{\al}u\p_{\be}u=0.
 \end{equation}
\end{definition}
In the paper, we will choose the initial data $u(t_0,x)=u(1+2\de,x)=1+2\de-r$ and pose the condition $\p_tu>0$ for \eqref{eikonal equation}. According to the definition of optical function, define the \textit{inverse foliation density} $\mu$ as in \cite{Ch07},
 \begin{equation}\label{Definition inverse foliation density}
  \mu=-\frac{1}{g^{\al\be}\p_{\al}u\p_{\be}t}\big(=\frac{1}{c^{-2}\p_tu}\big).
 \end{equation}
By \eqref{eikonal equation} and \eqref{g inverse}, then $-c^{-2}(\p_tu)^2+\displaystyle\sum_{i=1}^3(\p_iu)^2=0$
holds. Since $\p_iu=-\p_ir=-\frac{x^i}{r}$ on $t=t_0$, one then has $\p_tu|_{t_0}=c|_{t_0}$ and
 \begin{equation}\label{mu t0}
  \mu|_{t_0}=c|_{t_0}.
 \end{equation}
Note that the authors in \cite{Ch07}, \cite{MY17} and \cite{Sp16} apply the inverse foliation density to prove the formation of shocks when $\mu\rightarrow 0+$ holds with the development of time $t$. In the paper, on the contrary, we will show $\mu\geq C>0$
as long as the smooth solution $\phi$ of equation \eqref{main equation} exists as in \cite{Ding4}-\cite{Ding5}.

Note that
 \begin{equation}\label{Lr}
  \Lr=-\mathrm{grad}\ u=-g^{\al\be}\p_{\al}u\p_{\be}
 \end{equation}
is a tangent vector field for the outgoing light cone $\{u=C\}$. In addition, it is easy to know that $\Lr$ is geodesic and $\Lr t=\mu^{-1}$. Thus, we rescale $\Lr$ as
 \begin{equation}\label{L}
  L=\mu\Lr.
 \end{equation}
To obtain the vector field of the incoming light cone, just as in \cite{MY17}, let
 \begin{equation}\label{Tr}
  \Tr=c^{-1}(\p_t-L).
 \end{equation}
Then by the definition of null frames, we set
 \begin{equation}\label{T}
  T=c^{-1}\mu\Tr,
 \end{equation}
 \begin{equation}\label{Lb}
  \Lb=c^{-2}\mu L+2T,
 \end{equation}
where $L$ and $\Lb$ are two vector fields in the null frame. About the other vector fields $\{X_1,X_2\}$ in the null frame,
we take use of the vector field $L$ to construct them. To this end, one extends the local coordinates $\{\th^1,\th^2\}$
on $\mathbb{S}^2$ as follows
 \begin{equation*}
  L\vt^{A}=0,\ \vt^{A}|_{t=1}=\th^{A},
 \end{equation*}
here and below $A=1,2$. Subsequently, let
 \begin{equation}\label{X}
  X_1=\frac{\p}{\p\vt^1},\ X_2=\frac{\p}{\p\vt^2}.
 \end{equation}
A direct computation yields
\begin{lemma}\label{null frame}
$\{L,\Lb,X_1,X_2\}$ constitutes a null frame with respect to the metric $g$ in \eqref{g}, and admits the following identities
\begin{equation*}
 g(L,L)=g(\Lb,\Lb)=g(L,X_A)=g(\Lb,X_A)=0,\ g(L,\Lb)=-2\mu.
\end{equation*}
In addition,
 \begin{equation*}
  g(L,T)=-\mu,\ g(T,T)=c^{-2}\mu^2.
 \end{equation*}
And
 \begin{equation*}
  Lt=1,\ Lu=0,\ Tt=0,\ Tu=1,\ \Lb t=c^{-2}\mu,\ \Lb u=2.
 \end{equation*}
\end{lemma}

\vskip 0.2 true cm

{\bf 2.2.3. Domains, coordinates and norms}

\vskip 0.2 true cm

As in \cite{Sp16}, one can perform the change of coordinates:
$(t, x^1, x^2, x^3)\longrightarrow (\mathfrak t, u, \vartheta^1, \vartheta^2)$ near $C_0$ with
\begin{equation}\label{H0-7}
\left\{
\begin{aligned}
&\mathfrak{t}=t,\\
&u=u(t,x),\\
&\vartheta^1=\vartheta^1(t,x),\\
&\vartheta^2=\vartheta^2(t,x).
\end{aligned}
\right.
\end{equation}
For notational convenience, we introduce the following subsets
\begin{definition}\label{Definition domain}
 \begin{align*}
  \Sigma_\mathfrak t^u&=\{(\mathfrak t',u',\vartheta):\ \mathfrak t'=\mathfrak t, 0\leq u'\leq u\}, u\in [0,4\delta],\\
  C_u&=\{(\mathfrak t',u',\vartheta):\ \mathfrak t'\geq t_0, u'=u\},\\
  C_u^\mathfrak t&=\{(\mathfrak t',u',\vartheta):\ t_0\leq \mathfrak t'\leq\mathfrak t, u'=u\},\\
  S_{\mathfrak t,u}&=\Sigma_\mathfrak t\cap C_u,\\
  D^{\mathfrak t,u}&=\{(\mathfrak t',u',\vartheta):\ t_0\leq \mathfrak t'<\mathfrak t, 0\leq u'\leq u\}.
 \end{align*}
\end{definition}

Note that $\vartheta=(\vartheta^1,\vartheta^2)$ are the coordinates on sphere $S_{\mathfrak t,u}$. Then under the new coordinate system $(\mathfrak t,u,\vartheta^1,\vartheta^2)$, one has $L=\frac{\partial}{\partial \mathfrak t}, T=\frac{\partial}{\partial u}-\Xi$ with $\Xi=\Xi^AX_A$.
In addition, it follows from direct computation that
\begin{lemma}\label{Lemma jacobian}
In domain $D^{\mathfrak t,u}$, the Jacobian determinant of map $(\mathfrak t,u,\vartheta^1,\vartheta^2)\rightarrow (x^0,x^1,x^2,x^3)$ is
 \begin{equation*}
  \det\frac{\p(x^0,x^1,x^2,x^3)}{\p(\mathfrak t,u,\vartheta^1,\vartheta^2)}=c^{-1}\mu\sqrt{\det\gs}.
 \end{equation*}
\end{lemma}

\begin{remark}
From Lemma \ref{Lemma jacobian}, it is easy to know that if the metric $\gs$ are locally regular, that is, $\det\gs>0$,
the transformation of coordinates between $(\mathfrak t,u,\vartheta^1,\vartheta^2)$ and $(x^0,x^1,x^2,x^3)$
then makes sense as long as $\mu>0$.
\end{remark}

For the domains with $\mu>0$, we now give some definitions of related integrations and norms, which will be utilized repeatedly in subsequent sections.
\begin{definition}\label{Definition norm}For any continuous function $f$, set
 \begin{align*}
  &\int_{S_{\mathfrak t,u}}f=\int_{S_{\mathfrak t,u}}f(\mathfrak t,u,\vartheta)\sqrt{\det\gs(\mathfrak t,u,\vartheta)}\d\vartheta,\qquad\|f\|_{L^2(S_{\mathfrak t,u})}^2=\int_{S_{\mathfrak t,u}}|f|^2,\\
  &\int_{C_u^\mathfrak t}f=\int_{t_0}^\mathfrak t\int_{S_{\tau,u}}f(\tau,u,\vartheta)\sqrt{\det\gs(\tau,u,\vartheta)}\d\vartheta \d\tau,\qquad\|f\|_{L^2(C_u^\mathfrak t)}^2=\int_{C_u^\mathfrak t}|f|^2,\\
  &\int_{\Sigma_\mathfrak t^u}f=\int_{0}^u\int_{S_{\mathfrak t,u'}}f(\mathfrak t,u',\vartheta)\sqrt{\det\gs(\mathfrak t,u',\vartheta)}\d\vartheta \d u',\qquad\|f\|_{L^2(\Sigma_\mathfrak t^u)}^2=\int_{\Sigma_\mathfrak t^u}|f|^2,\\
  &\int_{D^{\mathfrak t,u}}f=\int_{t_0}^\mathfrak t\int_{0}^u\int_{S_{\tau,u'}}f(\tau,u',\vartheta)\sqrt{\det\gs(\tau,u',\vartheta)}\d\vartheta \d u'\d\tau,\qquad\|f\|_{L^2(D^{\mathfrak t,u})}^2=\int_{D^{\mathfrak t,u}}|f|^2.
 \end{align*}
\end{definition}

\vskip 0.2 true cm

{\bf 2.2.4. Connection, the second fundamental form and torsion form}

\vskip 0.2 true cm

Let $\D$ be the Levi-Civita connection of $g$. Without causing confusion, we use $\nas$ to denote the Levi-Civita connection of $\gs$.

Under the null frame $\{L,\Lb,X_1,X_2\}$, define \textit{the second fundamental forms} $\chi$ and $\sigma$ as
 \begin{equation}\label{Definition second fundamental form}
  \chi_{AB}=g(\D_AL,X_B),\ \si_{AB}=g(\D_A\Tr,X_B).
 \end{equation}
And \textit{the torsion one forms} $\ze$ and $\eta$ are defined by
 \begin{equation}\label{Definition torsion one form}
  \ze_A=g(\D_AL,T),\ \eta_A=-g(\D_AT,L).
 \end{equation}
Direct computation yields
 \begin{align}
 &\si_{AB}=-c^{-1}\chi_{AB},\label{sigma}\\
 &\ze_A=-c^{-1}\mu\ds_Ac,\label{zeta}\\
 &\eta_A=-c^{-1}\mu\ds_Ac+\ds_A\mu.\label{eta}
 \end{align}

\begin{lemma}\label{Lemma connection coefficients}
For the connection coefficients of the related frames, it holds that
\begin{equation*}
 \begin{split}
  &\mathscr{D}_{L}L=\mu^{-1}L\mu L,\ \mathscr{D}_{T}L= \eta^{A}X_A-c^{-1}L(c^{-1}\mu)L,\ \mathscr{D}_{A}L=-\mu^{-1}\ze_AL+\chi_A^BX_B,\\
  &\mathscr{D}_{L}T=- \ze^{A}X_A-c^{-1}L(c^{-1}\mu)L,\ \mathscr{D}_AT=\mu^{-1}\eta_AT-c^{-2}\mu\chi_A^BX_B,\\
  &\mathscr{D}_TT=c^{-3}\mu[Tc+L(c^{-1}\mu)]L+\{c^{-1}[Tc+L(c^{-1}\mu)]+T\ln(c^{-1}\mu)\}T-c^{-1}\mu\ds ^A(c^{-1}\mu)X_A,\\
  &\mathscr{D}_LX_A=\mathscr{D}_AL,\ \mathscr{D}_TX_A=\mathscr{D}_AT,\ \mathscr{D}_AX_B=\nas_AX_B+\mu^{-1}\chi_{AB}T,\ \mathscr{D}_A\Lb=\mu^{-1}\eta_A\Lb-c^{-2}\mu\chi_A^BX_B,\\
  &\mathscr{D}_{\Lb}L=-L(c^{-2}\mu)L+2\eta^AX_A,\ \mathscr{D}_L\Lb=-2\ze^AX_A,\ \mathscr{D}_{\Lb}\Lb=[\mu^{-1}\Lb\mu+L(c^{-2}\mu)]\Lb-2\mu\ds^A(c^{-2}\mu)X_A,
 \end{split}
\end{equation*}
where $\ze^{A}=\slashed g^{AB}\ze_B$ and $\eta^{A}=\slashed g^{AB}\eta_B$.
\end{lemma}

\vskip 0.2 true cm

{\bf 2.2.5. Error vectors and rotation vectors}

\vskip 0.2 true cm
On the initial hypersurface $\Sigma_{t_0}^{4\delta}$, one has that $ \Tr^i=-\f{x^i}{r}$,
$L^0=1$, $L^i=\f{x^i}{r}+O(\delta^{(1-\varepsilon_0)p})$ and
$\chi_{AB}=\f1r\slashed g_{AB}$.
Note that on $\Sigma_{t_0}$, $r$ is just $t_0-u$. For $\mathfrak t\geq t_0$, we
define the ``{\it{error vectors}}" with the components being
\begin{definition}\label{Definition error vectors and trace-free vectors}
 \begin{align*}
  \check{L}^0&=0,\\
  \check{L}^i&=L^i-\frac{x^i}{\rho},\\
  \check{T}^i&=\Tr^i+\frac{x^i}{\rho},\\
  \check{\chi}_{AB}&=\chi_{AB}-\frac{1}{\rho}\gs_{AB},
 \end{align*}
here and below $\rho=\mathfrak t-u$.
\end{definition}

\begin{lemma}\label{Lemma relation of error vectors and trace-free vectors}
The error vectors satisfy
 \begin{equation*}
  \begin{split}
   \check{L}^i&=-c\check{T}^i+(c-1)\rho^{-1}x^i,\\
   \trgs\check{\chi}&=\trgs\chi-2\rho^{-1},\\
   |\check{\chi}|^2&=|\chi|^2-2\rho^{-1}\trgs\chi+2\rho^{-2}.
  \end{split}
 \end{equation*}
\end{lemma}

Let
 \begin{equation}\label{vi}
  v_i=g(\Om_i,\Tr)={\ep_{ijk}}x^j\check{T}^k.
 \end{equation}
Then
 \begin{equation}\label{Ri}
  R_i=\Om_i-v_i\Tr
 \end{equation}
are the rotation vector fields of $S_{\mathfrak t,u}$.

\begin{remark}\label{Remark componets of gs}
The components of $\gs$ satisfy
 \begin{equation*}
  \begin{split}
   &\gs^{\mu\nu}=g^{\mu\nu}+\frac{1}{2}\mu^{-1}(L^{\mu}\Lb^{\nu}+\Lb^{\mu}L^{\nu}),\\
   &\gs^{\mu\nu}=\gs^{AB}X_A^{\mu}X_B^{\nu}\\
   \end{split}
 \end{equation*}
 and for any smooth function $\Psi$,
$$\gs^{\mu\nu}\p_{\mu}\Psi\p_{\nu}\Psi=|\ds\Psi|^2.$$
\end{remark}

\vskip 0.2 true cm

{\bf 2.2.6. Curvature tensor, energy-momentum tensor and deformation tensor}

\vskip 0.2 true cm
\textit{The Riemann curvature tensor} $\mathscr{R}$ of $g$ can be defined as follows
 \begin{equation}\label{Definition Riemann curvature tensor}
  \mathscr{R}_{WXYZ}=-g(\mathscr{D}_W{\mathscr{D}_X}Y-\mathscr{D}_X{\mathscr{D}_W}Y-{\mathscr{D}_{[W,X]}}Y,Z).
 \end{equation}
Since
 \begin{equation*}
  \mathscr{R}_{\mu\nu\alpha\beta}=\partial_{\nu}\Gamma_{\mu\beta\alpha}-\partial_{\mu}\Gamma_{\nu\beta\alpha}+g^{\kappa\lambda}(\Gamma_{\nu\kappa\alpha}\Gamma_{\mu\lambda\beta}-\Gamma_{\mu\kappa\alpha}\Gamma_{\nu\lambda\beta}),
 \end{equation*}
by the definitions \eqref{Definition Riemann curvature tensor} and \eqref{Definition 1st Christoffel symbol},
the only non-vanishing components of $\mathscr{R}$ for the metric $g$ in \eqref{g} are
 \begin{equation}\label{R0i0j}
  \mathscr{R}_{i0j0}=\mathscr{R}_{0i0j}=c\partial_i\partial_jc,
 \end{equation}
and the only non-vanishing component  under the frame $\{L,T,X_1,X_2\}$ is
 \begin{equation}\label{RLALB}
  \begin{split}
   \mathscr{R}_{LALB}&=-\mu^{-1}c(Tc)\check{\chi}_{AB}-\mu^{-1}\rho^{-1}c(Tc)\slashed{g}_{AB}\\
                     &\quad-\frac{1}{2}pc^4\varphi_0^{p-1}\nas_{AB}^2\varphi_0-\frac{3}{2}pc^3\varphi_0^{p-1}\ds_Ac\ds_B\varphi_0-\frac{1}{2}p(p-1)c^4\varphi_0^{p-2}\ds_A\varphi_0\ds_B\varphi_0.
  \end{split}
 \end{equation}
For the sake of convenience, define
 \begin{equation}\label{RLALB small}
  \check{\mathscr{R}}_{LALB}=\mathscr{R}_{LALB}+\mu^{-1}c(Tc)\check{\chi}_{AB}+\mu^{-1}\rho^{-1}c (Tc)\slashed{g}_{AB}.
 \end{equation}
Here we point out that $\check{\mathscr{R}}_{LALB}$ will admit
the better time-decay rate and higher smallness orders of $\de$ than $\mathscr{R}_{LALB}$.

For any smooth function $\Psi$, one can denote its associate \textit{energy-momentum tensor} by
 \begin{equation}\label{Definition energy-momentum tensor}
  Q_{\alpha\beta}=Q_{\alpha\beta}[\Psi]=(\partial_{\alpha}\Psi)(\partial_{\beta}\Psi)-\frac{1}{2}g_{\alpha\beta}g^{\ka\lambda}(\partial_{\ka}\Psi)(\partial_{\lambda}\Psi).
 \end{equation}
The components of energy-momentum tensor in terms of the null frame can be computed as follows
 \begin{equation}\label{components of energy-momentum tensor}
  \begin{split}
   &Q_{LL}=(L\Psi)^2,\ Q_{\Lb\Lb}=(\Lb\Psi)^2,\ Q_{L\Lb}=\mu|\ds\Psi|^2,\\
   &Q_{LA}=L\Psi\ds_A\Psi,\ Q_{\Lb A}=\Lb\Psi\ds_A\Psi,\\
   &Q_{AB}=\ds_A\Psi\ds_B\Psi-\frac{1}{2}\gs_{AB}(|\ds\Psi|^2-\mu^{-1}L\Psi\Lb\Psi).
  \end{split}
 \end{equation}

For any vector field $V$,  denote its associate \textit{deformation tensor} by
 \begin{equation}\label{Definition deformation tensor}
  {}^{(V)}\pi_{\alpha\beta}=g(\mathscr{D}_{\alpha}V,\partial_{\beta})+g(\mathscr{D}_{\beta}V,\partial_{\alpha}).
 \end{equation}
Moreover, for any two vector fields $X$, $Y$, one has
 \begin{align*}
  {}^{(V)}\pi_{XY}&={}^{(V)}\pi_{\al\be}X^{\al}Y^{\be}=g(\D_XV,Y)+g(\D_YV,X).
 \end{align*}
The components of $^{(V)}\pi$ under the related frames and the metric $g$ in \eqref{g} can be obtained as follows

(1) When $V=L$,
\begin{equation}\label{L pi}
 \begin{split}
  &^{(L)}\pi_{LL}=0,\ ^{(L)}\pi_{LT}=-L\mu,\ ^{(L)}\pi_{TT}=2c^{-1}\mu L(c^{-1}\mu),\\
  &^{(L)}\slashed{\pi}_{LA}=0,\ ^{(L)}\slashed{\pi}_{TA}=c^2\ds_A(c^{-2}\mu),\ ^{(L)}\slashed{\pi}_{AB}=2\chi_{AB},\\
  &^{(L)}\pi_{L\Lb}=-2L\mu,\ ^{(L)}\pi_{\Lb\Lb}=4\mu L(c^{-2}\mu),\ ^{(L)}\slashed{\pi}_{\Lb A}=2c^2\ds_A(c^{-2}\mu).
 \end{split}
\end{equation}

(2) When $V=T$,
\begin{equation}\label{T pi}
 \begin{split}
  &^{(T)}\pi_{LL}=0,\ ^{(T)}\pi_{LT}=-T\mu,\ ^{(T)}\pi_{TT}=T(c^{-2}\mu^2),\\
  &^{(T)}\slashed{\pi}_{LA}=-c^2\ds_A(c^{-2}\mu),\ ^{(T)}\slashed{\pi}_{TA}=0,\ ^{(T)}\slashed{\pi}_{AB}=-2c^{-2}\mu\chi_{AB},\\
  &^{(T)}\pi_{L\Lb}=-2T\mu,\ ^{(T)}\pi_{\Lb\Lb}=4\mu T(c^{-2}\mu),\ ^{(T)}\slashed{\pi}_{\Lb A}=-\mu\ds_A(c^{-2}\mu).
 \end{split}
\end{equation}

(3) When $V=R_i$,
\begin{equation}\label{Ri pi}
 \begin{split}
  &^{(R_i)}\pi_{LL}=0,\ ^{(R_i)}\pi_{LT}=-R_i\mu,\ ^{(R_i)}\pi_{TT}=2c^{-1}\mu R_i(c^{-1}\mu),\\
  &^{(R_i)}\slashed{\pi}_{LA}=-\check{\chi}_{AB}R_i^B+\epsilon_{ijk}\check{L}^j\ds_Ax^k-v_i\ds_Ac,\\
  &^{(R_i)}\slashed{\pi}_{TA}=c^{-2}\mu\check{\chi}_{AB}R_i^B-c^{-2}(c-1)\mu\rho^{-1}\gs_{AB}R_i^B
  +c^{-1}\mu\epsilon_{ijk}\check{T}^j\ds_Ax^k+v_i\ds_A(c^{-1}\mu),\\
  &^{(R_i)}\slashed{\pi}_{AB}=2c^{-1}v_i\chi_{AB},\
  ^{(R_i)}\pi_{L\Lb}=-2R_i\mu,\ ^{(R_i)}\pi_{\Lb\Lb}=4\mu R_i(c^{-2}\mu).
 \end{split}
\end{equation}

\vskip 0.2 true cm

{\bf 2.2.7. Lie derivatives and commutators}

\vskip 0.2 true cm

According to (8.26) in \cite{Sp16}, one has
\begin{lemma}\label{Lemma commutator [Des,Z]}
For any symmetric  2-tensor $\xi$ on $S_{\mathfrak t,u}$,
 \begin{equation}\label{commutator [nas,Lies]}
  ([\nas_A,\Lies_Z]\xi)_{BC}=(\check{\nas}_A{}^{(Z)}\pis_B^D)\xi_{CD}+(\check{\nas}_A{}^{(Z)}\pis_C^D)\xi_{BD},
 \end{equation}
where
 \begin{equation*}
  \check{\nas}_A{}^{(Z)}\pis_{BC}=\frac{1}{2}(\nas_A{}^{(Z)}\pis_{BC}+\nas_B{}^{(Z)}\pis_{AC}-\nas_C{}^{(Z)}\pis_{AB}).
 \end{equation*}
For any vector field $Z\in\{L,\rho L,T,R_1,R_2,R_3\}$ and smooth function $f$,
 \begin{equation}\label{commutator of [nas^2,Lies]}
   \big([\nas^2,\Lies_Z]f\big)_{AB}=(\check{\nas}_A{}^{(Z)}\pis_B^C)\ds cf
 \end{equation}
and
 \begin{equation*}
  [\Des,Z]f={}^{(Z)}\pis^{AB}\nas_{AB}^2f+\check\nas_A{}^{(Z)}\pis^{AB}\ds_Bf,
 \end{equation*}
moreover, for any vector field $Z\in\{L,T,R_1,R_2,R_3\}$,
 \begin{equation*}
  \Lies_Z\slashed g_{AB}=^{(Z)}\slashed{\pi}_{AB}.
 \end{equation*}
\end{lemma}


\vskip 0.2 true cm

{\bf 2.2.8. Covariant wave equations and structure equations}

\vskip 0.2 true cm
We now look for the equation of $\varphi_{\lambda}$ $(\lambda=0,1,2,3)$ under the action of the covariant wave operator $\Box_g=g^{\alpha\beta}\mathscr{D}_{\alpha\beta}^2$ with the help of metric and Christoffel symbol.
 It follows from direct computation that
 \begin{equation}\label{Y-6}
  \begin{split}
   \Box_g\varphi_{\lambda}&=g^{\alpha\beta}\partial_{\alpha\beta}^2\varphi_{\lambda}
   -g^{\alpha\beta}\Gamma_{\alpha\beta}^{\gamma}\partial_{\gamma}\varphi_{\lambda}\\
   &=(-\frac{1}{c^2}\partial_t^2\varphi_{\lambda}+\Delta\varphi_{\lambda})-\Ga^{\ga}\p_{\ga}\vp_{\la}.
  \end{split}
 \end{equation}
In addition, taking the derivative on two sides of \eqref{main equation} with respect to the variable $x^{\lambda}$ derives
 \begin{equation}\label{Y-7}
  -\frac{1}{c^2}\partial_t^2\varphi_{\lambda}+\Delta\varphi_{\lambda}=p\varphi_0^{p-1}\partial_t\varphi_0\partial_t\varphi_{\lambda}.
 \end{equation}
By $\partial_t=L+c^2\mu^{-1}T$ and $\p_i=c^2\mu^{-2}T^iT+\ds^Ax^iX_A$, then it follows from \eqref{Y-6} and \eqref{Y-7} that
 \begin{equation}\label{covariant wave equation}
  \mu\Box_g\varphi_{\lambda}=F_{\la},
 \end{equation}
where
 \begin{equation}\label{F la}
  F_{\la}=\frac{1}{2}p\mu\vp_0^{p-1}L\vp_0L\vp_{\la}+\frac{1}{2}pc^2\vp_0^{p-1}L\vp_0T\vp_{\la}+\frac{1}{2}pc^2\vp_0^{p-1}T\vp_0L\vp_{\la}
  -\frac{1}{2}pc^2\mu\vp_0^{p-1}\ds\vp_0\cdot\ds\vp_{\la}.
 \end{equation}
On the other hand, due to
 \begin{align*}
  \mu\Box_g\vp_\lambda&=\mu g^{\al\be}\D_{\al\be}^2\vp_\lambda=-(L+\frac{1}{2}\trgs\chi)\Lb\vp_\lambda+2c^{-1}\mu\ds c\cdot\ds\vp_\lambda+\mu\Des\vp_\lambda+\frac{1}{2}c^{-2}\mu\trgs\chi L\vp_\lambda,
 \end{align*}
this yields
 \begin{equation}\label{transport equation of Lb vp0}
  (L+\frac{1}{2}\trgs\chi)\Lb\vp_\lambda=\mu\Des\vp_\lambda+H_\lambda-F_\lambda,
 \end{equation}
where
 \begin{equation}\label{H0}
  H_\lambda=\frac{1}{2}c^{-2}\mu\trgs\chi L\vp_\lambda+2c^{-1}\mu\ds c\cdot\ds\vp_\lambda.
 \end{equation}
For convenience, define
 \begin{equation}\label{H'0}
  \tilde H_\lambda=2c^{-1}\mu\ds c\cdot\ds\vp_\lambda.
 \end{equation}

Next, we give the structure equations of $\mu$, $\chi$, $L^i$ and $\check{L}^i$.
\begin{lemma}\label{3.6}
It holds that
 \begin{equation}\label{transport equation of mu}
  L\mu=-cTc+(c^{-1}Lc)\mu,
 \end{equation}
 \begin{equation}\label{transport equation of chi}
  L\chi_{AB}=c^{-1}(Lc)\chi_{AB}+\chi_A^C\chi_{BC}-\check{\R}_{LALB},
 \end{equation}
 \begin{equation}\label{transport equation of chi along T}
  \begin{split}
   \Lies_T\chi_{AB}
   &=-c^{-1}L(c^{-1}\mu)\chi_{AB}+\frac{1}{2}c^{-1}\mu(\si_{AC}\chi_B^C+\si_{BC}\chi_A^C)\\
   &\quad+\frac{1}{2}(\nas_A\et_B+\nas_B\et_A)+\frac{1}{2}\mu^{-1}(\ze_A\et_B+\ze_B\et_A),
  \end{split}
 \end{equation}
 \begin{equation}\label{elliptic equation of chi}
  (\divgs\chi)_A=\ds_A\trgs\chi+c^{-1}(\ds^Bc)\chi_{AB}-c^{-1}(\ds_Ac)\trgs\chi,
 \end{equation}
 \begin{equation}\label{transport equation of L^i}
  LL^i=c^{-1}(Lc)\check{L}^i+c^{-1}(Lc)\rho^{-1}x^i-c\ds x^i\cdot\ds c,
 \end{equation}
 \begin{equation}\label{TLi}
 TL^i=(-c^{-1}\mu\ds_Ac+\ds_A\mu)\ds^Ax^i-c^{-1}L(c^{-1}\mu)L^i,
 \end{equation}
 \begin{equation}\label{structure equation of L^i small}
  \ds_A\check{L}^i=\check{\chi}_{AB}\ds^Bx^i.
 \end{equation}
\end{lemma}
\begin{proof}
Since \eqref{transport equation of mu}-\eqref{elliptic equation of chi} are completely similar to (2.32)-(2.36) in \cite{MY17},
we omit the details here.

With the help of Lemma \ref{Lemma connection coefficients} and the fact $\p_i=c^2\mu^{-2}T^iT+\ds^Ax^iX_A$,
one has that by \eqref{T}, Definition \ref{Definition error vectors and trace-free vectors}, Lemma \ref{Lemma connection coefficients} and \eqref{transport equation of mu},
 \begin{equation*}
  \begin{split}
   LL^i&=\D_LL^i-L^{\al}L^{\be}\Ga_{\al\be}^i=\mu^{-1}L\mu L^i-c\p_ic\\
   &=c^{-1}(Lc)\check{L}^i+c^{-1}(Lc)\rho^{-1}x^i-c\ds x^i\cdot\ds c.
  \end{split}
 \end{equation*}
Analogously, \eqref{TLi} can be obtained by $TL^i=\D_TL^i-L^{\al}T^{j}\Ga_{\al j}^i$
and Lemma \ref{Lemma connection coefficients}.

Finally, by (4.10c) in \cite{Sp16} and Definition \ref{Definition G function}, we arrive at
 \begin{equation*}
  \begin{split}
   \ds_A\check{L}^i
   &=\gs^{BC}\check{\chi}_{AB}\ds_Cx^i+\gs^{BC}(-\frac{1}{2}\slashed{G}_{AB}^{\al}L\vp_{\al}
   +\frac{1}{2}\slashed{G}_{LA}^{\al}\ds_B\vp_{\al}-\frac{1}{2}\slashed{G}_{LB}^{\al}\ds_A\vp_{\al})\ds_Cx^i\\
   &\quad+(-G_{L\Tr}^{\al}\ds_A\vp_{\al}-\frac{1}{2}G_{\Tr\Tr}^{\al}\ds_A\vp_{\al})\Tr^i\\
   &=\check{\chi}_{AB}\ds^Bx^i.
  \end{split}
 \end{equation*}
\end{proof}

\section{Bootstrap assumptions and lower order derivative estimates}\label{Section 3}

To show the global existence of solution $\phi$ to equation$\eqref{main equation}$ near $C_0$, we will utilize the bootstrap argument. For this purpose, we make the following bootstrap assumptions in $D^{\mathfrak t,u}$
\begin{equation}\label{BA}
  \begin{split}
      \de^l\|LZ^{\al}\vp_{\ga}\|_{L^{\infty}(\Si_\mathfrak t^u)}&\leq M\de^{1-\ve0}\mathfrak t^{-2},\\
   \de^l\|\ds Z^{\al}\vp_{\ga}\|_{L^{\infty}(\Si_\mathfrak t^u)}&\leq M\de^{1-\ve0}\mathfrak t^{-2},\\
      \de^l\|TZ^{\al}\vp_{\ga}\|_{L^{\infty}(\Si_\mathfrak t^u)}&\leq M\de^{-\ve0}\mathfrak t^{-1},\\
             \|\nas^2\vp_{\ga}\|_{L^{\infty}(\Si_\mathfrak t^u)}&\leq M\de^{1-\ve0}\mathfrak t^{-3},\\
                   \|\vp_{\ga}\|_{L^{\infty}(\Si_\mathfrak t^u)}&\leq M\de^{1-\ve0}\mathfrak t^{-1},
  \end{split}
 \end{equation}
where $|\alpha|\leqslant N$, $N$ is a large positive integer, $M$ is some positive number which is
suitably chosen (at least double bounds of the corresponding quantities on time $t_0$), $
Z\in\{\rho L,T,R_1,R_2,R_3\}$, $l$ is the number of $T$ included in $Z^{\alpha}$.

As $L\mu=\mu c^{-1}Lc-cTc$ by \eqref{transport equation of mu}, then $\|L\mu\|_{L^{\infty}(\Si_\mathfrak t^u)}\lesssim M^p\de^{(1-\ve0)p-1}\mathfrak t^{-p}+M^p\de^{(1-\ve0)p}\mathfrak t^{-(p+1)}\mu$ from \eqref{BA}, which deduces that by
integrating $L\mu$ along integral curves of $L$,
\begin{equation}\label{estimate of mu}
\|\mu-1\|_{L^{\infty}(\Si_\mathfrak t^u)}\lesssim M^p\de^{(1-\ve0)p-1}.
\end{equation}
Hence,
\begin{equation}\label{estimate of L mu}
\|L\mu\|_{L^{\infty}(\Si_\mathfrak t^u)}\lesssim M^p\de^{(1-\ve0)p-1}\mathfrak t^{-p}.
\end{equation}
In addition, taking the $i$-th component on both sides of $\p_i=c^2\mu^{-2}T^iT+\ds^Ax^iX_A$ yields
\begin{equation*}
\begin{split}
1=c^2\mu^{-2}|T^i|^2+\ds^Ax^i\ds_Ax^i=c^2\mu^{-2}|T^i|^2+|\ds x^i|^2,
\end{split}
\end{equation*}
which immediately gives
\begin{equation}\label{estimate of ds x^i and gs}
|\ds x^i|\lesssim1.
\end{equation}

Similarly to the estimate of $\mu$ in \eqref{estimate of mu}, $\check\chi$ can also be estimated by integrating $L\check\chi$
along integral curves of $L$.
\begin{proposition}
For sufficiently small $\de>0$, it holds that
 \begin{equation}\label{estimate of chi small}
  \|\chic\|_{L^{\infty}(\Si_\mathfrak t^u)}\lesssim M^p\de^{(1-\ve0)p}\mathfrak t^{-2}.
 \end{equation}
\end{proposition}
\begin{proof}
Substituting $\chi_{AB}=\chic_{AB}+\rho^{-1}\gs_{AB}$ into $\eqref{transport equation of chi}$, in view of $L\gs_{AB}=2\chi_{AB}$, one has
 \begin{equation}\label{transport equation of chi small}
   L\chic_{AB}=c^{-1}(Lc)\chic_{AB}+\chic_A^C\chic_{BC}+c^{-1}(Lc)\rho^{-1}\gs_{AB}-\check{\R}_{LALB},
 \end{equation}
and then,
 \begin{equation}\label{weighted transport equation of |chi small|^2}
  L(\rho^4|\chic|^2)=2\rho^4(c^{-1}Lc|\chic|^2-\chic_A^B\chic_B^C\chic_C^A+\rho^{-1}c^{-1}Lc\trgs\chic-\chic^{AB}\check{\R}_{LALB}).
 \end{equation}
Note that $\check{\R}_{LALB}$'s explicit expression has been given in \eqref{RLALB small} and \eqref{RLALB}.
Then by \eqref{BA} and \eqref{estimate of ds x^i and gs}, we obtain
\begin{equation}\label{HC-2}
|L(\rho^2|\chic|)|
\lesssim M^p\de^{(1-\ve0)p}\mathfrak t^{-p}+M^p\de^{(1-\ve0)p}\mathfrak t^{-2}\cdot\rho^2|\chic|+\rho^2|\check\chi|^2.
\end{equation}

Due to
 \begin{equation}\label{estimate of chi small t0}
   |\chic_{AB}|_{t_0}|=|(c-1)\frac{1}{r}\gs_{AB}|_{t_0}|\lesssim M^p\de^{(1-\ve0)p},
 \end{equation}
then integrating \eqref{HC-2} along integral curves of $L$ from $t_0$ to $\mathfrak t$ yields
 \begin{equation*}
  \begin{split}
   |\chic|
   &\lesssim M^p\de^{(1-\ve0)p}\mathfrak t^{-2}.
  \end{split}
 \end{equation*}
\end{proof}

With the help of \eqref{estimate of chi small}, we can estimate $\ds\mu$.
\begin{proposition}
For sufficiently small $\de>0$, it holds that
 \begin{align}
    &\|\ds\mu\|_{L^{\infty}(\Si_\mathfrak t^u)}\lesssim M^p\de^{(1-\ve0)p-1}\mathfrak t^{-1}.\label{estimate of ds mu}
 \end{align}
\end{proposition}
\begin{proof}
By \eqref{transport equation of mu}, one has
 \begin{equation}\label{HC-3}
  \begin{split}
   \Lies_L\ds\mu&=\ds L\mu=c^{-1}Lc\ds\mu+\ds(cTc)+\mu\ds(c^{-1}Lc).
  \end{split}
 \end{equation}
Then
  \begin{equation}\label{weighted transport equation of |ds mu|^2}
   L(\rho^2|\ds\mu|^2)=2\rho^2\big\{-\chic^{AB}\ds_A\mu\ds_B\mu+c^{-1}Lc|\ds\mu|^2+\big(\ds_A(cTc)+\mu\ds_A(c^{-1}Lc)\big)\ds^A\mu\big\}.
 \end{equation}
Together with \eqref{estimate of ds x^i and gs}, \eqref{estimate of chi small} and \eqref{BA}, this yields
 \begin{equation*}
  L(\rho|\ds\mu|)\lesssim M^p\de^{(1-\ve0)p-1}\mathfrak t^{-p}+M^p\de^{(1-\ve0)p}\mathfrak t ^{-2}\cdot\rho|\ds\mu|.
 \end{equation*}
Thus,
 \begin{equation*}
  |\ds\mu|\lesssim M^p\de^{(1-\ve0)p-1}\mathfrak t^{-1}.
 \end{equation*}
\end{proof}

It follows from \eqref{zeta} and \eqref{eta} that
\begin{corollary}
For sufficiently small $\de>0$, one has
 \begin{equation}\label{estimate of ze and et}
  \begin{split}
       \|\ze\|_{L^{\infty}(\Si_\mathfrak t^u)}&\lesssim M^p\de^{(1-\ve0)p}\mathfrak t^{-(p+1)},\\
       \|\et\|_{L^{\infty}(\Si_\mathfrak t^u)}&\lesssim M^p\de^{(1-\ve0)p-1}\mathfrak t^{-1}.
    \end{split}
 \end{equation}
\end{corollary}

Based on \eqref{estimate of ze and et}, we now estimate $T\mu$.
\begin{proposition}
For sufficiently small $\de>0$, it holds that
 \begin{equation}\label{estimate of T mu}
  \|T\mu\|_{L^{\infty}(\Si_\mathfrak t^u)}\lesssim M^p\de^{(1-\ve0)p-2}.
 \end{equation}
\end{proposition}
\begin{proof}
It follows from \eqref{transport equation of mu} and Lemma \ref{Lemma connection coefficients} that
\begin{equation*}
  \begin{split}
   LT\mu&=TL\mu+[L,T]\mu\\
        &=c^{-1}LcT\mu+\big[-T(cTc)+\mu T(c^{-1}Lc)-(\ze^A+\et^A)\ds_A\mu\big].
  \end{split}
 \end{equation*}
In addition, by \eqref{BA}, \eqref{estimate of mu}, \eqref{estimate of ze and et}
and \eqref{estimate of ds mu}, one has
 \begin{equation*}
  |LT\mu|\lesssim M^p\de^{(1-\ve0)p}\mathfrak t^{-(p+1)}|T\mu|+M^p\de^{(1-\ve0)p-2}\mathfrak t^{-2}.
 \end{equation*}
By Gronwall's inequality, we arrive at
 \begin{equation*}
  |T\mu|\lesssim M^p\de^{(1-\ve0)p-2}.
 \end{equation*}
\end{proof}

At the last of this section, we give the lower order derivative $L^{\infty}$ estimates of some related quantities
which will be further utilized in subsequent sections.
\begin{lemma}\label{Lem4.1}
For sufficiently small $\de>0$, it holds that
 \begin{equation}\label{lower order estimate}
  \begin{split}
   &|\check{L}^j|+|\check{T}^j|+|R_i\check{L}^j|+|R_i\check{T}^j|+\rho|L\check L^j|+\rho|L\check T^j|\lesssim M^p\de^{(1-\ve0)p}\mathfrak t^{-1},\\
   &{|T\check L^j|+|T\check T^j|}\lesssim M^p\delta^{(1-\ve0)p-1}\mathfrak t^{-1},\\
   &|v_j|+|R_iv_j|+\rho|Lv_j|\lesssim M^p\de^{(1-\ve0)p},\qquad{|Tv_j|}\lesssim M^p\delta^{(1-\ve0)p},\\
   &|\Lies_{R_i}\ds x^j|+|\Lies_{\rho L}\ds x^j|\lesssim 1,\qquad|\Lies_T\ds x^j|\lesssim\mathfrak t^{-1},\\
   &|\Lies_{R_i}R_j|\lesssim \mathfrak t,\qquad|\Lies_{\rho L}R_j|\lesssim M^p\delta^{(1-\ve0)p},\qquad|\Lies_TR_j|\lesssim M^p\delta^{(1-\ve0)p}\mathfrak t^{-1},\\
   &|{}^{(R_i)}\pis|+|{}^{(R_i)}\pis_L|+|{}^{(R_i)}\pis_T|\lesssim M^p\de^{(1-\ve0)p}\mathfrak t^{-1},\\
   &|{}^{(T)}\pis|\lesssim\mathfrak t^{-1}.
  \end{split}
 \end{equation}
\end{lemma}
\begin{proof}
We deal with these quantities in \eqref{lower order estimate} one by one.

\vskip 0.2 true cm

\textbf{Part 1. Bounds of $\check{L}^j,\ \check{T}^j, R_i\check{L}^j, R_i\check{T}^j$, $v_j$ and $R_iv_j$}

\vskip 0.1 true cm

It is derived from Definition \ref{Definition error vectors and trace-free vectors} and \eqref{transport equation of L^i} that
\begin{equation}\label{weighted transport equation of L^i small}
\begin{split}
L(\rho\check{L}^i)=c^{-1}(Lc)\rho\check{L}^i+c^{-1}(Lc)x^i-\rho c(\ds_A x^i)\ds^A c,
\end{split}
\end{equation}
then
\begin{equation*}
|L(\rho\check{L}^i)|\lesssim M^p\de^{(1-\ve0)p}\mathfrak t^{-(p+1)}|\rho\check{L}^j|+M^p\de^{(1-\ve0)p}\mathfrak t^{-p}.
\end{equation*}
This, together with Gronwall's inequality, yields
\begin{equation*}
|\check{L}^j|\lesssim M^p\de^{(1-\ve0)p}\mathfrak t^{-1},
\end{equation*}
which also implies
\begin{equation*}
|\check{T}^j|\lesssim M^p\de^{(1-\ve0)p}\mathfrak t^{-1}
\end{equation*}
in view of Lemma \ref{Lemma relation of error vectors and trace-free vectors}.
In addition, it follows from \eqref{structure equation of L^i small}, \eqref{estimate of chi small}, \eqref{transport equation of L^i} and Lemma \ref{Lemma relation of error vectors and trace-free vectors} that
 \begin{equation*}
  |\ds\check{L}^j|+|\ds\check{T}^j|+|L\check L^j|+|L\check T^j|\lesssim M^p\de^{(1-\ve0)p}\mathfrak t^{-2}.
 \end{equation*}
Together with $R_i\sim r\ds$,
we have
 \begin{equation*}
  |R_i\check{L}^j|+|R_i\check{T}^j|\lesssim M^p\de^{(1-\ve0)p}\mathfrak t^{-1}.
 \end{equation*}
 On the other hand, since $T\check L^i=TL^i-\f{c^{-1}\mu}{\rho}\check T^i+\f{(\mu-1)x^i}{c\rho^2}+\f{c^{-1}x^i}{\rho^2}(1-c)$, then
 \begin{equation*}
 |T\check L^i|+|T\check T^i|\lesssim M^p\delta^{(1-\ve0)p-1}\mathfrak t^{-1}
 \end{equation*}
 by Lemma \ref{Lemma relation of error vectors and trace-free vectors}.
 The estimates for $v_j$, $R_iv_j$, $Lv_j$ and $Tv_j$ can be obtained directly after using the definition of $v_j$ in \eqref{vi}.

 \vskip 0.2 true cm

\textbf{Part 2. Bounds of $\Lies_{R_i}\ds x^j,, \Lies_{\rho L}\ds x^j,\ \Lies_{T}\ds x^j$ and $\Lies_{R_i}R_j$}

\vskip 0.1 true cm

Due to ${R_i}^j={\Om_i}^j-v_i\Tr^j=\ep_{ik}{}^jx^k-v_i\Tr^j$, then
 \begin{equation*}
  |{R_i}^j|\lesssim\mathfrak t.
 \end{equation*}
In view of \eqref{Ri} and Definition \ref{Definition error vectors and trace-free vectors}, we have
$R_kR_ix^j=R_k{R_i}^j=R_k({\Om_i}^j-v_i\Tr^j)$.
Then it follows from the estimates in Part 1 that
 \begin{equation*}
  |R_kR_ix^j|
  \lesssim \mathfrak t.
 \end{equation*}
Thus,
 \begin{equation*}
  |\Lies_{R_i}\ds x^j|=|\ds R_ix^j|\lesssim r^{-1}\sum_{k=1}^3|R_kR_ix^j|\lesssim 1
 \end{equation*}
and
 \begin{equation*}
  \begin{split}
   |\Lies_{R_i}R_j^k|
   &=|[R_i,R_j]^k|=|R_iR_jx^k-R_jR_ix^k|\lesssim\mathfrak t.
  \end{split}
 \end{equation*}
 Similarly, the estimate for $\Lies_{\rho L}\ds x^j$ and $\Lies_T\ds x^j$ could be obtained by the facts $\Lies_{\rho L}\ds x^j=\ds(\rho L^j)=\rho\ds\check L^j+\ds x^j$ and $\Lies_T\ds x^j=\ds (c^{-1}\mu(\check T^j-\f{x^j}{\rho}))$.

  \vskip 0.2 true cm

\textbf{Part 3. Bounds of ${}^{(R_i)}\pis,\ {}^{(R_i)}\pis_L,\ {}^{(R_i)}\pis_T,\ {}^{(T)}\pis,\ \Lies_{\rho L}R_j$ and $\Lies_T R_j$}

 \vskip 0.1 true cm

It follows from \eqref{Ri pi} and the estimates of $\check{L}^i,\ v_i,\ \chic, \mu$ and $\ds\mu$ that
 \begin{equation*}
  \begin{split}
   |{}^{(R_i)}\pis|+|{}^{(R_i)}\pis_{L}|+|{}^{(R_i)}\pis_{T}|
   &\lesssim M^p\de^{(1-\ve0)p}\mathfrak t^{-1}.
  \end{split}
 \end{equation*}
In addition, by \eqref{T pi}, we have
 \begin{equation*}
  \begin{split}
   |{}^{(T)}\pis|
   &=2|c^{-2}\mu\chi|\lesssim\mathfrak t^{-1}.
  \end{split}
 \end{equation*}
 Since $\Lies_{\rho L}R_j=\rho {}^{(R_j)}\pis_{L}$ and $\Lies_{T}R_j={}^{(R_j)}\pis_{T}$, the estimates of $\Lies_{\rho L}R_j$ and $\Lies_{T}R_j$ can be obtained directly.
\end{proof}

\section{Energy estimates for the linearized covariant wave equation}\label{Section 4}

In this section, we focus on the global energy estimates for the smooth function $\Psi$ to the following linearized covariant wave equation
\begin{equation}\label{linearized covariant wave equation}
 \mu\Box_g\Psi=\Phi,
\end{equation}
where $\Psi$ and its derivatives vanish on $C_0^\mathfrak t$. This procedure is divided into the following four steps.

\vskip 0.2 true cm

\textbf{Step 1.} New divergence form of \eqref{linearized covariant wave equation}

\vskip 0.2 true cm

By \eqref{Definition energy-momentum tensor}, one has
 \begin{equation*}
   \Box_g\Psi\cdot\p_{\be}\Psi=\D^{\al}Q_{\al\be}.
 \end{equation*}
Then for any vector field $V$, it follows from \eqref{Definition deformation tensor} that
 \begin{equation}\label{key formula}
   \Box_g\Psi\cdot V\Psi=\D_{\al}{}^{(V)}J^{\al}-\frac{1}{2}Q^{\al\be}[\Psi]{}^{(V)}\pi_{\al\be},
 \end{equation}
where ${}^{(V)}J^{\al}=Q_{\be}^{\al}V^{\be}$ with $Q_\beta^\al=g^{\al\gamma}Q_{\beta\gamma}$ and $Q^{\al\beta}=g^{\al\al'}g^{\beta\beta'}Q_{\al'\beta'}$.

\vskip 0.2 true cm

\textbf{Step 2.} Integration by parts on domain $D^{\mathfrak t,u}$

\vskip 0.2 true cm

Under the optical coordinate $\{\mathfrak t,u,\vt^1,\vt^2\}$, we have
 \begin{equation}\label{J under the optical coordinate}
  {}^{(V)}J={}^{(V)}J^\mathfrak t\frac{\p}{\p\mathfrak t}+{}^{(V)}J^u\frac{\p}{\p u}+{}^{(V)}\slashed{J}^A\frac{\p}{\p\vt^A}.
 \end{equation}
Denoting $N=\p_t=L+c^2\mu^{-1}T$ by the normal vector. Then by taking the inner product with $N$ on both sides of \eqref{J under the optical coordinate}, one has
 \begin{equation*}
  \begin{split}
   {}^{(V)}J_N
   &={}^{(V)}J^\mathfrak tg(L,N)+{}^{(V)}J^ug(T,N)+{}^{(V)}\slashed{J}^Ag(X_A,N)\\
   &=-c^2{}^{(V)}J^\mathfrak t.
  \end{split}
 \end{equation*}
Hence, ${}^{(V)}J^\mathfrak t=-c^{-2}{}^{(V)}J_N$. Similarly, ${}^{(V)}J^u=-\mu^{-1}{}^{(V)}J_L$. Therefore,
it follows from $\sqrt{|\det g|}=\mu\sqrt{\det\gs}$ that
 \begin{equation*}
  \begin{split}
   \D_{\al}{}^{(V)}J^{\al}
   &=\frac{1}{\sqrt{|\det g|}}\big[\frac{\p}{\p\mathfrak t}(\sqrt{|\det g|}{}^{(V)}J^\mathfrak t)+\frac{\p}{\p u}(\sqrt{|\det g|}{}^{(V)}J^u)+\frac{\p}{\p\vt^A}(\sqrt{|\det g|}{}^{(V)}\slashed{J}^A)\big]\\
   &=\frac{1}{\sqrt{|\det g|}}\big[\frac{\p}{\p\mathfrak t}(-c^{-2}\mu{}^{(V)}J_N\sqrt{\det\gs})+\frac{\p}{\p u}(-{}^{(V)}J_L\sqrt{\det\gs})+\frac{\p}{\p\vt^A}(\sqrt{|\det g|}{}^{(V)}\slashed{J}^A)\big].
  \end{split}
 \end{equation*}
Integrating over $D^{\mathfrak t,u}$ to obtain
 \begin{equation}\label{divergence theorem}
  -\int_{D^{\mathfrak t,u}}\mu\D_{\al}{}^{(V)}J^{\al}=\int_{\Si_\mathfrak t^u}c^{-2}\mu{}^{(V)}J_N-\int_{\Si_{t_0}^u}c^{-2}\mu{}^{(V)}J_N+\int_{C_u^\mathfrak t}{}^{(V)}J_L,
 \end{equation}
where
 \begin{equation*}
  {}^{(V)}J_N={}^{(V)}J^{\al}N_{\al}=Q_{\be}^{\al}V^{\be}N_{\al}=Q_{VN},\ {}^{(V)}J_L=Q_{VL}.
 \end{equation*}

 \vskip 0.2 true cm

\textbf{Step 3.} Energy identity

\vskip 0.2 true cm

Choosing two vector fields $V_1=\rho^{2s}L,V_2=\Lb$ as multipliers with $0<s<\frac{1}{2}$ being
any fixed constant. By $T=\frac{1}{2}(\Lb-c^{-2}\mu L)$ and \eqref{components of energy-momentum tensor}, then
 \begin{equation*}
  \begin{split}
   Q_{V_1N}&=\frac{1}{2}\rho^{2s}\big[(L\Psi)^2+c^2|\ds\Psi|^2\big],\\
      Q_{V_1L}&=\rho^{2s}(L\Psi)^2,\\
         Q_{V_2N}&=\frac{1}{2}\big[c^2\mu^{-1}(\Lb\Psi)^2+\mu|\ds\Psi|^2\big],\\
      Q_{V_2L}&=\mu|\ds\Psi|^2.
  \end{split}
 \end{equation*}
By \eqref{divergence theorem} and \eqref{key formula}, we have the following energy identity
 \begin{equation}\label{energy identity}
  E_i[\Psi](\mathfrak t,u)-E_i[\Psi](t_0,u)+F_i[\Psi](\mathfrak t,u)=-\int_{D^{\mathfrak t,u}}\Phi\cdot V_i\Psi-\int_{D^{\mathfrak t,u}}\frac{1}{2}\mu Q^{\al\be}[\Psi]^{(V_i)}\pi_{\al\be},\ i=1,2,
 \end{equation}
where the \textit{energies} $E_i[\Psi](t,u)$ and \textit{fluxes} $F_i[\Psi](t,u)(i=1,2)$ are defined as follows
 \begin{equation}\label{Definition of energy and flux}
  \begin{split}
   E_1[\Psi](\mathfrak t,u)               &=\int_{\Si_\mathfrak t^u}\frac{1}{2}\rho^{2s}\big[c^{-2}\mu(L\Psi)^2+\mu|\ds\Psi|^2\big],\\
   E_2[\Psi](\mathfrak t,u)
                 &=\int_{\Si_\mathfrak t^u}\frac{1}{2}\big[(\Lb\Psi)^2+c^{-2}\mu^2|\ds\Psi|^2\big],\\
   F_1[\Psi](\mathfrak t,u)&=\int_{C_u^\mathfrak t}\rho^{2s}(L\Psi)^2,\\
   F_2[\Psi](\mathfrak t,u)&=\int_{C_u^\mathfrak t}\mu|\ds\Psi|^2.
  \end{split}
 \end{equation}

 \vskip 0.2 true cm

\textbf{Step 4.} Error estimates and the energy inequality

\vskip 0.2 true cm

Next, we deal with the second integral in the right hand side of $\eqref{energy identity}$. By Remark \ref{Remark componets of gs},
one has
 \begin{equation*}
  \begin{split}
   &\quad-\frac{1}{2}\mu Q^{\al\be}[\Psi]{}^{(V)}\pi_{\al\be}=-\frac{1}{2}\mu g^{\al\ka}g^{\be\la}Q_{\ka\la}[\Psi]{}^{(V)}\pi_{\al\be}\\
   &=-\frac{1}{2}\mu\big[-\frac{1}{2}\mu^{-1}(L^{\al}\Lb^{\ka}+\Lb^{\al}L^{\ka})+\gs^{AB}X_A^{\al}X_B^{\ka}\big]\big[-\frac{1}{2}\mu^{-1}(L^{\be}\Lb^{\la}+\Lb^{\be}L^{\la})+\gs^{CD}X_C^{\be}X_D^{\la}\big]Q_{\ka\la}{}^{(V)}\pi_{\al\be}\\
   &=-\frac{1}{8}\mu^{-1}(Q_{LL}{}^{(V)}\pi_{\Lb\Lb}+Q_{\Lb\Lb}{}^{(V)}\pi_{LL})-\frac{1}{4}\mu^{-1}Q_{L\Lb}{}^{(V)}\pi_{L\Lb}+\frac{1}{2}(Q_L^A{}^{(V)}\pis_{\Lb A}+Q_{\Lb}^A{}^{(V)}\pis_{LA})-\frac{1}{2}\mu Q^{AB}{}^{(V)}\pis_{AB}.
  \end{split}
 \end{equation*}
Then by \eqref{components of energy-momentum tensor}, we obtain
 \begin{equation}\label{Q*pi 1}
  \begin{split}
   &-\frac{1}{2}\mu Q^{\al\be}[\Psi]{}^{(V_1)}\pi_{\al\be}\\
   =&\big[-\frac{1}{2}\rho^{2s}L(c^{-2}\mu)+s\rho^{2s-1}c^{-2}\mu-2s\rho^{2s-1}\big](L\Psi)^2+(\frac{1}{2}\rho^{2s}L\mu+s\rho^{2s-1}\mu)|\ds\Psi|^2\\
   &+\rho^{2s}c^2\ds_A(c^{-2}\mu)L\Psi\ds^A\Psi-\rho^{2s}\mu\check\chi_{AB}\ds^A\Psi\ds^B\Psi+\f12\rho^{2s}\mu\trgs\check\chi|\ds\Psi|^2-\frac{1}{2}\rho^{2s}\trgs\chi L\Psi\Lb\Psi\\
  \end{split}
 \end{equation}
and
 \begin{equation}\label{Q*pi 2}
  \begin{split}
   &-\frac{1}{2}\mu Q^{\al\be}[\Psi]{}^{(V_2)}\pi_{\al\be}\\
   =&\frac{1}{2}\big[\Lb\mu+\mu L(c^{-2}\mu)\big]|\ds\Psi|^2-\mu\ds_A(c^{-2}\mu)L\Psi\ds^A\Psi-c^2\ds_A(c^{-2}\mu)\Lb\Psi\ds^A\Psi\\
   &+c^{-2}\mu^2\check\chi_{AB}\ds^A\Psi\ds^B\Psi-\f12c^{-2}\mu^2\trgs\check\chi|\ds\Psi|^2+\frac{1}{2}c^{-2}\mu\trgs\chi L\Psi\Lb\Psi.
  \end{split}
 \end{equation}
Applying the results in Section \ref{Section 3} to estimate all the coefficients in \eqref{Q*pi 1} and \eqref{Q*pi 2} yields
 \begin{equation}\label{estimate of Q*pi 1}
  \begin{split}
   &\int_{D^{\mathfrak t,u}}-\frac{1}{2}\mu Q^{\al\be}[\Psi]{}^{(V_1)}\pi_{\al\be}\\
   \lesssim&\int_{D^{\mathfrak t,u}}\Big\{\delta^{-1}\rho^{2s}(L\Psi)^2+|\ds\Psi|^2+\tau^{-2+2s}\cdot\de(\Lb\Psi)^2\Big\}\\
   \lesssim&\de^{-1}\int_0^uF_1[\Psi](\mathfrak t, u')\d u'+\int_{t_0}^\mathfrak t\tau^{-2+2s}\cdot\de E_2[\Psi](\tau,u)\d\tau+\de^{-1}\int_0^u\cdot\de F_2[\Psi](\mathfrak t, u')\d u'
  \end{split}
 \end{equation}
and
 \begin{equation}\label{estimate of Q*pi 2}
  \begin{split}
   &\de\int_{D^{\mathfrak t,u}}-\frac{1}{2}\mu Q^{\al\be}[\Psi]{}^{(V_2)}\pi_{\al\be}\\
   \lesssim&\int_{D^{\mathfrak t,u}}\Big\{M^p\de^{(1-\ve0)p-2}\cdot\de|\ds\Psi|^2+\big[\de(L\Psi)^2+\tau^{-2}\cdot\de(\Lb\Psi)^2\big]\Big\}\\
   \lesssim&\int_0^uM^p\de^{(1-\ve0)p-2}\cdot\de F_2[\Psi](\mathfrak t, u')\d u'+\int_0^u\de F_1[\Psi](\mathfrak t, u')\d u'+\int_{t_0}^\mathfrak t\tau^{-2}\cdot\de E_2[\Psi](\tau, u)\d\tau.
  \end{split}
 \end{equation}
Substituting \eqref{estimate of Q*pi 1} and \eqref{estimate of Q*pi 2} into \eqref{energy identity} and using the Gronwall's inequality, we obtain
 \begin{equation}\label{energy inequality}
  \begin{split}
   &\quad E_1[\Psi](\mathfrak t,u)+F_1[\Psi](\mathfrak t,u)+\de E_2[\Psi](\mathfrak t,u)+\de F_2[\Psi](\mathfrak t,u)\\
   &\lesssim E_1[\Psi](t_0,u)+\de E_2[\Psi](t_0,u)+\int_{D^{\mathfrak t,u}}\rho^{2s}|\Phi\cdot L\Psi|+\de\int_{D^{\mathfrak t,u}}|\Phi\cdot\Lb\Psi|.
  \end{split}
 \end{equation}

At the last of this subsection, we define the higher order energy and flux as follows
 \begin{equation}\label{definition of high order energy and flux}
  \begin{split}
   E_{i,m+1}(\mathfrak t,u)=\sum_{\ga=0}^3\sum_{|\al|=m}&\de^{2l}E_i[Z^{\al}\vp_{\ga}](\mathfrak t,u),\ i=1,2,\\
   F_{i,m+1}(\mathfrak t,u)=\sum_{\ga=0}^3\sum_{|\al|=m}&\de^{2l}F_i[Z^{\al}\vp_{\ga}](\mathfrak t,u),\ i=1,2,\\
   E_{i,\leq m+1}(\mathfrak t,u)=\sum_{0\leq n\leq m}E_{i,n+1}(\mathfrak t,u),\quad&\Et_{i,\leq m+1}(\mathfrak t,u)=\sup_{t_0\leq\tau\leq \mathfrak t}E_{i,\leq m+1}(\tau,u),\ i=1,2,\\
   F_{i,\leq m+1}(\mathfrak t,u)=\sum_{0\leq n\leq m}F_{i,n+1}(\mathfrak t,u),\quad&\tilde F_{i,\leq m+1}(\mathfrak t,u)=\sup_{t_0\leq\tau\leq \mathfrak t}F_{i,m+1}(\tau,u),\ i=1,2,
  \end{split}
 \end{equation}
where $l$ is the number of $T$ appearing in $Z^{\al}$.

\section{Non-top order derivative estimates}\label{Section 5}

\subsection{$L^{\infty}$ estimates}\label{Section 5.1}

In Section \ref{Section 3}, we have obtained the lower order $L^{\infty}$ estimates for some quantities. Since our aim is to close the bootstrap assumptions \eqref{BA}, the results obtained in Section \ref{Section 3} are far from enough. For this purpose, we give the higher order $L^{\infty}$ estimates.

\begin{proposition}\label{Proposition higher order L^infty estimates}
Under the assumptions \eqref{BA}, for any vector field $Z\in\{\rho L,T,R_1,R_2,R_3\}$, when $\de>0$ is small,
it holds that for $|\al|\leq N-1$,
 \begin{equation}\label{P6.1}
  \begin{split}
   &\de^l|\Lies_{Z}^{\al}\chic|\lesssim M^p\de^{(1-\ve0)p}\mathfrak t^{-2},\qquad \de^l|Z^{\al+1}\mu|\lesssim M^p\de^{(1-\ve0)p-1},\\
   &\de^l|Z^{\al+1}\check{L}^j|\lesssim M^p\de^{(1-\ve0)p}\mathfrak t^{-1},\qquad\delta^l|Z^{\al+1}v_j|\lesssim M^p\de^{(1-\ve0)p},\\
   &\de^l|\Lies_{Z}^{\al+1}\ds x^j|\lesssim 1,\qquad \de^l|\Lies_{Z}^{\al+1}R_j|\lesssim\mathfrak t,\qquad\delta^l|{\Lies_{Z}^{\al}}{}^{(T)}\pis|\lesssim\mathfrak t^{-1},\\
   &\de^l\big(|{\Lies_{Z}^{\al}}{}^{(R_j)}\pis_L|+|{\Lies_{Z}^{\al}}{}^{(R_j)}\pis_T|+|{\Lies_{Z}^{\al}}{}^{(R_j)}\pis|\big)\lesssim M^p\de^{(1-\ve0)p}\mathfrak t^{-1},
  \end{split}
 \end{equation}
where $l$ is the number of $T$ appearing in the string of $Z$.
\end{proposition}
\begin{proof}
We will prove this proposition by the induction method with respect to the index $\al$. Note that we have proved \eqref{P6.1} for the case $\al=0$ in Section \ref{Section 3}.

We first prove \eqref{P6.1} which only involves the rotational vector fields, that is, $Z\in\{R_1, R_2, R_3\}$.
Assume that \eqref{P6.1} holds up to the order $k-1$ $(1\leq k\leq N-1)$,
one needs to show that \eqref{P6.1} is also true for $\al$ with $|\al|=k$.
\begin{enumerate}
	\item
\textbf{Bounds of $\Lies_{R_i}^{\al}\chic$ and ${\Lies_{R_i}^{\al}}{}^{(R_j)}\pis_L$}

Using the expression of ${\Lies_{R_i}^{\al}}{}^{(R_j)}\pis_{LA}$ in \eqref{Ri pi} and the induction assumption,
we can check that
\begin{equation}\label{Y-35}
|\slashed{\mathcal{L}}_{R_i}^\al\leftidx{^{(R_j)}}{\slashed\pi}_{L}|\lesssim \mathfrak t|\slashed{\mathcal{L}}_{R_i}^\al\check{\chi}|+M^p\delta^{(1-\varepsilon_0)p}\mathfrak t^{-1}.
\end{equation}
This means that the estimate of ${\Lies_{R_i}^{\al}}{}^{(R_j)}\pis_L$ can be obtained once $\Lies_{R_i}^{\al}\chic$ is bounded.
Commuting $\Lies_{R_i}^{\al}$ with $L\chic_{AB}$ to derive
 \begin{equation}\label{commute Lies_Ri^al with L chic}
  \begin{split}
   \Lies_L\Lies_{R_i}^{\al}\chic_{AB}
   &=\Lies_{R_i}^{\al}(L\chic_{AB})+\sum_{\be_1+\be_2=\al-1}\Lies_{R_i}^{\be_1}\Lies_{[L,R_i]}\Lies_{R_i}^{\be_2}\chic_{AB}.
  \end{split}
 \end{equation}

Substituting \eqref{transport equation of chi small} and the identity $[L,R_i]={\leftidx{^{(R)}}{\slashed\pi}_{L}}^CX_C$
into \eqref{commute Lies_Ri^al with L chic}, and applying Lemma \ref{Lemma commutator [Des,Z]}, \eqref{Y-35} and induction argument,
we arrive at
 \begin{equation}\label{L^infty estimate of I1+I2}
  \begin{split}
   |\Lies_L\Lies_{R_i}^{\al}\chic|
   &\lesssim M^p\de^{(1-\ve0)p}\mathfrak t^{-2}|\Lies_{R_i}^{\al}\chic|+M^p\de^{(1-\ve0)p}\mathfrak t^{-(p+2)}.
  \end{split}
 \end{equation}
Similarly to \eqref{weighted transport equation of |chi small|^2}, one has
 \begin{equation}\label{weighted transport equation of Lies_{R_i}^{al} chic}
  L(\rho^4|\Lies_{R_i}^{\al}\chic|^2)=-4\rho^4\chic_A^B\cdot\Lies_{R_i}^{\al}\chic_B^C\cdot\Lies_{R_i}^{\al}\chic_C^A
  +2\rho^4\Lies_{R_i}^{\al}\chic^{AB}\cdot\Lies_L\Lies_{R_i}^{\al}\chic_{AB}.
 \end{equation}
Then combining \eqref{L^infty estimate of I1+I2} and \eqref{weighted transport equation of Lies_{R_i}^{al} chic} yields
 \begin{equation*}
  L(\rho^2|\Lies_{R_i}^{\al}\chic|)\lesssim M^p\de^{(1-\ve0)p}\mathfrak t^{-2}\cdot\rho^2|\Lies_{R_i}^{\al}\chic|+M^p\de^{(1-\ve0)p}\mathfrak t^{-p}.
 \end{equation*}
By Gronwall's inequality, then
 \begin{equation*}
 |\Lies_{R_i}^{\al}\chic|\lesssim M^p\de^{(1-\ve0)p}\mathfrak t^{-2}.
 \end{equation*}
This, together with \eqref{Y-35}, derives
  \begin{equation*}
 |\slashed{\mathcal{L}}_{R_i}^\al\leftidx{^{(R_j)}}{\slashed\pi}_{L}|\lesssim M^p\de^{(1-\ve0)p}\mathfrak t^{-1}.
 \end{equation*}

\item \textbf{Bounds of $R_i^{\al+1}\mu$}

Similarly to the estimate $\Lies_{R_i}^\al\chic$, we commute $R_i^{\al+1}$ with $L$ to get
 \begin{equation}\label{commute Ri^al with L mu}
  \begin{split}
   LR_i^{\al+1}\mu=R_i^{\al+1}L\mu+\sum_{\be_1+\be_2=\al}R_i^{\be_1}[L,R_i]R_i^{\be_2}\mu.
  \end{split}
 \end{equation}
By \eqref{transport equation of mu}, the induction assumptions and \eqref{BA}, one has
 \begin{equation*}
  \begin{split}
   |LR_i^{\al+1}\mu|\lesssim M^p\de^{(1-\ve0)p}\mathfrak t^{-2}|R_i^{\al+1}\mu|+M^p\de^{(1-\ve0)p-1}\mathfrak t^{-2}.
  \end{split}
 \end{equation*}
By Gronwall's inequality, then
 \begin{equation}
  |R_i^{\al+1}\mu|\lesssim M^p\de^{(1-\ve0)p-1}.\label{Rmu}
 \end{equation}

\item \textbf{Bounds of $R_i^{\al+1}\check{L}^j,\ R_i^{\al+1}v_j$, $\Lies_{R_i}^{\al+1}\ds x^j,\ \text{and}\ \Lies_{R_i}^{\al+1}R_j$}

It follows \eqref{structure equation of L^i small} that $R_i^{\al+1}\check{L}^j=R_i^{\al}(R_i^A\chic_{AB}\ds^Bx^j)$.
Then by the induction assumptions, the estimate of $\Lies_{R_i}^{\al}\chic$ and
the identity $\check{L}^i=-c\check{T}^i+(c-1)\rho^{-1}x^i$, one has
 \begin{equation*}
  |R_i^{\al+1}\check{L}^j|+|R_i^{\al+1}\check{T}^j|\lesssim M^p\de^{(1-\ve0)p}\mathfrak t^{-1},
 \end{equation*}
 which gives the estimate of $R_i^{\al+1}v_j$ directly by \eqref{vi}. And hence, $|R_i^{\al+2}x^j|\lesssim\mathfrak t$
 holds because of $R_i^{\al+2}x^j=R_i^{\al+1}(\Om_i^j-v_i\Tr^j)$. Therefore,
 \begin{equation*}
 |\Lies_{R_i}^{\al+1}\ds x^j|\lesssim 1,\ \text{and}\ |\Lies_{R_i}^{\al+1}R_j^k|\lesssim\mathfrak t.
 \end{equation*}

\item \textbf{Bounds of $\Lies_{R_i}^{\al}{}^{(R_j)}\pis,\ \Lies_{R_i}^{\al}{}^{(R_j)}\pis_T,\ \text{and}\ \Lies_{R_i}^{\al}{}^{(T)}\pis$}

According the expressions in \eqref{Ri pi} and \eqref{T pi}, the estimates of $\Lies_{R_i}^{\al}{}^{(R_j)}\pis$, $\Lies_{R_i}^{\al}{}^{(R_j)}\pis_T$ and $\Lies_{R_i}^{\al}{}^{(T)}\pis$ can be obtained immediately with the help of all the results
in Parts 1-3 and the induction argument.
\end{enumerate}

It follows from \eqref{transport equation of chi along T} and Definition \ref{Definition error vectors and trace-free vectors} that the formula of $\Lies_T\check\chi$ is obtained. After taking the Lie derivatives of this formula with respect to the rotational vector fields, one can estimate  $\Lies_{R_i}^{\al-1}\Lies_T\check\chi$ directly by using the estimates above and \eqref{BA}. Therefore, the bound
of $\Lies_Z^\al\check\chi$ can be obtained when $Z\in\{T, R_1, R_2, R_3\}$
and there is only one $T$ in $Z^\al$.

Similarly, if $Z\in\{T, R_1, R_2, R_3\}$ and $Z^{\al+1}$ only contains one $T$, then
$Z^{\al+1}\check L^j$, $\Lies_Z^{\al+1}\ds x^j$ and $\Lies_Z^{\al+1}R_j$ can be estimated
since the expressions of $T\check L^i$, $\Lies_T\ds x^j$ and $\Lies_T R_j$ have
appeared in the proof of Lemma \ref{Lem4.1}. With the same procedure for deriving \eqref{estimate of T mu}
or \eqref{Rmu}, $Z^{\al+1}\mu$ is also estimated. The reminder quantities such
as $Z^{\al+1}v_j$, ${\Lies_{Z}^{\al}}{}^{(R_j)}\pis_L$ e.t.c. that contain one derivative of $T$
and other derivatives of the rotational vector fields can be estimated by \eqref{vi}, \eqref{T pi} and \eqref{Ri pi}.

By induction argument with respect to the number of $T$, \eqref{P6.1} can be proved for $Z\in\{T, R_1, R_2, R_3\}$.

Finally, when the derivatives in $Z^\al$ involve the scaling vectorfield $\rho L$, the transport
equations (e.g. \eqref{transport equation of mu}, \eqref{transport equation of chi}, \eqref{transport equation of L^i})
can be utilized to derive \eqref{P6.1}.
\end{proof}

\subsection{$L^2$ estimates}\label{Section 5.2}
In this subsection, we shall establish the higher order $L^2$ estimates for some related quantities so that the last two terms of \eqref{energy inequality} can be absorbed by the left hand side, and hence the higher order energy estimates on \eqref{covariant wave equation} can be completed.

At first, we list two lemmas which are similar to Lemma 7.3 in \cite{MY17} and Lemma 12.57 in \cite{Sp16}.
\begin{lemma}\label{Lemma 1 in higher order L^2 estimates}
For any $\Psi\in C^1(D^{\mathfrak t,u})$ which vanishes on $C_0$, one has that for small $\de>0$,
 \begin{equation*}
  \begin{split}
   &\int_{S_{\mathfrak t,u}}\Psi^2\lesssim\de\int_{\Si_\mathfrak t^u}\big[(\Lb\Psi)^2+\mu^2(L\Psi)^2\big]\lesssim\de\big\{\rho^{-2s}E_1[\Psi](\mathfrak t,u)+E_2[\Psi](\mathfrak t,u)\big\},\\
   &\int_{\Si_\mathfrak t^u}\Psi^2\lesssim\de^2\int_{\Si_\mathfrak t^u}\big[(\Lb\Psi)^2+\mu^2(L\Psi)^2\big]\lesssim\de^2\big\{\rho^{-2s}E_1[\Psi](\mathfrak t,u)+E_2[\Psi](\mathfrak t,u)\big\}.
  \end{split}
 \end{equation*}
\end{lemma}

\begin{lemma}\label{Lemma 2 in higher order L^2 estimates}
For any $f\in C(D^{\mathfrak t,u})$, set
 \begin{equation*}
  F(\mathfrak t,u,\vt)=\int_{t_0}^\mathfrak tf(\tau,u,\vt)\d\tau.
 \end{equation*}
Under the assumptions \eqref{BA}, it holds that for small $\de>0$,
 \begin{equation*}
  \|F\|_{L^2(\Si_\mathfrak t^u)}\lesssim\rho(\mathfrak t,u)\int_{t_0}^\mathfrak t\tau^{-1}\|f\|_{L^2(\Si_{\tau}^u)}\d\tau.
 \end{equation*}
\end{lemma}

Based on the preparations above, we are ready to derive the higher order $L^2$ estimates for some related quantities.
\begin{proposition}\label{Proposition higher order L^2 estimates}
Under the assumptions \eqref{BA}, when $\de>0$ is small, it holds that for $|\al|\leq 2N-6$,
 \begin{equation*}
  \begin{split}
   &\mathfrak t\de^l\|\Lies_{Z}^{\al}\chic\|_{L^2(\Si_\mathfrak t^u)}+\de^l\|\Lies_{Z}^{\al}{}^{(R_j)}\pis\|_{L^2(\Si_\mathfrak t^u)}+\de^l\|\Lies_{Z}^{\al}{}^{(R_j)}\pis_L\|_{L^2(\Si_\mathfrak t^u)}\lesssim
   M^p\de^{(1-\ve0)p+\frac{1}{2}}+\Theta_M^1(\mathfrak t,u),\\
   &\de^l\|Z^{\al+1}\check{L}^j\|_{L^2(\Si_\mathfrak t^u)}+\mathfrak t^{-1}\de^l\|Z^{\al+1}v_j\|_{L^2(\Si_\mathfrak t^u)}\lesssim M^p\de^{(1-\ve0)p+\frac{1}{2}}+\Theta_M^1(\mathfrak t,u),\\
   &\de^l\|\Lies_{Z}^{\al+1}\ds x^j\|_{L^2(\Si_\mathfrak t^u)}+\mathfrak t^{-1}\de^l\|\Lies_{Z}^{\al+1}R_j\|_{L^2(\Si_\mathfrak t^u)}\lesssim\de^{\frac{1}{2}}\mathfrak t+\Theta_M^1(\mathfrak t,u),\\
   &\de^l\|Z^{\al+1}\mu\|_{L^2(\Si_\mathfrak t^u)}\lesssim M^p\de^{(1-\ve0)p-\frac{1}{2}}\mathfrak t+\mathfrak t\Theta_M^2(\mathfrak t,u),\\
   &\de^l\|\Lies_{Z}^{\al}{}^{(R_j)}\pis_T\|_{L^2(\Si_\mathfrak t^u)}\lesssim M^p\de^{(1-\ve0)p+\frac{1}{2}}+\Theta_M^1(\mathfrak t,u),\\
   &\de^l\|\Lies_{Z}^{\al}{}^{(T)}\pis\|_{L^2(\Si_\mathfrak t^u)}\lesssim\de^{\frac{1}{2}}+\Theta_M^2(\mathfrak t,u),
  \end{split}
 \end{equation*}
where $l$ is the number of $T$ appearing in the string of $Z$, and
\begin{align*}
&\Theta_M^1(\mathfrak t,u)=M^{p-1}\de^{(1-\ve0)(p-1)}\big[\sqrt{\Et_{1,\leq|\al|+2}(\mathfrak t,u)}+\de\sqrt{\Et_{2,\leq|\al|+2}(\mathfrak t,u)}\big],\\
&\Theta_M^2(\mathfrak t,u)=M^{p-1}\de^{(1-\ve0)(p-1)}\big[\sqrt{\Et_{1,\leq|\al|+2}(\mathfrak t,u)}+\sqrt{\Et_{2,\leq|\al|+2}(\mathfrak t,u)}\big].
\end{align*}
\end{proposition}
\begin{proof}
We will prove this proposition by the induction method with respect to the index $\al$. As in the proof of Proposition \ref{Proposition higher order L^infty estimates}, we first prove the results which only involve the rotational vector fields.

When $\al=0$, in view of Proposition \ref{Proposition higher order L^infty estimates}, the corresponding $L^2$ estimates can be directly obtained by the fact $\|1\|_{L^2(\Si_\mathfrak t^u)}\lesssim\de^{1/2}\mathfrak t$ (similar to Corollary 11.30.3 in \cite{Sp16}). For example, $\|\chic\|_{L^2(\Si_\mathfrak t^u)}\lesssim \|\check\chi\|_{L^{\infty}(\Si_\mathfrak t^u)}\cdot\|1\|_{L^2(\Si_\mathfrak t^u)}\lesssim M^p\de^{(1-\ve0)p+\frac{1}{2}}\mathfrak t^{-1}$.

Assume that Proposition \eqref{Proposition higher order L^2 estimates} holds for the index $\al$ ($|\al|=k\leq 2N-7$) and $Z\in\{R_1, R_2, R_3\}$.
For the treatment on the case $|\al|=k+1$, our concrete strategies are to take
 the $L^2$ norms for the factors equipped with the highest order derivatives of the related quantities,
 meanwhile applying Proposition \ref{Proposition higher order L^infty estimates} to estimate
 the corresponding $L^{\infty}$ coefficients in these terms.

\begin{enumerate}
	\item
\textbf{Bounds of $\Lies_{R_i}^{\al}\chic$ and ${\Lies_{R_i}^{\al}}{}^{(R_j)}\pis_L$}

With the help of induction assumption, then by \eqref{Ri pi} and Lemma \ref{Lemma 1 in higher order L^2 estimates},
\begin{equation}\label{RpiL}
\begin{split}
&\|{\Lies_{R_i}^{\al}}{}^{(R_j)}\pis_L\|_{L^2(\Si_{\mathfrak t}^u)}\\
\lesssim&\mathfrak t\|\Lies_{R_i}^{\leq\al}\chic\|_{L^2(\Si_{\mathfrak t}^u)}+M^p\de^{(1-\ve0)p}\mathfrak t^{-2}\|\Lies_{R_i}^{\leq\al}R_j\|_{L^2(\Si_{\mathfrak t}^u)}+\|R_i^{\leq\al}\check{L}^j\|_{L^2(\Si_{\mathfrak t}^u)}\\
&+M^p\de^{(1-\ve0)p}\mathfrak t^{-1}\|\Lies_{R_i}^{\leq\al}\ds x^j\|_{L^2(\Si_{\mathfrak t}^u)}+M^p\de^{(1-\ve0)p}\mathfrak t^{-(p+1)}\|R_i^{\leq\al}v_j\|_{L^2(\Si_{\mathfrak t}^u)}\\
&+M^{2p-1}\de^{(1-\ve0){(2p-1)}}\mathfrak t^{-p}\|R_i^{\leq\al+1}\varphi\|_{L^2(\Si_{\mathfrak t}^u)}\\
\lesssim&\mathfrak t\|\Lies_{R_i}^{\al}\chic\|_{L^2(\Si_{\mathfrak t}^u)}+M^p\delta^{(1-\ve0)p+\f12}\\
&+M^{p-1}\de^{(1-\ve0)(p-1)}\big[\sqrt{\Et_{1,\leq|\al|+2}(\mathfrak t,u)}+\delta\sqrt{\Et_{2,\leq|\al|+2}(\mathfrak t,u)}\big].
\end{split}
\end{equation}
Thus, it follows from \eqref{commute Lies_Ri^al with L chic} and induction assumption that
\begin{equation}\label{L^2 estimate of Lies_L Lies_Ri^al chic}
\begin{split}
&\|\Lies_L\Lies_{R_i}^{\al}\chic\|_{L^2(\Si_{\mathfrak t}^u)}\\
\lesssim& M^p\de^{(1-\ve0)p}\tau^{-2}\|\Lies_{R_i}^{\al}\chic\|_{L^2(\Si_{\mathfrak t}^u)}+M^{2p}\delta^{(1-\ve0)2p+\f12}\mathfrak t^{-3}\\
&+M^{p-1}\de^{(1-\ve0)(p-1)}\mathfrak t^{-2}\big\{\mathfrak t^{-s}\sqrt{\Et_{1,\leq|\al|+2}(\mathfrak t,u)}+\de\mathfrak t^{-1}\sqrt{\Et_{2,\leq|\al|+2}(\mathfrak t,u)}\big\}.
\end{split}
\end{equation}

Recalling \eqref{weighted transport equation of Lies_{R_i}^{al} chic}, one then has
 \begin{equation}\label{transport inequality of Lies_{Ri}^{al} chic}
  |L(\rho^2|\Lies_{R_i}^{\al}\chic|)|\leq 2|\chic|\cdot\rho^2|\Lies_{R_i}^{\al}\chic|+\rho^2|\Lies_L\Lies_{R_i}^{\al}\chic|.
 \end{equation}
Integrating \eqref{transport inequality of Lies_{Ri}^{al} chic} and utilizing Lemma \ref{Lemma 2 in higher order L^2 estimates}
yield
 \begin{equation*}
  \begin{split}
   \|\rho^2\Lies_{R_i}^{\al}\chic\|_{L^2(\Si_\mathfrak t^u)}
   &\lesssim\|\rho^2\Lies_{R_i}^{\al}\chic\|_{L^2(\Si_{t_0}^u)}+\rho\int_{t_0}^\mathfrak t\tau^{-1}\|L(\rho^2|\Lies_{R_i}^{\al}\chic|)\|_{L^2(\Si_{\tau}^u)}\d\tau.
  \end{split}
 \end{equation*}
Hence,
 \begin{equation}\label{preliminary L^2 estimate of chi small}
  \|\rho\Lies_{R_i}^{\al}\chic\|_{L^2(\Si_\mathfrak t^u)}\lesssim M^p\de^{(1-\ve0)p+\frac{1}{2}}+\int_{t_0}^\mathfrak t\tau\|\Lies_L\Lies_{R_i}^{\al}\chic\|_{L^2(\Si_{\tau}^u)}\d\tau.
 \end{equation}

Substituting \eqref{L^2 estimate of Lies_L Lies_Ri^al chic} into \eqref{preliminary L^2 estimate of chi small}, together with Gronwall's inequality, we have
 \begin{equation*}
  \begin{split}
   \|\rho\Lies_{R_i}^{\al}\chic\|_{L^2(\Si_\mathfrak t^u)}
   &\lesssim M^p\de^{(1-\ve0)p+\frac{1}{2}}+M^{p-1}\de^{(1-\ve0)(p-1)}\big[\sqrt{\Et_{1,\leq|\al|+2}(\mathfrak t,u)}+\de \sqrt{\Et_{2,\leq|\al|+2}(\mathfrak t,u)}\big].
  \end{split}
 \end{equation*}
Therefore,
 \begin{equation}\label{alchi}
  \begin{split}
   \|\Lies_{R_i}^{\al}\chic\|_{L^2(\Si_{\mathfrak t}^u)}\lesssim M^p\de^{(1-\ve0)p+\frac{1}{2}}\mathfrak t^{-1}+M^{p-1}\de^{(1-\ve0)(p-1)}\mathfrak t^{-1}\big[\sqrt{\Et_{1,\leq|\al|+2}(\mathfrak t,u)}+\de\sqrt{\Et_{2,\leq|\al|+2}(\mathfrak t,u)}\big],
  \end{split}
 \end{equation}
 which also gives the estimate of $L^2$ norm for ${\Lies_{R_i}^{\al}}{}^{(R_j)}\pis_L$ by \eqref{RpiL}.

\item\textbf{Bounds of $R_i^{\al+1}\check{L}^j, R_i^{\al+1}v_j,\Lies_{R_i}^{\al+1}\ds x^j,\Lies_{R_i}^{\al+1}R_j$, $\Lies_{R_i}^{\al}{}^{(R_j)}\pis$ and $\Lies_{R_i}^{\al}{}^{(T)}\pis$}

In the proof of Proposition \ref{Proposition higher order L^infty estimates}, one has known that $R_i^{\al+1}\check{L}^j=R_i^{\al}(R_i^A\chic_{AB}\ds^Bx^j)$. Then by \eqref{alchi} and induction argument,
 \begin{equation*}
  \begin{split}
   &\|R_i^{\al+1}\check{L}^j\|_{L^2(\Si_{\mathfrak t}^u)}\\
   \lesssim&\mathfrak t\|\Lies_{R_i}^{\leq\al}\chic\|_{L^2(\Si_{\mathfrak t}^u)}+M^p\de^{(1-\ve0)p}\mathfrak t^{-1}\|\Lies_{R_i}^{\leq\al}\ds x^j\|_{L^2(\Si_{\mathfrak t}^u)}\\
   &+M^p\de^{(1-\ve0)p}\mathfrak t^{-2}\|\Lies_{R_i}^{\leq\al}R_j\|_{L^2(\Si_{\mathfrak t}^u)}+M^p\de^{(1-\ve0)p}\mathfrak t^{-1}\|\Lies_{R_i}^{\leq\al-1}{}^{(R_j)}\pis\|_{L^2(\Si_{\mathfrak t}^u)}\\
   \lesssim& M^p\de^{(1-\ve0)p+\frac{1}{2}}+M^{p-1}\de^{(1-\ve0)(p-1)}\big[\sqrt{\Et_{1,\leq|\al|+2}(\mathfrak t,u)}+\de\sqrt{\Et_{2,\leq|\al|+2}(\mathfrak t,u)}\big].
  \end{split}
 \end{equation*}
 For the other quantities $R_i^{\al+1}v_j, \Lies_{R_i}^{\al+1}\ds x^j, \Lies_{R_i}^{\al+1}R_j$, $\Lies_{R_i}^{\al}{}^{(R_j)}\pis$ and $\Lies_{R_i}^{\al}{}^{(T)}\pis$, their $L^2$ norms over $\Si_{\mathfrak t}^u$ can be similarly obtained
 by the procedure of estimating $\|R_i^{\al+1}\check{L}^j\|_{L^2(\Si_{\mathfrak t}^u)}$.

\item \textbf{Bounds of $R_i^{\al+1}\mu$ and $\Lies_{R_i}^{\al}{}^{(R_j)}\pis_T$}

According to \eqref{Ri pi}, by induction argument and Lemma \ref{Lemma 1 in higher order L^2 estimates},
we obtain
\begin{equation}\label{RpiT}
\begin{split}
&\|\Lies_{R_i}^{\al}{}^{(R_j)}\pis_T\|_{L^2(\Si_{\mathfrak t}^u)}\\
\lesssim&M^p\delta^{(1-\ve0)p}\mathfrak t^{-1}\|R_i^{\leq\al+1}\mu\|_{L^2(\Si_{\mathfrak t}^u)}+\mathfrak t\|\Lies_{R_i}^{\leq\al}\check\chi\|_{L^2(\Si_{\mathfrak t}^u)}+M^p\delta^{(1-\ve0)p}\mathfrak t^{-2}\|\Lies_{R_i}^{\leq\al}R_j\|_{L^2(\Si_{\mathfrak t}^u)}\\
&+M^p\delta^{(1-\ve0)p}\big\{\mathfrak t^{-p}\|\Lies_{R_i}^{\leq\al-1}{}^{(R_j)}\pis\|_{L^2(\Si_{\mathfrak t}^u)}+\mathfrak t^{-1}\|\Lies_{R_i}^{\leq\al}\ds x^j\|_{L^2(\Si_{\mathfrak t}^u)}\big\}+\|R_i^{\leq\al}\check L^j\|_{L^2(\Si_{\mathfrak t}^u)}\\
&+M^p\delta^{(1-\ve0)p-1}\mathfrak t^{-1}\|R_i^{\leq\al}v_i\|_{L^2(\Si_{\mathfrak t}^u)}+M^{p-1}\delta^{(1-\ve0)(p-1)}\mathfrak t^{-1}\|R^{\leq\al+1}\varphi\|_{L^2(\Si_{\mathfrak t}^u)}\\
\lesssim&M^p\delta^{(1-\ve0)p}\mathfrak t^{-1}\|R_i^{\al+1}\mu\|_{L^2(\Si_{\mathfrak t}^u)}+M^p\delta^{(1-\ve0)p+\f12}\\
&+M^{p-1}\delta^{(1-\ve0)(p-1)}\big[\sqrt{\Et_{1,\leq|\al|+2}(\mathfrak t,u)}+\de\sqrt{\Et_{2,\leq|\al|+2}(\mathfrak t,u)}\big].
\end{split}
\end{equation}
This means that once the bound of $\|R_i^{\al+1}\mu\|_{L^2(\Si_{\mathfrak t}^u)}$ is obtained, then the estimate of $\|\Lies_{R_i}^{\al}{}^{(R_j)}\pis_T\|_{L^2(\Si_{\mathfrak t}^u)}$ comes naturally.

Similarly to $\chic$, by Lemma \ref{Lemma 2 in higher order L^2 estimates}, one has
 \begin{equation*}
  \begin{split}
   \|R_i^{\al+1}\mu\|_{L^2(\Si_\mathfrak t^u)}
   &\lesssim\|R_i^{\al+1}\mu\|_{L^2(\Si_{t_0}^u)}+\rho\int_{t_0}^\mathfrak t\tau^{-1}\|LR_i^{\al+1}\mu\|_{L^2(\Si_{\tau}^u)}\d\tau\\
   &\lesssim M^p\de^{(1-\ve0)p-\f12}\mathfrak t+\rho\int_{t_0}^\mathfrak t\tau^{-1}\|LR_i^{\al+1}\mu\|_{L^2(\Si_{\tau}^u)}\d\tau.
  \end{split}
 \end{equation*}
Hence,
 \begin{equation}\label{preliminary L^2 estimate of mu}
  \|\mathfrak t^{-1}R_i^{\al+1}\mu\|_{L^2(\Si_\mathfrak t^u)}\lesssim M^p\de^{(1-\ve0)p-\frac{1}{2}}+\int_{t_0}^\mathfrak t\tau^{-1}\|LR_i^{\al+1}\mu\|_{L^2(\Si_{\tau}^u)}\d\tau.
 \end{equation}

Next, we estimate $\|LR_i^{\al+1}\mu\|_{L^2(\Si_{\tau}^u)}$ in \eqref{preliminary L^2 estimate of mu}.
Recalling \eqref{commute Ri^al with L mu}, \eqref{transport equation of mu} and $[L,R_i]={}^{(R_i)}\pis_{LA}X^A$,
then by \eqref{RpiT}, we have
 \begin{equation}\label{I1}
  \begin{split}
   &\|R_i^{\al+1}L\mu\|_{L^2(\Si_{\tau}^u)}\\
   \lesssim&(M\de^{1-\ve0}\tau^{-1})^{p-1}\|TR_i^{\leq\al+1}\vp\|_{L^2(\Si_{\tau}^u)}
   +M^p\delta^{(1-\ve0)p}\tau^{-(p+1)}\|\Lies_{R_i}^{\al}{}^{(R_j)}\pis_T\|_{L^2(\Si_{\tau}^u)}\\
   &+(M\de^{1-\ve0}\tau^{-1})^{p-1}\|R_i^{\leq\al+1}L\vp\|_{L^2(\Si_{\tau}^u)}
   +(M\de^{1-\ve0}\tau^{-1})^{p-1}\de^{-1}\|R_i^{\leq\al+1}\varphi\|_{L^2(\Si_{\tau}^u)}\\
   &+M^p\de^{(1-\ve0)p}\tau^{-(p+1)}\|R_i^{\leq\al+1}\mu\|_{L^2(\Si_{\tau}^u)}\\
   \lesssim&M^p\de^{(1-\ve0)p}\tau^{-(p+1)}\|R_i^{\al+1}\mu\|_{L^2(\Si_{\tau}^u)}+M^{2p}\delta^{(1-\ve0)2p-\f12}\tau^{-p}\\
   &+M^{p-1}\de^{(1-\ve0)(p-1)}\tau^{-(p-1)}\big[\tau^{-s}\sqrt{\Et_{1,\leq|\al|+2}(\tau,u)}+\sqrt{\Et_{2,\leq|\al|+2}(\tau,u)}\big]
  \end{split}
 \end{equation}
and
 \begin{equation}\label{I2}
  \begin{split}
   &\|\sum_{\be_1+\be_2=\al}R_i^{\be_1}[L,R_i]R_i^{\be_2}\mu\|_{L^2(\Si_{\tau}^u)}\\
   \lesssim&M^p\de^{(1-\ve0)p}\tau^{-2}\|R_i^{\al+1}\mu\|_{L^2(\Si_{\tau}^u)}+M^{2p}\delta^{2p(1-\ve0)-\f12}\tau^{-1}\\
   &+M^{2p-1}\de^{(1-\ve0)(2p-1)-1}\tau^{-1}\big[\sqrt{\Et_{1,\leq|\al|+2}(\tau,u)}+\sqrt{\Et_{2,\leq|\al|+2}(\tau,u)}\big].
  \end{split}
 \end{equation}
Substituting \eqref{I1} and \eqref{I2} into \eqref{preliminary L^2 estimate of mu}, utilizing  Gronwall's inequality, one has
 \begin{equation*}
  \|\mathfrak t^{-1}R_i^{\al+1}\mu\|_{L^2(\Si_\mathfrak t^u)}\lesssim M^p\de^{(1-\ve0)p-\frac{1}{2}}+M^{p-1}\de^{(1-\ve0)(p-1)}\big[\sqrt{\Et_{1,\leq|\al|+2}(\mathfrak t,u)}+\sqrt{\Et_{2,\leq|\al|+2}(\mathfrak t,u)}\big].
 \end{equation*}
Hence, the estimates of $\|R_i^{\al+1}\mu\|_{L^2(\Si_\mathfrak t^u)}$ and $\|\Lies_{R_i}^{\al}{}^{(R_j)}\pis_T\|_{L^2(\Si_{\mathfrak t}^u)}$ are obtained with help of \eqref{RpiT}.
\end{enumerate}

If there are $T$ or $\rho L$ in $Z^{\al+1}$, we could use the identities in
Lemma \ref{3.6}, \eqref{T pi}, \eqref{Ri pi} e.t.c. to get the corresponding $L^2$ bounds.
Since the treatments are analogous to those in the end of proof for Proposition \ref{Proposition higher order L^infty estimates},
we omit the details here.
\end{proof}

\section{Top order $L^2$ estimates for the derivatives of $\slashed\nabla\chi$ and $\slashed\nabla^2\mu$}\label{Section 6}

When we try to close the energy estimate in Section \ref{Section 7} below, it is found that the
top orders of derivatives of $\varphi$, $\chi$ and $\mu$ for the energy estimates are $2N-4$, $2N-5$ and $2N-4$
respectively. However, as shown in Proposition \ref{Proposition higher order L^2 estimates}, the $L^2$ estimates
for the $(2N-5)^{\text{th}}$ order derivatives of $\chi$ and $(2N-4)^{\text{th}}$ order derivatives of $\mu$ can
be controlled by the $(2N-3)^{\text{th}}$ order energy of $\varphi$. So there is a mismatch here.
To overcome this difficulty, we need to
deal with $\chi$ and $\mu$ with the corresponding top order derivatives.
Once the estimate of $\slashed\triangle\mu$ is established, we can use the elliptic estimate to
treat $\slashed\nabla^2\mu$. Before making the elliptic estimates, we shall estimate the
Gaussian curvature $\mathcal G$ of $\slashed g$ firstly.

\begin{lemma}\label{Gc}
Under the assumptions \eqref{BA}, when $\de>0$ is small, the Gaussian curvature $\mathcal{G}$ of $\gs$ in $D^{\mathfrak t,u}$ satisfies
 \begin{equation}\label{estimate of Gaussian curvature}
  \mathcal{G}=\big(1+O(M^{p}\de^{(1-\ve0)p}\rho^{-1})\big)\rho^{-2}.
 \end{equation}
\end{lemma}
\begin{proof}
By (2.29) in \cite{MY17},
 \begin{equation}\label{Gaussian curvature}
  \mathcal{G}=\frac{1}{2}c^{-2}\big((\trgs\chi)^2-|\chi|^2\big).
 \end{equation}
Together with Lemma \ref{Lemma relation of error vectors and trace-free vectors} and \eqref{estimate of chi small}, it follows that
 \begin{equation*}
  \begin{split}
   \mathcal{G}
   &=\frac{1}{2}c^{-2}\big[2\rho^{-2}+(\trgs\chic)^2+2\rho^{-1}\trgs\chic-|\chic|^2\big]\\
   &=\big(1+O(M^{p}\de^{(1-\ve0)p}\rho^{-1})\big)\rho^{-2}.
  \end{split}
 \end{equation*}
\end{proof}

With Lemma \ref{Gc}, we can get the following two lemmas for the elliptic estimates.

\begin{lemma}
For any trace-free symmetric 2-tensor $\xi$ on $S_{\mathfrak t,u}$, it holds that
 \begin{equation}\label{elliptic estimate 1}
  \int_{S_{\mathfrak t,u}}\big(|\nas\xi|^2+2\mathcal{G}|\xi|^2\big)=\int_{S_{\mathfrak t,u}}2|\divgs\xi|^2.
 \end{equation}
For any function $f\in C^2(D^{\mathfrak t,u})$, it holds that
 \begin{equation}\label{elliptic estimate 2}
  \int_{S_{\mathfrak t,u}}\big(|\nas^2f|^2+\mathcal{G}|\ds f|^2\big)=\int_{S_{\mathfrak t,u}}(\Des f)^2.
 \end{equation}
\end{lemma}
\begin{proof}
Taking $\mu=1$ in (18.23) and (18.12) of \cite{Sp16}, then \eqref{elliptic estimate 1} and \eqref{elliptic estimate 2} hold immediately.
\end{proof}

Next, we deal with $\slashed\nabla\chi$ and $\slashed\nabla^2\mu$.

\subsection{Estimates on the derivatives of $\slashed\nabla\chi$}\label{Section 6.1}

Recalling \eqref{RLALB small} and \eqref{RLALB}, we now define
 \begin{equation}\label{RLL small}
  \check{\R}_{LL}=\gs^{AB}\check{\R}_{LALB}=-\frac{1}{2}pc^4\vp_0^{p-1}\Des\vp_0+R_0,
 \end{equation}
where
 \begin{equation}\label{R0}
  R_0=-\frac{3}{2}pc^3\vp_0^{p-1}\ds c\cdot\ds\vp_0-\frac{1}{2}p(p-1)c^4\vp_0^{p-2}|\ds\vp_0|^2.
 \end{equation}
By \eqref{transport equation of chi} and Lemma \ref{Lemma relation of error vectors and trace-free vectors}, then
 \begin{equation}\label{transport equation of trgs chi}
  \begin{split}
   L\trgs\chi=(c^{-1}Lc-2\rho^{-1})\trgs\chi+2\rho^{-2}-|\chic|^2-\check{\R}_{LL}.
  \end{split}
 \end{equation}
Note that the highest order derivative of $\vp_0$ in the right hand side of \eqref{RLL small}
 is $\Des\vp_0$. We will replace $\Des\vp_0$ in $\check{\R}_{LL}$ with $L(\p\vp)+\lot$,
where and below ``$\lot$" stands for the phrase ``lower order terms". Indeed,
 by \eqref{transport equation of Lb vp0} and \eqref{H0}, we have
 \begin{equation}\label{Des vp0}
  \Des\vp_0=\mu^{-1}(L\Lb\vp_0+\trgs\chi T\vp_0-\tilde H_0+F_0).
 \end{equation}
Hence,
 \begin{equation}\label{RLL small replaced}
  \begin{split}
   \check{\R}_{LL}=-LE_{\chi}-e_{\chi}-\frac{1}{2}pc^4\mu^{-1}\vp_0^{p-1}T\vp_0\trgs\chi,
  \end{split}
 \end{equation}
where
 \begin{equation}\label{E chi and e chi}
  \begin{split}
   E_{\chi}&=\frac{1}{2}pc^4\mu^{-1}\vp_0^{p-1}\Lb\vp_0,\\
   e_{\chi}&=-L(\frac{1}{2}pc^4\mu^{-1}\vp_0^{p-1})\Lb\vp_0-\frac{1}{2}pc^4\mu^{-1}\vp_0^{p-1}(\tilde H_0-F_0)-R_0.
  \end{split}
 \end{equation}
Substituting \eqref{RLL small replaced} into \eqref{transport equation of trgs chi} yields
 \begin{equation}\label{modified transport equation of trgs chi}
  L(\trgs\chi-E_{\chi})=(\mu^{-1}L\mu-2\rho^{-1})\trgs\chi+2\rho^{-2}-|\chic|^2+e_{\chi}.
 \end{equation}

Let $E_{\chi}^{\al}=\ds\bar Z^{\al}(\trgs\chi-E_{\chi})$ with $\bar Z\in\{T,R_1,R_2,R_3\}$. Then by the induction argument on \eqref{modified transport equation of trgs chi}, we have
 \begin{equation}\label{higher order modified transport equation of E chi}
  \begin{split}
   \Lies_LE_{\chi}^{\al}=(\mu^{-1}L\mu-2\rho^{-1})E_{\chi}^{\al}+(\mu^{-1}L\mu-2\rho^{-1})\ds\bar Z^{\al}E_{\chi}-\ds\bar Z^{\al}(|\chic|^2)+e_{\chi}^{\al},
  \end{split}
 \end{equation}
where
 \begin{equation}\label{higher order e chi}
  \begin{split}
   e_{\chi}^{\al}&=\Lies_{\bar Z}^{\al}e_{\chi}^0+\sum_{\substack{\be_1+\be_2=\al,\\|\be_1|\geq 1}}\bar Z^{\be_1}(\mu^{-1}L\mu-2\rho^{-1})\ds \bar Z^{\be_2}\trgs\chi+\sum_{\be_1+\be_2=\al-1}\Lies_{\bar Z}^{\be_1}\Lies_{[L,\bar Z]}E_{\chi}^{\be_2},\\
   e_{\chi}^0&=\ds(\mu^{-1}L\mu)\trgs\chi+\ds e_{\chi}.
  \end{split}
 \end{equation}
For any 1-form $\xi$, one has
 \begin{equation}\label{weighted transport equation of 1-form}
  \begin{split}
   L(\rho^2|\xi|^2)=-2\rho^2\chic^{AB}\xi_A\xi_B+2\rho^2\gs^{AB}(\Lies_L\xi_A)\xi_B.
  \end{split}
 \end{equation}
By taking $\xi=\rho^2E_{\chi}^{\al}$ in \eqref{weighted transport equation of 1-form} and
using \eqref{higher order modified transport equation of E chi}, then
 \begin{equation*}
  \begin{split}
   L(\rho^6|E_{\chi}^{\al}|^2)
   =&2\rho^6\big[(\mu^{-1}L\mu)E_{\chi}^{\al}+(\mu^{-1}L\mu-2\rho^{-1})\ds\bar Z^{\al}E_{\chi}\\
   &-\ds\bar Z^{\al}(|\chic|^2)+e_{\chi}^{\al}\big]\cdot E_{\chi}^{\al}-2\rho^{6}\check\chi^{AB}E_{\chi A}^{\al}E_{\chi B}^{\al}.
  \end{split}
 \end{equation*}
Hence,
 \begin{equation}\label{transport inequality of E chi}
  L(\rho^3|E_{\chi}^{\al}|)\lesssim\rho^3\big[M^p\delta^{(1-\ve0)p-1}\rho^{-2}|E_{\chi}^{\al}|+\rho^{-1}|\ds\bar Z^{\al}E_{\chi}|+|\ds\bar Z^{\al}(|\chic|^2)|+|e_{\chi}^{\al}|\big].
 \end{equation}
Using Lemma \ref{Lemma 2 in higher order L^2 estimates} and \eqref{transport inequality of E chi} for $\rho^3|E_{\chi}^{\al}|$ to get
 \begin{equation}\label{preliminary L^2 estimate of E chi}
  \begin{split}
   \de^l\|\rho^2E_{\chi}^{\al}\|_{L^2(\Si_{\mathfrak t}^u)}
   &\lesssim M^p\de^{(1-\ve0)p-\frac{1}{2}}+\de^l\int_{t_0}^\mathfrak t\big\{M^p\de^{(1-\ve0)p-1}\tau^{-2}\|\rho^2E_{\chi}^{\al}\|_{L^2(\Si_{\tau}^u)}\\
   &\quad+\tau\|\ds\bar Z^{\al}E_{\chi}\|_{L^2(\Si_{\tau}^u)}+\tau^2\|\ds\bar Z^{\al}(|\chic|^2)\|_{L^2(\Si_{\tau}^u)}+\tau^2\|e_{\chi}^{\al}\|_{L^2(\Si_{\tau}^u)}\big\}\d\tau.
  \end{split}
 \end{equation}

Next, we estimate the terms in integrand of \eqref{preliminary L^2 estimate of E chi} one by one.

\vskip 0.2 true cm

\textbf{(1-a) Estimate of $\ds\bar Z^{\al}E_{\chi}$}

\vskip 0.1 true cm

Due to $E_{\chi}=\frac{1}{2}pc^4\mu^{-1}\vp_0^{p-1}\Lb\vp_0$, then
it follows from \eqref{BA}, Proposition \ref{Proposition higher order L^2 estimates}
and Lemma \ref{Lemma 1 in higher order L^2 estimates} that
 \begin{equation}\label{top L^2 estimate for E chi}
  \begin{split}
   &\quad\de^l\|\ds\bar Z^{\al}E_{\chi}\|_{L^2(\Si_{\tau}^u)}\\
   \lesssim&(M\de^{1-\ve0}\tau^{-1})^{p-1}\Big\{\tau^{-1}\de^l\|T Z^{\leq\al+1}\vp\|_{L^2(\Si_{\tau}^u)}+\tau^{-1}\de^l\|\ds Z^{\leq\al+1}\vp\|_{L^2(\Si_{\tau}^u)}\\
   &+\delta^{-1}\tau^{-1}\delta^l\|Z^{\leq\al+1}\vp\|_{L^2(\Si_{\tau}^u)}+M\de^{-\ve0}\tau^{-2}\de^l\|Z^{\leq\al+1}\mu\|_{L^2(\Si_{\tau}^u)}\\
   &+M\delta^{1-\ve0}\tau^{-3}\delta^l\|\Lies_Z^{\leq\al}{}^{(R_i)}\pis_T\|_{L^2(\Si_{\tau}^u)}\Big\}\\
   \lesssim& M^{2p}\de^{(1-\ve0)2p-\frac{3}{2}}\tau^{-p}+M^{p-1}\de^{(1-\ve0)(p-1)}\tau^{-p}\big\{\sqrt{\Et_{1,\leq|\al|+2}(\tau,u)}+\sqrt{\Et_{2,\leq|\al|+2}(\tau,u)}\big\}.
  \end{split}
 \end{equation}

 \vskip 0.2 true cm

\textbf{(1-b) Estimate of $\ds\bar Z^{\al}(|\chic|^2)$}

\vskip 0.1 true cm

According to \eqref{elliptic equation of chi}, we have
 \begin{equation}\label{elliptic equation of chi small}
  (\divgs\chic)_A=\ds_A\trgs\chi+I_A,
 \end{equation}
where
 \begin{equation}\label{I}
  I_A=c^{-1}(\ds^Bc)\chi_{AB}-c^{-1}(\ds_Ac)\trgs\chi.
 \end{equation}
By commuting $\Lies_{\bar Z}^{\al}$ with $\divgs$, one has from \eqref{commutator [nas,Lies]} that
 \begin{equation*}
  \begin{split}
   &\Lies_{\bar Z}^{\al}(\divgs\chic)_A=\Lies_{\bar Z}^{\al}(\gs^{BC}\nas_C\chic_{AB})\\
   =&(\divgs\Lies_{\bar Z}^{\al}\chic)_A+\sum_{\be_1+\be_2=\al-1}\Lies_{\bar Z}^{\be_1}\big[\gs^{BC}(\check{\nas}_C{}^{(\bar Z)}\pis_A^D)\Lies_Z^{\be_2}\chic_{BD}+\gs^{BC}(\check{\nas}_C{}^{(\bar Z)}\pis_B^D)\Lies_{\bar Z}^{\be_2}\chic_{AD}+{}^{(\bar Z)}\pis^{BC}\nas_C\Lies_{\bar Z}^{\be_2}\chic_{AB}\big].
  \end{split}
 \end{equation*}
Then \eqref{elliptic equation of chi small} implies
 \begin{equation}\label{higher order elliptic equation of chi small}
  (\divgs\Lies_{\bar Z}^{\al}\chic)_A=\ds_A\bar Z^{\al}\trgs\chi+{I^{\al}}_A,
 \end{equation}
where
 \begin{equation}\label{higher order I}
  {I^{\al}}_A=\Lies_{\bar Z}^{\al}I_A-\sum_{\be_1+\be_2=\al-1}\Lies_{\bar Z}^{\be_1}\big[\gs^{BC}(\check{\nas}_C{}^{(\bar Z)}\pis_A^D)\Lies_{\bar Z}^{\be_2}\chic_{BD}+\gs^{BC}(\check{\nas}_C{}^{(\bar Z)}\pis_B^D)\Lies_{\bar Z}^{\be_2}\chic_{AD}+{}^{(\bar Z)}\pis^{BC}\nas_C\Lies_{\bar Z}^{\be_2}\chic_{AB}\big]
 \end{equation}
with
 \begin{equation}\label{L^2 estimate of higher order I}
  \de^l\|I^{\al}\|_{L^2(\Si_{\tau}^u)}\lesssim\delta \tau^{-2}\de^l\|\Lies_{\bar Z}^{\leq\al}\chic\|_{L^2(\Si_{\tau}^u)}+M^p\de^{(1-\ve0)p}\tau^{-3}\big\{\de^l\|\Lies_{\bar Z}^{\leq\al}{}^{(R_j)}\pis\|_{L^2(\Si_{\tau}^u)}+\delta\delta^l\|\Lies_{\bar Z}^{\leq\al}{}^{(T)}\pis\|_{L^2(\Si_{\tau}^u)}\big\}.
 \end{equation}

It follows from \eqref{estimate of Gaussian curvature}, \eqref{elliptic estimate 1} and \eqref{higher order elliptic equation of chi small} that
 \begin{equation}\label{L^2 estimate of chih}
  \begin{split}
   \de^l\|\nas\Lies_{\bar Z}^{\al}\check\chi\|_{L^2(\Si_{\tau}^u)}&\lesssim\de^l\|\divgs\Lies_{\bar Z}^{\al}\check\chi\|_{L^2(\Si_{\tau}^u)}\lesssim\de^l\|\ds_A{\bar Z}^{\al}\trgs\chi\|_{L^2(\Si_{\tau}^u)}+\de^l\|I^{\al}\|_{L^2(\Si_{\tau}^u)},
  \end{split}
 \end{equation}
and then by \eqref{L^2 estimate of chih} and \eqref{L^2 estimate of higher order I},
 \begin{equation*}
  \begin{split}
   &\de^l\|\ds\bar Z^{\al}(|\chic|^2)\|_{L^2(\Si_{\tau}^u)}\\
   \lesssim& M^p\de^{(1-\ve0)p}\tau^{-2}\de^l\|\nas\Lies_{\bar Z}^{\al}\chic\|_{L^2(\Si_{\tau}^u)}+M^p\delta^{(1-\ve0)p}\tau^{-3}\delta^l\|\Lies_{\bar Z}^{\leq\al}\check\chi\|_{L^2(\Si_{\tau}^u)}\\
   &+M^{2p}\de^{(1-\ve0)2p}\tau^{-5}\big\{\de^l\|\Lies_{\bar Z}^{\leq\al}{}^{(R_j)}\pis\|_{L^2(\Si_{\tau}^u)}
   +\delta^{l+1}\|\Lies_{\bar Z}^{\leq\al}{}^{(T)}\pis\|_{L^2(\Si_{\tau}^u)}\big\}\\
   \lesssim &M^p\de^{(1-\ve0)p}\tau^{-2}\big[\de^l\|\ds\bar Z^{\al}\trgs\chi\|_{L^2(\Si_{\tau}^u)}+\de^l\|I^{\al}\|_{L^2(\Si_{\tau}^u)}+\tau^{-1}\delta^l\|\Lies_{\bar Z}^{\leq\al}\check\chi\|_{L^2(\Si_{\tau}^u)}\big]\\
   &+M^{2p}\de^{(1-\ve0)2p}\tau^{-5}\big\{\de^l\|\Lies_{\bar Z}^{\leq\al}{}^{(R_j)}\pis\|_{L^2(\Si_{\tau}^u)}
   +\delta^{l+1}\|\Lies_{\bar Z}^{\leq\al}{}^{(T)}\pis\|_{L^2(\Si_{\tau}^u)}\big\}\\
   \lesssim& M^p\de^{(1-\ve0)p}\tau^{-2}\big\{\de^l\|E_{\chi}^{\al}\|_{L^2(\Si_{\tau}^u)}+\de^l\|\ds\bar Z^{\al}E_{\chi}\|_{L^2(\Si_{\tau}^u)}+\tau^{-1}\de^l\|\Lies_{\bar Z}^{\leq\al}\chic\|_{L^2(\Si_{\tau}^u)}\big\}\\
   &+M^{2p}\de^{(1-\ve0)2p}\tau^{-5}\big\{\de^l\|\Lies_{\bar Z}^{\leq\al}{}^{(R_j)}\pis\|_{L^2(\Si_{\tau}^u)}+\delta^{l+1}\|\Lies_{\bar Z}^{\leq\al}{}^{(T)}\pis\|_{L^2(\Si_{\tau}^u)}\big\}.
  \end{split}
 \end{equation*}
 Therefore, using Proposition \ref{Proposition higher order L^2 estimates} and \eqref{top L^2 estimate for E chi} to get
 \begin{equation}\label{top L^2 estimate for chi small}
  \begin{split}
   &\de^l\|\ds\bar Z^{\al}(|\chic|^2)\|_{L^2(\Si_{\tau}^u)}\\
   \lesssim& M^p\de^{(1-\ve0)p}\tau^{-2}\de^l\|E_{\chi}^{\al}\|_{L^2(\Si_{\tau}^u)}+M^{2p}\de^{(1-\ve0)p+\frac{1}{2}}\tau^{-4}\\
   &+M^{2p-1}\de^{(1-\ve0)(2p-1)}\tau^{-4}\big\{\sqrt{\Et_{1,\leq|\al|+2}(\tau,u)}+\sqrt{\Et_{2,\leq|\al|+2}(\tau,u)}\big\}.
  \end{split}
 \end{equation}

\vskip 0.2 true cm

\textbf{(1-c) Estimate of $e_{\chi}^{\al}$}

\vskip 0.1 true cm

By \eqref{E chi and e chi}, \eqref{H'0}, \eqref{F la}, \eqref{transport equation of chi} and \eqref{R0}, one has
 \begin{equation}\label{L^2 estimate of e chi}
  \begin{split}
   &\de^l\|\Lies_{\bar Z}^{\al}\ds e_{\chi}\|_{L^2(\Si_{\tau}^u)}\\
   \lesssim&M^p\de^{(1-\ve0)p-1}\tau^{-(p+2)}\de^l\|Z^{\leq\al+1}\mu\|_{L^2(\Si_{\tau}^u)}\\
   &+\delta^{-1}\tau^{-1}(M\de^{1-\ve0}\tau^{-1})^{p-1}(\de^l\|\ds Z^{\leq\al+1}\vp\|_{L^2(\Si_{\tau}^u)}+\tau^{-1}\de^l\|Z^{\leq\al+1}\vp\|_{L^2(\Si_{\tau}^u)})\\
   \lesssim& M^p\de^{(1-\ve0)p-1}\tau^{-(p+2)}\de^l\|Z^{\leq\al+1}\mu\|_{L^2(\Si_{\tau}^u)}\\
   &+M^{p-1}\de^{(1-\ve0)(p-1)-1}\tau^{-p}\big\{\tau^{-s}\sqrt{E_{1,\leq|\al|+2}(\tau,u)}+\delta\tau^{-1}\sqrt{E_{2,\leq|\al|+2}(\tau,u)}\big\}.
  \end{split}
 \end{equation}

For the third term in \eqref{higher order e chi},
it follows from ${}^{(T)}\pis_L=-c^2\ds(c^{-2}\mu)$, \eqref{higher order modified transport equation of E chi} and the induction argument that
 \begin{equation}\label{L^2 estimate of the third term in higher order e chi}
  \begin{split}
   &\de^l\|\sum_{\be_1+\be_2=\al-1}\Lies_{\bar Z}^{\be_1}\Lies_{[L,\bar Z]}E_{\chi}^{\be_2}\|_{L^2(\Si_{\tau}^u)}\\
   \lesssim&\de^l\|\sum_{\be_1+\be_2=\al-1}\Lies_{\bar Z}^{\be_1}({}^{(\bar Z)}\pis_L\nas E_{\chi}^{\be_2}+\nas{}^{(\bar Z)}\pis_LE_{\chi}^{\be_2})\|_{L^2(\Si_{\tau}^u)}\\
   \lesssim& M^p\de^{(1-\ve0)p}\tau^{-2}\big\{\de^l\|E_{\chi}^{\al}\|_{L^2(\Si_{\tau}^u)}+\tau^{-1}\de^l\|\bar Z^{\leq\al}\trgs\check\chi\|_{L^2(\Si_{\tau}^u)}+\tau^{-1}\de^l\|\ds\bar Z^{\leq\al-1}E_\chi\|_{L^2(\Si_{\tau}^u)}\big\}\\
   &+M^p\de^{(1-\ve0)p}\tau^{-5}\big\{\de^l\|{\bar Z}^{\leq\al+1}\mu\|_{L^2(\Si_{\tau}^u)}+\delta^l\|\bar Z^{\leq\al+1}c\|_{L^2(\Si_{\tau}^u)}\big\}\\
   &+M^p\de^{(1-\ve0)p-1}\tau^{-4}\de^l\|\Lies_{\bar Z}^{\leq\al}{}^{(R_j)}\pis_L\|_{L^2(\Si_{\tau}^u)}.
  \end{split}
 \end{equation}
We point out that the estimates of the other terms in \eqref{higher order e chi} are easily to be
handled. Then after substituting \eqref{L^2 estimate of e chi} and \eqref{L^2 estimate of the third term in higher order e chi}
into \eqref{higher order e chi}, one has
 \begin{equation*}
  \begin{split}
   &\de^l\|e_{\chi}^{\al}\|_{L^2(\Si_{\tau}^u)}\\
   \lesssim& M^p\de^{(1-\ve0)p}\tau^{-2}\de^l\|E_{\chi}^{\al}\|_{L^2(\Si_{\tau}^u)}+\tau^{-3}\de^l\|\Lies_Z^{\leq\al}\chic\|_{L^2(\Si_{\tau}^u)}+M^{p-1}\de^{(1-\ve0)(p-1)-1}\tau^{-p}\delta^l\|LZ^{\al+1}\vp\|_{L^2(\Si_{\tau}^u)}\\
   &+M^{p-1}\de^{(1-\ve0)(p-1)}\tau^{-(p+1)}\big\{\tau^{-s}\sqrt{E_{1,\leq|\al|+2}(\tau,u)}+\sqrt{E_{2,\leq|\al|+2}(\tau,u)}\big\}\\
   &+M^p\delta^{(1-\ve0)p}\tau^{-5}\big\{\delta^l\|\Lies_Z^{\leq\al}{}^{(R_i)}\pis\|_{L^2(\Si_{\tau}^u)}
   +\delta^{l+1}\|\Lies_Z^{\leq\al}{}^{(T)}\pis\|_{L^2(\Si_{\tau}^u)}+\tau^2\de^l\|\ds\bar Z^{\leq\al-1}E_\chi\|_{L^2(\Si_{\tau}^u)}\big\}\\
   &+M^p\de^{(1-\ve0)p-1}\tau^{-4}\big\{\de^l\|Z^{\leq\al+1}\mu\|_{L^2(\Si_{\tau}^u)}
   +\de^l\|\Lies_Z^{\leq\al}{}^{(R_j)}\pis_L\|_{L^2(\Si_{\tau}^u)}\big\}.
  \end{split}
 \end{equation*}
Therefore, by Proposition \ref{Proposition higher order L^2 estimates} and \eqref{top L^2 estimate for E chi}, we arrive at
 \begin{equation}\label{top L^2 estimate for higher order e chi}
  \begin{split}
   &\int_{t_0}^\mathfrak t\tau^2\de^l\|e_{\chi}^{\al}\|_{L^2(\Si_{\tau}^u)}d\tau\\
   \lesssim& \int_{t_0}^\mathfrak tM^p\de^{(1-\ve0)p}\de^l\|E_{\chi}^{\al}\|_{L^2(\Si_{\tau}^u)}d\tau+M^{p-1}\delta^{(1-\ve0)(p-1)-1}\mathfrak t^{\f12-s}\sqrt{\int_0^uF_{1,|\al|+2}(\mathfrak t,u')du'}\\
   &+M^p\de^{(1-\ve0)p-\frac{1}{2}}\ln\mathfrak t+M^{p-1}\de^{(1-\ve0)(p-1)}\ln\mathfrak t\big\{\sqrt{\Et_{1,\leq|\al|+2}(\mathfrak t,u)}+\sqrt{\Et_{2,\leq|\al|+2}(\mathfrak t,u)}\big\}.
  \end{split}
 \end{equation}

Substituting \eqref{top L^2 estimate for E chi}, \eqref{top L^2 estimate for chi small}
and \eqref{top L^2 estimate for higher order e chi} into \eqref{preliminary L^2 estimate of E chi},
and applying Gronwall's inequality, one obtains
 \begin{equation*}
  \begin{split}
   \de^l\|\rho^2E_{\chi}^{\al}\|_{L^2(\Si_\mathfrak t^u)}
   \lesssim& M^p\de^{(1-\ve0)p-\frac{1}{2}}\ln \mathfrak t+M^{p-1}\delta^{(1-\ve0)(p-1)-1}\mathfrak t^{\f12-s}\sqrt{\int_0^u\tilde F_{1,\leq|\al|+2}(\mathfrak t,u')du'}\\
   &+M^{p-1}\de^{(1-\ve0)(p-1)}\ln\mathfrak t\big\{\sqrt{\Et_{1,\leq|\al|+2}(\mathfrak t,u)}+\sqrt{\Et_{2,\leq|\al|+2}(\mathfrak t,u)}\big\}.
  \end{split}
 \end{equation*}
Due to $\ds\bar Z^{\al}\trgs\chi=E_{\chi}^{\al}+\ds\bar Z^{\al}E_{\chi}$, then  by \eqref{top L^2 estimate for E chi} and \eqref{L^2 estimate of chih},
 \begin{equation}\label{top order L^2 estimate of bchi}
  \begin{split}
   &\de^l\|\ds\bar Z^{\al}\trgs\chi\|_{L^2(\Si_\mathfrak t^u)}+\de^l\|\slashed\nabla\Lies_{\bar Z}^{\al}\chic\|_{L^2(\Si_\mathfrak t^u)}\\
   \lesssim& M^p\de^{(1-\ve0)p-\frac{1}{2}}\mathfrak t^{-2}\ln\mathfrak t+M^{p-1}\delta^{(1-\ve0)(p-1)-1}\mathfrak t^{-\f32-s}\sqrt{\int_0^u\tilde F_{1,\leq|\al|+2}(\mathfrak t,u')du'}\\
   &+M^{p-1}\de^{(1-\ve0)(p-1)}\mathfrak t^{-2}\ln\mathfrak t\big\{\sqrt{\Et_{1,\leq|\al|+2}(\mathfrak t,u)}+\sqrt{\Et_{2,\leq|\al|+2}(\mathfrak t,u)}\big\}.
  \end{split}
 \end{equation}

If there exists at least one $\rho L$ in  $\bar Z^{\al}$, then
 by making use of \eqref{transport equation of chi small}, the commutators of vector fields and Proposition \ref{Proposition higher order L^2 estimates}, we have
\begin{equation}\label{top order L^2 estimate of Lchi}
\begin{split}
\de^l\|\slashed\nabla\Lies_{Z}^{\al}\chic\|_{L^2(\Si_\mathfrak t^u)}\lesssim& M^p\de^{(1-\ve0)p+\frac{3}{2}}\mathfrak t^{-3}\\
&+M^{p-1}\de^{(1-\ve0)(p-1)}\big\{\mathfrak t^{-2-s}\sqrt{\Et_{1,\leq|\al|+2}(\mathfrak t,u)}+\delta\mathfrak t^{-3}\sqrt{\Et_{2,\leq|\al|+2}(\mathfrak t,u)}\big\}.
\end{split}
\end{equation}

Therefore, for any vector filed $Z\in\{\rho L, T, R_1, R_2, R_3\}$, \eqref{top order L^2 estimate of bchi}
and \eqref{top order L^2 estimate of Lchi} give that
\begin{equation}\label{top order L^2 estimate of chi}
\begin{split}
&\de^l\|\slashed\nabla\Lies_{Z}^{\al}\chic\|_{L^2(\Si_\mathfrak t^u)}+\de^l\|\ds Z^{\al}\trgs\chi\|_{L^2(\Si_\mathfrak t^u)}\\
\lesssim& M^p\de^{(1-\ve0)p-\frac{1}{2}}\mathfrak t^{-2}\ln\mathfrak t+M^{p-1}\delta^{(1-\ve0)(p-1)-1}\mathfrak t^{-\f32-s}\sqrt{\int_0^u\tilde F_{1,\leq|\al|+2}(\mathfrak t,u')du'}\\
&+M^{p-1}\de^{(1-\ve0)(p-1)}\mathfrak t^{-2}\ln\mathfrak t\big\{\sqrt{\Et_{1,\leq|\al|+2}(\mathfrak t,u)}+\sqrt{\Et_{2,\leq|\al|+2}(\mathfrak t,u)}\big\}.
\end{split}
\end{equation}

\subsection{Estimates on the derivatives of $\slashed\nabla^2\mu$}\label{Section 6.2}

Similarly to $\trgs\chi$, we use the transport equation \eqref{transport equation of mu} to estimate $\Des\mu$.
By Lemma \ref{Lemma commutator [Des,Z]}, one has
 \begin{equation}\label{transport equation of Des mu}
  \begin{split}
   L\Des\mu
   =&\Des L\mu+[L,\Des]\mu\\
   =&-\boxed{cT\Des c}-2c{\big(\ds_A(c^{-2}\mu)\check\chi^{AB}-\f12\ds^B(c^{-2}\mu)\trgs\check\chi\big)}\ds_Bc-2c^{-1}\mu\chic^{AB}\nas_{AB}^2c\\
   &-2\rho^{-1}c^{-1}\mu\Des c-2\ds c\cdot\ds Tc-(\Des c)Tc+\underline{L(\mu\Des\ln c)}-(L\mu)\Des\ln c\\
   &+2\mu\chic^{AB}\nas_{AB}^2\ln c+2\rho^{-1}\mu\Des\ln c+2\ds(c^{-1}Lc)\cdot\ds\mu+c^{-1}Lc\uwave{\Des\mu}\\
   &-(\ds_A\trgs\check\chi)\ds^A\mu-2\chic^{AB}\uwave{\nas_{AB}^2\mu}-2\rho^{-1}\uwave{\Des\mu}-2I_A\ds^A\mu,
  \end{split}
 \end{equation}
where $I_A$ has been given in \eqref{I}. Observe that the term with underline can be removed to the left hand side
of \eqref{transport equation of Des mu}, while the terms with wavy line will be treated by the elliptic estimate
and Gronwall's inequality. The strategy to treat the boxed term which contains the third order derivative of the solution is to
transfer it into such a form $L(\p^{\leq2}\vp)+\lot$. Indeed, by \eqref{transport equation of chi along T}, we have
 \begin{equation}\label{transport equation of trgs chi along T}
   T\trgs\chi=T(\gs^{AB}\chi_{AB})=\Des\mu+J
 \end{equation}
with
 \begin{equation}\label{J}
  J=c^{-2}\mu|\chi|^2-c^{-1}L(c^{-1}\mu)\trgs\chi-\nas^A(c^{-1}\mu\ds_Ac)+c^{-2}\mu|\ds c|^2-c^{-1}(\ds_A c)\ds^A\mu.
 \end{equation}
In addition, \eqref{Des vp0} gives that
 \begin{equation}\label{Des c}
  \begin{split}
   \Des c=&\nas^A(-\frac{1}{2}pc^3\vp_0^{p-1}\ds_A\vp_0)\\
         =&-\frac{1}{2}pc^3\mu^{-1}\vp_0^{p-1}(L\Lb\vp_0)-\frac{1}{2}pc^3\mu^{-1}\vp_0^{p-1}(T\vp_0)\trgs\chi\\
         &+\frac{1}{2}pc^3\mu^{-1}\vp_0^{p-1}(H'_0-F_0)+\ds^A(-\frac{1}{2}pc^3\vp_0^{p-1})\ds_A\vp_0.
  \end{split}
 \end{equation}
Combining \eqref{Des c} and \eqref{transport equation of trgs chi along T} yields
 \begin{equation}\label{computation 3}
  \begin{split}
   &-cT\Des c\\
   =&\underline{L(\frac{1}{2}pc^4\mu^{-1}\vp_0^{p-1}T\Lb\vp_0)}-c\mu^{-1}(Tc)\uwave{\Des\mu}-L(\frac{1}{2}pc^4\mu^{-1}\vp_0^{p-1})T\Lb\vp_0\\
   &-\frac{1}{2}pc^4\mu^{-1}\vp_0^{p-1}{}^{(T)}\pis_L\cdot\ds\Lb\vp_0+\frac{1}{2}pc^4\mu^{-1}\vp_0^{p-1}(T\vp_0)J\\
   &+cT(\frac{1}{2}pc^3\mu^{-1}\vp_0^{p-1})L\Lb\vp_0+cT(\frac{1}{2}pc^3\mu^{-1}\vp_0^{p-1}T\vp_0)\trgs\chi\\
   &-cT\big[\frac{1}{2}pc^3\mu^{-1}\vp_0^{p-1}(H'_0-F_0)+\ds(-\frac{1}{2}pc^3\vp_0^{p-1})\cdot\ds\vp_0\big].
  \end{split}
 \end{equation}

Substituting \eqref{computation 3} into \eqref{transport equation of Des mu}, by direct computation
then one has
 \begin{equation}\label{modified transport equation of Des mu}
  L(\Des\mu-E_{\mu})=(\mu^{-1}L\mu-2\rho^{-1})\Des\mu-2\chic^{AB}\nas_{AB}^2\mu-(\ds_A\trgs\chi)\ds^A\mu+e_{\mu},
 \end{equation}
where
 \begin{equation}\label{E mu and e mu}
  \begin{split}
   E_{\mu}=&\frac{1}{2}pc^4\mu^{-1}\vp_0^{p-1}T\Lb\vp_0+\mu\Des\ln c,\\
   e_{\mu}=&-L(\frac{1}{2}pc^4\mu^{-1}\vp_0^{p-1})T\Lb\vp_0-\frac{1}{2}pc^4\mu^{-1}\vp_0^{p-1}{}^{(T)}\pis_L\cdot\ds\Lb\vp_0+\frac{1}{2}pc^4\mu^{-1}\vp_0^{p-1}(T\vp_0) J\\
          &+cT(\frac{1}{2}pc^3\mu^{-1}\vp_0^{p-1})L\Lb\vp_0+cT(\frac{1}{2}pc^3\mu^{-1}\vp_0^{p-1}T\vp_0)\trgs\chi+2\ds(c^{-1}Lc)\cdot\ds\mu\\
          &-cT\big[\frac{1}{2}pc^3\mu^{-1}\vp_0^{p-1}(H'_0-F_0)+\ds(-\frac{1}{2}pc^3\vp_0^{p-1})\cdot\ds\vp_0\big]+2\rho^{-1}\mu\Des\ln c\\
          &-2c\ds_Bc\big[\ds_A(c^{-2}\mu)\check\chi^{AB}-\frac{1}{2}\ds^B(c^{-2}\mu)\trgs\check\chi\big]-(L\mu)\Des\ln c+2\mu\chic^{AB}\nas_{AB}^2\ln c\\
          &-2c^{-1}\mu\chic^{AB}\nas_{AB}^2c-2\rho^{-1}c^{-1}\mu\Des c-2\ds c\cdot\ds Tc-\Des c\cdot Tc-2\ds\mu\cdot I.
  \end{split}
 \end{equation}
It is easy to see that $e_{\mu}$ is composed of the terms which contain the factors $\p^{\leq 2}\vp$, $\p^{\leq 1}\mu$ and $\check\chi$.

Let $E_{\mu}^{\al}=\bar Z^{\al}(\Des\mu-E_{\mu})$ with $Z\in\{T,R_1,R_2,R_3\}$. Then by the
analogous induction argument on \eqref{modified transport equation of Des mu}, we have
 \begin{equation}\label{higher order modified transport equation of E mu}
  \begin{split}
   LE_{\mu}^{\al}=(\mu^{-1}L\mu-2\rho^{-1})E_{\mu}^{\al}+[L,\bar Z]E_{\mu}^{\al-1}-2\chic^{AB}\Lies_{\bar Z}^{\al}\nas_{AB}^2\mu-\ds\mu\cdot\ds {\bar Z}^{\al}\trgs\chi+e_{\mu}^{\al},
  \end{split}
 \end{equation}
where
 \begin{equation}\label{higher order e mu}
  \begin{split}
   e_{\mu}^{\al}
   =&\bar Z^{\al}e_{\mu}+(\mu^{-1}L\mu-2\rho^{-1})\bar Z^{\al}E_{\mu}\\
   &+\sum_{\substack{\be_1+\be_2=\al,\\|\beta_1|\geq1}}\bar Z^{\be_1}(\mu^{-1}L\mu-2\rho^{-1})\bar Z^{\be_2}\Des\mu+\sum_{\substack{\be_1+\be_2=\al,\\|\beta_1|\geq1}}\Lies_{\bar Z}^{\be_1}\chic\cdot\Lies_{\bar Z}^{\be_2}\nas^2\mu\\
   &+\sum_{\substack{\be_1+\be_2=\al,\\|\beta_1|\geq1}}(\Lies_{\bar Z}^{\beta_1}\ds^A\mu)\ds_A\bar Z^{\be_2}\trgs\chi+\sum_{\substack{\be_1+\be_2=\al-1,\\|\beta_1|\geq1}}{\bar Z}^{\be_1}[L,\bar Z]E_{\mu}^{\be_2}.
  \end{split}
 \end{equation}
Since
 \begin{equation*}
  \begin{split}
   L(\rho^2E_{\mu}^{\al})=\rho^2\big\{\mu^{-1}L\mu\cdot E_{\mu}^{\al}+[L,\bar Z]E_{\mu}^{\al-1}-2\chic^{AB}\Lies_{\bar Z}^{\al}\nas_{AB}^2\mu-\ds\mu\cdot\ds\bar Z^{\al}\trgs\chi+e_{\mu}^{\al}\big\},
  \end{split}
 \end{equation*}
it follows from Lemma \ref{Lemma 2 in higher order L^2 estimates} for $\rho^2E_\mu^\al$ and Proposition \ref{Proposition higher order L^infty estimates} that
 \begin{equation}\label{preliminary L^2 estimate of E mu}
  \begin{split}
   &\de^l\|\rho E_{\mu}^{\al}\|_{L^2(\Si_{\mathfrak t}^u)}\\
   \lesssim &M^p\de^{(1-\ve0)p-\frac{3}{2}}+\de^l\int_{t_0}^\mathfrak t\big[M^p\de^{(1-\ve0)p-1}\tau^{-2}\|\rho E_{\mu}^{\al}\|_{L^2(\Si_{\tau}^u)}\\
   &+M^p\de^{(1-\ve0)p}\tau^{-1}\|\Lies_{\bar Z}^{\al}\nas^2\mu\|_{L^2(\Si_{\tau}^u)}+M^p\de^{(1-\ve0)p-1}\|\ds\bar Z^{\al}\trgs\chi\|_{L^2(\Si_{\tau}^u)}+\tau\|e_{\mu}^{\al}\|_{L^2(\Si_{\tau}^u)}\big]\d\tau.
  \end{split}
 \end{equation}

Next, we estimate the terms in the integrand of \eqref{preliminary L^2 estimate of E mu} one by one (notice that the $L^2$ estimate of $\ds \bar Z^{\al}\trgs\chi$ has been obtained in \eqref{top order L^2 estimate of chi}).

\textbf{(2-a) Estimate of $\Lies_{\bar Z}^{\al}\nas^2\mu$}

\vskip 0.1 true cm

With the help of Lemma \ref{Lemma commutator [Des,Z]}, we have from \eqref{estimate of Gaussian curvature} and the elliptic estimate \eqref{elliptic estimate 2} that
 \begin{equation}\label{top L^2 estimate for nas^2 mu}
  \begin{split}
   &\de^l\|\Lies_{\bar Z}^{\al}\nas^2\mu\|_{L^2(\Si_{\tau}^u)}\\
   \lesssim&\de^l\|\bar Z^{\al}\Des\mu\|_{L^2(\Si_{\tau}^u)}+\de^l\|[\Des,\bar Z^{\al}]\mu\|_{L^2(\Si_{\tau}^u)}+\de^l\|[\Lies_{\bar Z}^{\al},\nas^2]\mu\|_{L^2(\Si_{\tau}^u)}\\
   \lesssim&\de^l\|E_{\mu}^{\al}\|_{L^2(\Si_{\tau}^u)}+\de^l\|\bar Z^{\al}E_{\mu}\|_{L^2(\Si_{\tau}^u)}+\de^l\|[\Des,\bar Z^{\al}]\mu\|_{L^2(\Si_{\tau}^u)}+\de^l\|[\Lies_{\bar Z}^{\al},\nas^2]\mu\|_{L^2(\Si_{\tau}^u)}\\
   \lesssim&\de^l\| E_{\mu}^{\al}\|_{L^2(\Si_{\tau}^u)}+(M\delta^{1-\ve0}\tau^{-1})^{p-1}\big\{\tau^{-1-s}\sqrt{E_{1,\leq|\al|+2}(\tau,u)}
   +\delta^{-1}\sqrt{E_{2,\leq|\al|+2}(\tau,u)}\big\}\\
   &+(M\delta^{1-\ve0}\tau^{-1})^{p-1}\big\{ M\delta^{-\ve0}\tau^{-2}\|\Lies_{R_i}^{\leq\al}{}^{(R_j)}\pis_T\|_{L^2(\Si_{\tau}^u)}
   +\delta^{-2}\delta^l\|Z^{\leq\al+1}\vp\|_{L^2(\Si_{\tau}^u)}\big\}\\
   &+M^p\delta^{(1-\ve0)p-1}\tau^{-2}\big\{\delta^l\|\Lies_Z^{\leq\al}{}^{(R_i)}\pis\|_{L^2(\Si_{\tau}^u)}
   +\delta^{l+1}\|\Lies_Z^{\leq\al}{}^{(T)}\pis\|_{L^2(\Si_{\tau}^u)}\big\}+\delta\tau^{-3}\delta^l\|Z^{\leq\al+1}\mu\|_{L^2(\Si_{\tau}^u)}\\
   \lesssim&\de^l\|E_{\mu}^{\al}\|_{L^2(\Si_{\tau}^u)}+M^{2p}\de^{(1-\ve0)p+\frac{1}{2}}\tau^{-2}\\
   &+M^{p-1}\de^{(1-\ve0)(p-1)}\tau^{-2}\sqrt{\Et_{1,\leq|\al|+2}(\tau,u)}
   +M^{p-1}\de^{(1-\ve0)(p-1)-1}\tau^{-2}\sqrt{\Et_{2,\leq|\al|+2}(\tau,u)}.
  \end{split}
 \end{equation}

\textbf{(2-b) Estimate for $e_{\mu}^{\al}$}

\vskip 0.1 true cm

In view of \eqref{higher order e mu} and \eqref{E mu and e mu}, one knows that all terms in \eqref{higher order e mu} can be estimated with the help of Propositions \ref{Proposition higher order L^2 estimates}, \ref{Proposition higher order L^infty estimates} and \eqref{BA}, that is,
 \begin{equation}\label{top L^2 estimate for higher order e mu}
  \begin{split}
   \de^l\|e_{\mu}^{\al}\|_{L^2(\Si_{\tau}^u)}\lesssim& M^{2p}\de^{(1-\ve0)2p-\frac{5}{2}}\tau^{-2}\\
   &+M^{p-1}\de^{(1-\ve0)(p-1)-1}\tau^{-2}\big\{\tau^{-s}\sqrt{\Et_{1,\leq|\al|+2}(\tau,u)}+\sqrt{\Et_{2,\leq|\al|+2}(\tau,u)}\big\}.
  \end{split}
 \end{equation}

Substituting \eqref{top L^2 estimate for nas^2 mu}, \eqref{top L^2 estimate for higher order e mu}
and \eqref{top order L^2 estimate of chi} into \eqref{preliminary L^2 estimate of E mu}, and
then applying Gronwall's inequality to obtain
 \begin{equation*}
  \begin{split}
   \de^l\|\rho E_{\mu}^{\al}\|_{L^2(\Si_\mathfrak t^u)}
   \lesssim &M^p\de^{(1-\ve0)p-\frac{3}{2}}\ln\mathfrak t+M^{2p-1}\delta^{(1-\ve0)(2p-1)-2}\sqrt{\int_0^u\tilde F_{1,\leq|\al|+2}(\mathfrak t,u')\d u'}\\
   &+M^{p-1}\de^{(1-\ve0)(p-1)-1}\big\{\sqrt{\Et_{1,\leq|\al|+2}(\mathfrak t,u)}+\ln\mathfrak t\sqrt{\Et_{2,\leq|\al|+2}(\mathfrak t,u)}\big\}.
  \end{split}
 \end{equation*}
Hence,
 \begin{equation*}
  \begin{split}
   \de^l\|E_{\mu}^{\al}\|_{L^2(\Si_\mathfrak t^u)}
   \lesssim& M^p\de^{(1-\ve0)p-\frac{3}{2}}\mathfrak t^{-1}\ln\mathfrak t+M^{2p-1}\delta^{(1-\ve0)(2p-1)-2}\mathfrak t^{-1}\sqrt{\int_0^u\tilde F_{1,\leq|\al|+2}(\mathfrak t,u')\d u'}\\
   &+M^{p-1}\de^{(1-\ve0)(p-1)-1}\mathfrak t^{-1}\big\{\sqrt{\Et_{1,\leq|\al|+2}(\mathfrak t,u)}+\ln\mathfrak t\sqrt{\Et_{2,\leq|\al|+2}(\mathfrak t,u)}\big\}.
  \end{split}
 \end{equation*}
It follows from the definition of $E_\mu^\al$ and \eqref{top L^2 estimate for nas^2 mu} that
 \begin{equation}\label{top order L^2 estimate of barmu}
  \begin{split}
   &\de^l\|\bar Z^{\al}\Des\mu\|_{L^2(\Si_\mathfrak t^u)}+\de^l\|\Lies_{\bar Z}^{\al}\nas^2\mu\|_{L^2(\Si_\mathfrak t^u)}\\
   \lesssim& M^p\de^{(1-\ve0)p-\frac{3}{2}}\mathfrak t^{-1}\ln\mathfrak t+M^{2p-1}\delta^{(1-\ve0)(2p-1)-2}\mathfrak t^{-1}\sqrt{\int_0^u\tilde F_{1,\leq|\al|+2}(\mathfrak t,u')\d u'}\\
   &+M^{p-1}\de^{(1-\ve0)(p-1)-1}\mathfrak t^{-1}\big\{\sqrt{\Et_{1,\leq|\al|+2}(\mathfrak t,u)}+\ln\mathfrak t\sqrt{\Et_{2,\leq|\al|+2}(\mathfrak t,u)}\big\}.
  \end{split}
 \end{equation}

 As in the estimate of $\slashed\nabla\Lies_Z^\al\check\chi$, when there exists at least one $\rho L$ in $Z^{\al}$, $\de^l\|\Lies_{Z}^{\al}\nas^2\mu\|_{L^2(\Si_\mathfrak t^u)}$ can be estimated with the help of \eqref{transport equation of mu} and Proposition \ref{Proposition higher order L^2 estimates}. Therefore,
 together with \eqref{top order L^2 estimate of barmu}, for any vector filed $Z\in\{\rho L,T,R_1,R_2,R_3\}$, we have
  \begin{equation}\label{top order L^2 estimate of mu}
 \begin{split}
 &\de^l\|Z^{\al}\Des\mu\|_{L^2(\Si_\mathfrak t^u)}+\de^l\|\Lies_{Z}^{\al}\nas^2\mu\|_{L^2(\Si_\mathfrak t^u)}\\
 \lesssim&  M^p\de^{(1-\ve0)p-\frac{3}{2}}\mathfrak t^{-1}\ln\mathfrak t+M^{2p-1}\delta^{(1-\ve0)(2p-1)-2}\mathfrak t^{-1}\sqrt{\int_0^u\tilde F_{1,\leq|\al|+2}(\mathfrak t,u')\d u'}\\
 &+M^{p-1}\de^{(1-\ve0)(p-1)-1}\mathfrak t^{-1}\big\{\sqrt{\Et_{1,\leq|\al|+2}(\mathfrak t,u)}+\ln\mathfrak t\sqrt{\Et_{2,\leq|\al|+2}(\mathfrak t,u)}\big\}.
 \end{split}
 \end{equation}

\section{Commutator estimates}\label{Section 7}

\subsection{Commuted covariant wave equation}\label{Section 7.1}

In order to get the energy estimates for $\vp_\gamma$ and its derivatives, we now choose $\Psi=\Psi_{\ga}^{|\al|+1}=Z_{|\al|+1}\cdots Z_1\vp_{\ga}$ and $\Phi=\Phi_{\ga}^{|\al|+1}=\mu\Box_g\Psi_{\ga}^{|\al|+1}$ in \eqref{energy inequality}. By commuting vectorfield $Z_k$ $(k=1,\dots,|\al|+1)$ with \eqref{linearized covariant wave equation}, the induction argument gives
\begin{equation}\label{higher order commuted covariant wave equation}
\begin{split}
\Phi_\gamma^{|\al|+1}=&\underbrace{\sum_{k=1}^{|\al|}\big(Z_{|\al|+1}+\leftidx{^{(Z_{|\al|+1})}}\lambda\big)\dots\big(Z_{|\al|+2-k}
+\leftidx{^{(Z_{|\al|+2-k})}}\lambda\big)\big(\mu\divg{}^{(Z_{|\al|+1-k})}C_{\ga}^{|\al|-k}\big)}_{\text{vanishes when $|\al|=0$}}\\
&+\mu\divg{}^{(Z_{|\al|+1})}C_{\ga}^{|\al|}+\big(Z_{|\al|+1}+\leftidx{^{(Z_{|\al|+1})}}\lambda\big)\dots\big(Z_{1}
+\leftidx{^{(Z_1)}}\lambda\big)\Phi_\gamma^0,\\
\end{split}
\end{equation}
where
 \begin{equation*}
  \begin{split}
   \divg{}^{(Z)}C_{\ga}^{j}&=\D_{\be}\big[\big({}^{(Z)}\pi^{\be\nu}-\frac{1}{2}g^{\be\nu}\trg{}^{(Z)}\pi\big)\p_{\nu}\Psi_{\ga}^{j}\big],\\
   {}^{(Z)}\la&=\frac{1}{2}\trg{}^{(Z)}\pi-\mu^{-1}Z\mu,\\
   \Psi_{\ga}^0&=\vp_{\ga},\ \Phi_{\ga}^0=\mu\Box_g\vp_{\ga}
  \end{split}
 \end{equation*}
with $\trg{}^{(Z)}\pi=g^{\al\be}{}^{(Z)}\pi_{\al\be}=-\f12\mu^{-1}{}^{(Z)}\pi_{L\underline L}+\f12\trgs{}^{(Z)}\pis$. Thus,
\begin{equation}\label{lamda}
\begin{split}
{}^{(\rho L)}\la=&\rho\trgs\check\chi+3,\\
{}^{(T)}\la=&-c^{-2}\mu\trgs\chi,\\
{}^{(R_i)}\la=&c^{-1}v_i\trgs\chi.
\end{split}
\end{equation}
Under the frame $\{L,\Lb,X_1,X_2\}$, $\mu\divg{}^{(Z)}C_{\ga}^{j}$ could be written as
 \begin{equation}\label{div C gamma alpha}
  \mu\divg{}^{(Z)}C_{\ga}^{j}=-\frac{1}{2}\big[\Lb+L(c^{-2}\mu)-c^{-2}\mu\trgs\chi\big]{}^{(Z)}C_{\ga,L}^{j}-\frac{1}{2}(L+\trgs\chi){}^{(Z)}C_{\ga,\Lb}^{j}+\nas^A(\mu{}^{(Z)}\slashed{C}_{\ga,A}^{j}),
 \end{equation}
where
 \begin{equation}\label{C gamma alpha}
  \begin{split}
   {}^{(Z)}C_{\ga,L}^{j}&=g({}^{(Z)}C_{\ga}^{j},L)=-\frac{1}{2}\trgs{}^{(Z)}\pis L\Psi_{\ga}^{j}+{}^{(Z)}\pis_{LA}\ds^A\Psi_{\ga}^{j},\\
   {}^{(Z)}C_{\ga,\Lb}^{j}&=g({}^{(Z)}C_{\ga}^{j},\Lb)=-\frac{1}{2}\mu^{-1}{}^{(Z)}\pi_{\Lb\Lb}L\Psi_{\ga}^{j}-\frac{1}{2}\trgs{}^{(Z)}\pis\Lb\Psi_{\ga}^{j}+{}^{(Z)}\pis_{\Lb A}\ds^A\Psi_{\ga}^{j},\\
   \mu{}^{(Z)}\slashed{C}_{\ga,A}^{j}&= g(\mu{}^{(Z)}C_{\ga}^{j},X_A)=-\frac{1}{2}{}^{(Z)}\pis_{\Lb A}L\Psi_{\ga}^{j}-\frac{1}{2}{}^{(Z)}\pis_{LA}\Lb\Psi_{\ga}^{j}\\
   &\qquad\qquad\qquad\qquad\quad+\frac{1}{2}{}^{(Z)}\pi_{L\Lb}\ds_A\Psi_{\ga}^{j}+\mu({}^{(Z)}\pis_{AB}-\frac{1}{2}\trgs{}^{(Z)}\pis\gs_{AB})\ds^B\Psi_{\ga}^{j}.
  \end{split}
 \end{equation}

Note that if we substitute \eqref{C gamma alpha} into \eqref{div C gamma alpha} directly, then a lengthy and tedious equality for $\mu\divg{}^{(Z)}C_{\ga}^{j}$ is obtained. To overcome this default and treat these terms more convenient, we will divide $\mu\divg{}^{(Z)}C_{\ga}^{j}$ into the following three parts as in \cite{MY17}:
 \begin{equation*}
  \mu\divg{}^{(Z)}C_{\ga}^{j}={}^{(Z)}\mathscr{N}_1^{j}+{}^{(Z)}\mathscr{N}_2^{j}+{}^{(Z)}\mathscr{N}_3^{j},
 \end{equation*}
where
 \begin{equation}\label{N1 pi Psi}
  \begin{split}
   {}^{(Z)}\mathscr{N}_1^{j}
   =&\big[\frac{1}{4}L(\mu^{-1}{}^{(Z)}\pi_{\Lb\Lb})+\frac{1}{4}\Lb(\trgs{}^{(Z)}\pis)-\frac{1}{2}\nas^A{}^{(Z)}\pis_{\Lb A}\big]L\Psi_{\ga}^{j}-\big[\frac{1}{2}\Lies_{\Lb}{}^{(Z)}\pis_{LA}\\
   &+\frac{1}{2}\Lies_L{}^{(Z)}\pis_{\Lb A}-\frac{1}{2}\ds_A{}^{(Z)}\pi_{L\Lb}-\nas^B(\mu{}^{(Z)}\slashed\pi_{AB}-\f12\mu\trgs{}^{(Z)}\pis\slashed g_{AB})\big]\ds^A\Psi_{\ga}^{j}\quad\\
   &+\big[\frac{1}{4}L(\trgs{}^{(Z)}\pis)-\frac{1}{2}\nas^A{}^{(Z)}\pis_{LA}\big]\Lb \Psi_{\ga}^{j},
  \end{split}
 \end{equation}
 \begin{equation}\label{N2 pi Psi}
  \begin{split}
   {}^{(Z)}\mathscr{N}_2^{j}
   =&\frac{1}{2}\trgs{}^{(Z)}\pis(L+\frac{1}{2}\trgs\chi)\Lb\Psi_{\ga}^{j}+\frac{1}{4}\mu^{-1}{}^{(Z)}\pi_{\Lb\Lb}L^2\Psi_{\ga}^{j}-{}^{(Z)}\pis_{\Lb A}\ds^AL\Psi_{\ga}^{j}\\
   &-{}^{(Z)}\pis_{LA}\ds^A\Lb\Psi_{\ga}^{j}+\frac{1}{2}{}^{(Z)}\pi_{L\Lb}\Des\Psi_{\ga}^{j}+\mu({}^{(Z)}\slashed\pi_{AB}-\f12\trgs{}^{(Z)}\pis\slashed g_{AB})\nas_{AB}^2\Psi_{\ga}^{j},
  \end{split}
 \end{equation}
 \begin{equation}\label{N3 pi Psi}
  \begin{split}
   {}^{(Z)}\mathscr{N}_3^{j}
   =&\big[\frac{1}{4}\mu^{-1}\trgs\chi{}^{(Z)}\pi_{\Lb\Lb}-\frac{1}{4}c^{-2}\mu\trgs\chi\trgs{}^{(Z)}\pis
   +\frac{1}{2}{}^{(Z)}\pis_{LA}\ds^A(c^{-2}\mu)\big]L\Psi_{\ga}^{j}\ \quad\quad\\
   &+\big[\frac{1}{2}c^2\ds_A(c^{-2}\mu)\trgs{}^{(Z)}\pis-\frac{1}{2}L(c^{-2}\mu){}^{(Z)}\pis_{LA}-\frac{1}{2}\trgs\chi{}^{(Z)}\pis_{\Lb A}\\
   &+\frac{1}{2}c^{-2}\mu\trgs\chi{}^{(Z)}\pis_{LA}-c^{-2}\mu{}^{(Z)}\pis_L^B\chi_{AB}+{}^{(Z)}\pis_{\Lb}^B\chi_{AB}\big]\ds^A\Psi_{\ga}^{j}.
  \end{split}
 \end{equation}
One sees that ${}^{(Z)}\mathscr{N}_1^{j}$ collects the products of the first order derivatives of
the deformation tensor and the first order derivatives of $\Psi_\ga^{j}$, and ${}^{(Z)}\mathscr{N}_2^{j}$
contains the terms which are the products of the deformation tensor and the second order derivatives of $\Psi_\ga^{j}$
except the first term which has the better smallness and higher time-decay rate due to \eqref{transport equation of Lb vp0}. In addition, ${}^{(Z)}\mathscr{N}_3^{j}$
is the collections of the products of the deformation tensor and the first order derivatives of $\Psi_\ga^{j}$
which can be treated more easily.

Through making the preparations for the $L^2$ estimates of the quantities in Subsection \ref{Section 5.2} and Section \ref{Section 6},
we are ready to handle the error terms $\int_{D^{\mathfrak t,u}}\rho^{2s}|\Phi\cdot L\Psi|$ and $\int_{D^{\mathfrak t,u}}|\Phi\cdot\Lb\Psi|$ in \eqref{energy inequality}, and then get the final energy estimates for $\vp_{\ga}$ and its derivatives.

\subsection{Error estimates for the terms originating from ${}^{(Z)}\mathscr{N}_1^{j}$}\label{Section 7.2}

For the most difficult term ${}^{(Z)}\mathscr{N}_1^{0}$ in $\Phi_{\ga}^{|\al|+1}$, the number of the top order derivatives is $|\al|$, which means that there will be some terms containing the $(|\al|+1)^{\text{th}}$ order derivatives of the deformation tensors. In this case, $\Et_{i,\leq|\al|+3}$ will appear in the right hand side of \eqref{energy inequality} if one only
adopts Proposition \ref{Proposition higher order L^2 estimates}.
This leads to that it can not be absorbed by the left hand side of \eqref{energy inequality}.
To overcome such a difficulty, we will carefully examine the expression of ${}^{(Z)}\mathscr{N}_1^{j}$ and apply the estimates
in Section \ref{Section 6} to deal with the top order derivatives of $\chi$ and $\mu$.

By substituting the components of the deformation tensor, we can obtain the expression of ${}^{(Z)}\mathscr{N}_1^{j}$.

\begin{itemize}
	\item
\textbf{The case of $Z=\rho L$}

By \eqref{L pi}, one has
 \begin{equation}\label{N1 pi Psi rho L}
  \begin{split}
   {}^{(\rho L)}\mathscr{N}_1^{j}
   =&\Big\{\rho L^2(c^{-2}\mu)+\underline{\frac{1}{2}\Lb(\rho\trgs\chi)-\rho\nas^A\big[c^2\ds_A(c^{-2}\mu)\big]}\Big\}L\Psi_\ga^j-\Big\{\Lies_L[\rho c^2\ds_A(c^{-2}\mu)]\\
   &+\ds_A(\rho L\mu+\mu)\uwave{-2\rho\nas^B(\mu\check\chi_{AB}-\f12\mu\trgs\check\chi\slashed g_{AB})}\Big\}\ds^A\Psi_\ga^j+\frac{1}{2}L(\rho\trgs\check\chi)\Lb\Psi_\ga^j.
  \end{split}
 \end{equation}

For the terms with underline in \eqref{N1 pi Psi rho L}, we point out there exists an important cancelation.
In fact, it follows from \eqref{transport equation of trgs chi along T} that
 \begin{equation}\label{important cancelation 1}
  \begin{split}
   &\frac{1}{2}\Lb(\rho\trgs\chi)-\rho\nas^A\big[c^2\ds_A(c^{-2}\mu)\big]\\
   =&\rho J-\trgs\chi+\frac{1}{2}c^{-2}\mu L(\rho\trgs\chi)-\rho\nas^A\big[c^2\mu\ds_A(c^{-2})\big],
  \end{split}
 \end{equation}
which implies that the term $\Des\mu$ has been eliminated.

For the term with wavy line in \eqref{N1 pi Psi rho L}, it follows from \eqref{elliptic equation of chi} that
 \begin{equation*}
  \begin{split}
   &-2\rho\nas^B(\mu\check\chi_{AB}-\f12\mu\trgs\check\chi\slashed g_{AB})\\
   =&-\rho\mu\uwave{\ds_A\trgs\chi}+2\rho\ze^B\check\chi_{AB}-2\rho\ze_A\trgs\check\chi-2\zeta_A-2\rho\ds^B\mu(\check\chi_{AB}-\f12\trgs\check\chi\slashed g_{AB}).
  \end{split}
 \end{equation*}

In addition, we notice that there are some terms  in \eqref{N1 pi Psi rho L}
whose factors are the derivatives with respect to $L$
and which contain the second order derivative of $\mu$ or the first derivative of $\check\chi$. In this case,
the derivatives of these terms can be treat directly by Proposition \ref{Proposition higher order L^2 estimates}
since $\chi$ and $\mu$ satisfy the transport equations \eqref{transport equation of chi} and \eqref{transport equation of mu}
respectively.

Therefore, we can arrive at
 \begin{equation}\label{principal term of N1 rho L}
  {}^{(\rho L)}\mathscr{N}_1^{j}=\rho\mu\ds\trgs\chi\cdot\ds\Psi_\ga^j+\lot.
 \end{equation}

\item
\textbf{The case of $Z=T$}

By \eqref{T pi}, we have
 \begin{equation}\label{N1 pi Psi T}
  \begin{split}
   {}^{(T)}\mathscr{N}_1^{j}
   =&\Big\{L T(c^{-2}\mu)\underbrace{-\frac{1}{2}\Lb(c^{-2}\mu\trgs\chi)+\frac{1}{2}\nas^A\big[\mu\ds_A(c^{-2}\mu)\big]}\Big\}L\Psi_\ga^j+\Big\{\underline{\frac{1}{2}\Lies_{\Lb}\big[c^2\ds_A(c^{-2}\mu)\big]}\\
   &+\frac{1}{2}\Lies_L\big[\mu\ds_A(c^{-2}\mu)\big]\underline{-\ds_AT\mu}\uwave{-2\nas^B\big(c^{-2}\mu^2(\check\chi_{AB}-\f12\trgs\check\chi\slashed g_{AB})\big)}\Big\}\ds^A\Psi_\ga^j\\
   &+\Big\{-\frac{1}{2}L(c^{-2}\mu\trgs\chi)+\underbrace{\frac{1}{2}\nas^A\big[c^2\ds_A(c^{-2}\mu)\big]}\Big\}\Lb\Psi_\ga^j.
  \end{split}
 \end{equation}

For the terms with underline in \eqref{N1 pi Psi T}, one has
 \begin{equation}\label{important cancelation 2}
  \begin{split}
   &\frac{1}{2}\Lies_{\Lb}\big[c^2\ds_A(c^{-2}\mu)\big]-\ds_AT\mu =\frac{1}{2}c^{-2}\mu\ds_AL\mu+\frac{1}{2}\Lies_{\Lb}\big[c^2\mu\ds_A(c^{-2})\big],
  \end{split}
 \end{equation}
which implies that the term $\ds_AT\mu$ has been eliminated.

For the terms with braces  in \eqref{N1 pi Psi T}, we have
 \begin{equation*}
  \begin{split}
   &-\frac{1}{2}\Lb(c^{-2}\mu\trgs\chi)+\frac{1}{2}\nas^A\big[\mu\ds_A(c^{-2}\mu)\big]\\
   =&-\frac{1}{2}c^{-2}\mu\underbrace{\Des\mu}-c^{-2}\mu J-\frac{1}{2}c^{-4}\mu^2L\trgs\chi-\frac{1}{2}\Lb(c^{-2}\mu)\trgs\chi+\frac{1}{2}\ds(c^{-2}\mu)\cdot\ds\mu+\frac{1}{2}\nas^A\big[\mu^2\ds_A(c^{-2})\big]
  \end{split}
 \end{equation*}
and
 \begin{equation*}
  \frac{1}{2}\nas^A\big[c^2\ds_A(c^{-2}\mu)\big]=\frac{1}{2}\underbrace{\Des\mu}+\frac{1}{2}\nas^A\big[c^2\mu\ds_A(c^{-2})\big].
 \end{equation*}

For the terms with wavy line  in \eqref{N1 pi Psi T},
 \begin{equation*}
 \begin{split}
&-2\nas^B\big(c^{-2}\mu^2(\check\chi_{AB}-\f12\trgs\check\chi\slashed g_{AB})\big)\\
=&-c^{-2}\mu^2\uwave{\ds_A\trgs\chi}+2c^{-2}\mu\ze^B\check\chi_{AB}-2c^{-2}\mu\ze_A\trgs\check\chi-2c^{-2}\mu\rho^{-1}\zeta_A-2\ds^B(c^{-2}\mu^2)(\check\chi_{AB}-\f12\trgs\check\chi\slashed g_{AB}).
\end{split}
 \end{equation*}

Therefore, we actually arrive at
 \begin{equation}\label{principal term of N1 T}
  {}^{(T)}\mathscr{N}_1^{j}=(\Des\mu)T\Psi_\ga^j-c^{-2}\mu^2\ds\trgs\chi\cdot\ds\Psi_\ga^j+\lot.
 \end{equation}

\item
\textbf{The case of $Z=R_i$}

By \eqref{Ri pi}, we have
 \begin{equation}\label{N1 pi Psi Ri}
  \begin{split}
   {}^{(R_i)}\mathscr{N}_1^{j}
   =&\Big\{LR_i(c^{-2}\mu)+\underline{\frac{1}{2}\Lb(c^{-1}v_i\trgs\chi)}\underbrace{-\frac{1}{2}\nas^A(c^{-2}\mu R_i^B\chic_{AB})}\underline{-\nas^A(c^{-1}v_i\ds_A\mu)}\\
   &-\frac{1}{2}\nas^A\big[-2c^{-2}(c-1)\mu\rho^{-1}\gs_{AB}R_i^B-3c^{-2}\mu v_i\ds_Ac+c^{-2}\mu\ep_{ijk}\check{L}^j\ds_Ax^k\\
   &+2c^{-1}\mu\ep_{ijk}\check{T}^j\ds_Ax^k\big]\Big\}L\Psi_\ga^j+\Big\{\underline{\frac{1}{2}\Lies_{\Lb}(R_i^B\chic_{AB}}-\ep_{ijk}\check{L}^j\ds_Ax^k+v_i\ds_Ac)\\
   &-\frac{1}{2}\Lies_L{}^{(R_i)}\pis_{\Lb A}\underline{-\ds_AR_i\mu}+\uwave{\nas^B\big(2c^{-1}\mu v_i(\check\chi_{AB}-\f12\trgs\check\chi\slashed g_{AB})\big)}\Big\}\ds^A\Psi_\ga^j\\
   &+\Big\{\frac{1}{2}L(c^{-1}v_i\trgs\chi)+\underbrace{\frac{1}{2}\nas^A(R_i^B\chic_{AB}}-\ep_{ijk}\check{L}^j\ds_Ax^k+v_i\ds_Ac)\Big\}\Lb\Psi_\ga^j.
  \end{split}
 \end{equation}

For the terms with underline in \eqref{N1 pi Psi Ri}, one has that by \eqref{transport equation of chi along T},
 \begin{equation*}
  \begin{split}
   &\frac{1}{2}\Lb(c^{-1}v_i\trgs\chi)-\nas^A(c^{-1}v_i\ds_A\mu)\\
   =&c^{-1}v_iJ+\frac{1}{2}c^{-3}\mu v_iL\trgs\chi+\frac{1}{2}\Lb(c^{-1}v_i)\trgs\chi-\ds(c^{-1}v_i)\cdot\ds\mu\\
   =&\lot
  \end{split}
 \end{equation*}
and
 \begin{equation*}
  \begin{split}
   &\quad\frac{1}{2}\Lies_{\Lb}(R_i^B\chic_{AB})-\ds_AR_i\mu=R_i^B\Lies_T\chic_{AB}-\nabla_A(R_i^B\ds_B\mu)+\lot=\lot,
  \end{split}
 \end{equation*}
 which implies that the terms $\Des\mu$ and $\ds_AR_i\mu$ have been eliminated.

For the terms with braces  in \eqref{N1 pi Psi Ri}, we have
 \begin{equation*}
  \begin{split}
   &-\frac{1}{2}\nas^A(c^{-2}\mu R_i^B\chic_{AB})=-\frac{1}{2}c^{-2}\mu\underbrace{R_i\trgs\chi}+\lot
  \end{split}
 \end{equation*}
and
 \begin{equation*}
  \frac{1}{2}\nas^A(R_i^B\chic_{AB})=\frac{1}{2}\underbrace{R_i\trgs\chi}+\lot.
 \end{equation*}

For the terms with wavy line  in \eqref{N1 pi Psi Ri},
 \begin{equation*}
  \begin{split}
   -\nas^B\big(2c^{-1}\mu v_i(\check\chi_{AB}-\f12\trgs\check\chi\slashed g_{AB})\big)=-c^{-1}\mu v_i\uwave{\ds_A\trgs\chi}+\lot.
  \end{split}
 \end{equation*}

Therefore, we arrive at
 \begin{equation}\label{principal term of N1 Ri}
  {}^{(R_i)}\mathscr{N}_1^{j}=R_i^A(\ds_A\trgs\chi)T\Psi_\ga^j-c^{-1}\mu v_i\ds\trgs\chi\cdot\ds\Psi_\ga^j+\lot.
 \end{equation}
\end{itemize}

By substituting ${}^{(Z)}\mathscr{N}_1^{j}$ into \eqref{higher order commuted covariant wave equation},
the resulting terms $(Z^{\al}\Des\mu)T\Psi_{\ga}^0$ from \eqref{principal term of N1 T} and $R_i^A(\ds_AZ^{\al}\trgs\chi) T\Psi_{\ga}^0$ from \eqref{principal term of N1 Ri} need to be treated especially.

Note that $(Z^{\al}\Des\mu)T\Psi_{\ga}^0$ comes from $Z^\al({}^{(T)}\mathscr{N}_1^{0})$ with $Z^{\al}=Z_{|\al|+1}\cdots Z_2$ and $Z_1=T$.
When the number of $T$ in $Z_{|\al|+1}\cdots Z_1$ is $l$, then $Z^\al$ contains $l-1$ vector field $T$'s.
In view of \eqref{energy inequality}, one has that by \eqref{top order L^2 estimate of mu},
 \begin{equation}\label{top error estimate 1}
  \begin{split}
   &\de^{2l}\int_{D^{\mathfrak t,u}}\rho^{2s}|(Z^{\al}\Des\mu)T\Psi_{\ga}^0||L\Psi_{\ga}^{|\al|+1}|\\
   \lesssim&M^2\delta^{3-2\ve0}\int_{t_0}^\mathfrak t\tau^{-2+2s}\big(\de^{l-1}\|Z^{\al}\Des\mu\|_{L^2(\Si_{\tau}^u)}\big)^2d\tau+\delta^{-1}\int_{D^{\mathfrak t,u}}\rho^{2s}\big(\de^{l}|L\Psi_{\ga}^{|\al|+1}|\big)^2\\
   \lesssim&M^{2p+2}\de^{(1-\ve0)(2p+2)-2}+\de^{-1}\int_0^u\tilde F_{1,\leq|\al|+2}(\mathfrak t,u')\d u'\\
   &+\delta\int_{t_0}^\mathfrak t\tau^{-4+2s}\big\{\Et_{1,\leq|\al|+2}(\tau,u)+(\ln^2\tau)\Et_{2,\leq|\al|+2}(\tau,u)\big\}\d\tau
  \end{split}
 \end{equation}
and
 \begin{equation}\label{top error estimate 2}
  \begin{split}
   &\de^{2l+1}\int_{D^{\mathfrak t,u}}|(Z^{\al}\Des\mu)T\Psi_{\ga}^0||\Lb\Psi_{\ga}^{|\al|+1}|\\
   \lesssim& M\de^{2-\ve0}\int_{t_0}^\mathfrak t\tau^{-1}\big(\de^{l-1}\|Z^{\al}\Des\mu\|_{L^2(\Si_{\tau}^u)}\big)\big(\de^l\|\Lb\Psi_{\ga}^{|\al|+1}\|_{L^2(\Si_{\tau}^u)}\big)\d\tau\\
   \lesssim&M^{2p+2}\de^{(1-\ve0)(2p+2)-2}+\de\int_0^u\tilde F_{1,\leq|\al|+2}(\mathfrak t,u')\d u'\\
   &+\delta\int_{t_0}^\mathfrak t\tau^{-2}\big\{\Et_{1,\leq|\al|+2}(\tau,u)+(\ln\tau)\Et_{2,\leq|\al|+2}(\tau,u)\big\}\d\tau.
  \end{split}
 \end{equation}

For the term $R_i^A(\ds_AZ^{\al}\trgs\chi) T\Psi_{\ga}^0$, we know $Z_1=R_i$. If there are $l$
vector field $T$'s in $Z_{|\al|+1}\cdots Z_1$, then the number of $T$ in $Z^\al=Z_{|\al|+1}\cdots Z_2$ is
just $l$. Thus, with the help of \eqref{top order L^2 estimate of chi}, we obtain

 \begin{equation}\label{top error estimate 3}
  \begin{split}
   &\de^{2l}\int_{D^{\mathfrak t,u}}\rho^{2s}|R_i^A(\ds_AZ^{\al}\trgs\chi) T\Psi_{\ga}^0||L\Psi_{\ga}^{|\al|+1}|\\
   \lesssim&M^2\delta^{1-2\ve0}\int_{t_0}^\mathfrak t\rho^{2s}\big(\de^l\|\ds Z^{\al}\trgs\chi\|_{L^2(\Si_{\tau}^u)}\big)^2d\tau+\delta^{-1}\int_{D^{\mathfrak t,u}}\rho^{2s}\big(\de^{l}|L\Psi_{\ga}^{|\al|+1}|\big)^2\\   \lesssim&M^{2p+2}\de^{(1-\ve0)(2p+2)-2}+\de^{-1}\int_0^u\tilde F_{1,\leq|\al|+2}(\mathfrak t,u')\d u'\\
   &+\delta\int_{t_0}^\mathfrak t\tau^{-4+2s}\ln^2\tau\big\{\Et_{1,\leq|\al|+2}(\tau,u)+\Et_{2,\leq|\al|+2}(\tau,u)\big\}\d\tau
  \end{split}
 \end{equation}
and
 \begin{equation}\label{top error estimate 4}
  \begin{split}
   &\de^{2l+1}\int_{D^{\mathfrak t,u}}|R_i^A(\ds_AZ^{\al}\trgs\chi)T\Psi_{\ga}^0||\Lb\Psi_{\ga}^{|\al|+1}|\\
   \lesssim&M\de^{1-\ve0}\int_{t_0}^\mathfrak t\big(\de^l\|\ds Z^{\al}\trgs\chi\|_{L^2(\Si_{\tau}^u)}\big)\big(\de^{l}\|\Lb\Psi_{\ga}^{\al+1}\|_{L^2(\Si_{\tau}^u)}\big)\d\tau\\
\lesssim&M^{2p+2}\de^{(1-\ve0)(2p+2)-2}+\de^{-1}\int_0^u\tilde F_{1,\leq|\al|+2}(\mathfrak t,u')\d u'\\
&+\delta\int_{t_0}^\mathfrak t\tau^{-3/2}\big\{\Et_{1,\leq|\al|+2}(\tau,u)+\Et_{2,\leq|\al|+2}(\tau,u)\big\}\d\tau.
  \end{split}
 \end{equation}

For the other terms coming from $\sum_{k=1}^{|\al|}\big(Z_{|\al|+1}+\leftidx{^{(Z_{|\al|+1})}}\lambda\big)\dots\big(Z_{|\al|+2-k}
	+\leftidx{^{(Z_{|\al|+2-k})}}\lambda\big)\big({}^{(Z_{|\al|+1-k})}\mathscr{N}_1^{|\al|-k}$
$+{}^{(Z_{|\al|+1})}\mathscr{N}_1^{|\al|}\big)$, they are easier to be handled in \eqref{energy inequality}
since \eqref{top order L^2 estimate of chi} can be adopted to treat $\ds Z^\al\trgs \chi$ as in \eqref{top error estimate 3}
and \eqref{top error estimate 4}.
 On the other hand, Proposition \ref{Proposition higher order L^2 estimates} can be used to estimate the others
 in ${}^{(Z)}\mathscr{N}_1^{j}$.

\subsection{Error estimates for the left terms in \eqref{higher order commuted covariant wave equation}}\label{Section 7.3}

\vskip 0.2 true cm

\textbf{$\bullet$ Display the expressions of ${}^{(Z)}\mathscr{N}_2^{j}$ and ${}^{(Z)}\mathscr{N}_3^{j}$}

\vskip 0.1 true cm

Similarly to ${}^{(Z)}\mathscr{N}_1^{j}$, one has
 \begin{equation}\label{N2 pi Psi rho L}
  \begin{split}
   {}^{(\rho L)}\mathscr{N}_2^{j}
   =&\rho\trgs\chi(L+\frac{1}{2}\trgs\chi)\Lb\Psi_\ga^j+\big[\rho L(c^{-2}\mu)-c^{-2}\mu+2\big]L^2\Psi_\ga^j-2\rho c^2\ds(c^{-2}\mu)\cdot\ds L\Psi_\ga^j\\
   &-(\rho L\mu+\mu)\Des\Psi_\ga^j+2\rho\mu(\check\chi^{AB}-\f12\trgs\check\chi\slashed g^{AB})\nas^2_{AB}\Psi_\ga^j,
  \end{split}
 \end{equation}

 \begin{equation}\label{N2 pi Psi T}
  \begin{split}
   {}^{(T)}\mathscr{N}_2^{j}
   =&-c^{-2}\mu\trgs\chi(L+\frac{1}{2}\trgs\chi)\Lb\Psi_\ga^j+T(c^{-2}\mu)L^2\Psi_\ga^j+\mu\ds(c^{-2}\mu)\cdot\ds L\Psi_\ga^j\\
   &+c^2\ds_A(c^{-2}\mu)\ds^A\Lb\Psi_\ga^j-T\mu\Des\Psi_\ga^j-2c^{-2}\mu^2(\check\chi^{AB}-\f12\trgs\check\chi\slashed g^{AB})\nas^2_{AB}\Psi_\ga^j,\qquad\qquad
  \end{split}
 \end{equation}

 \begin{equation}\label{N2 pi Psi Ri}
  \begin{split}
   {}^{(R_i)}\mathscr{N}_2^{j}
   =&c^{-1}v_i\trgs\chi(L+\frac{1}{2}\trgs\chi)\Lb\Psi_\ga^j-\big(c^{-2}\mu{}^{(R_i)}\pis_{LA}+2{}^{(R_i)}\pis_{TA}\big)\ds^AL\Psi_\ga^j-R_i\mu\Des\Psi_\ga^j\ \\
   &+{}^{(R_i)}\pis_{LA}\ds^A\Lb\Psi_\ga^j+R_i(c^{-2}\mu)L^2\Psi_\ga^j+2c^{-1}\mu v_i(\check\chi^{AB}-\f12\trgs\check\chi\slashed g^{AB})\nas^2_{AB}\Psi_\ga^j
  \end{split}
 \end{equation}
and
 \begin{equation}\label{N3 pi Psi rho L}
  \begin{split}
   {}^{(\rho L)}\mathscr{N}_3^{j}
   =&\trgs\chi\big\{\rho L(c^{-2}\mu)-2c^{-2}(\mu-1)+2(1-c^{-2})-\frac{1}{2}\rho c^{-2}\mu\trgs\check\chi\big\}L\Psi_\ga^j\qquad\qquad\qquad\\
   &+2\rho c^2\ds^B(c^{-2}\mu)\chi_{AB}\ds^A\Psi_\ga^j,
  \end{split}
 \end{equation}

 \begin{equation}\label{N3 pi Psi T}
  \begin{split}
   {}^{(T)}\mathscr{N}_3^{j}
   =&\big[T(c^{-2}\mu)\trgs\chi+\frac{1}{2}c^{-4}\mu^2(\trgs\chi)^2-\frac{1}{2}c^2|\ds(c^{-2}\mu)|^2\big]L\Psi_\ga^j+\big[\frac{1}{2}c^2L(c^{-2}\mu)\ds_A(c^{-2}\mu)\\
   &-\mu\trgs\chi\ds_A(c^{-2}\mu)\big]\ds^A\Psi_\ga^j,
  \end{split}
 \end{equation}

 \begin{equation}\label{N3 pi Psi Ri}
  \begin{split}
   {}^{(R_i)}\mathscr{N}_3^{j}
   =&\big[R_i(c^{-2}\mu)\trgs\chi-\frac{1}{2}c^{-3}\mu v_i(\trgs\chi)^2+\frac{1}{2}{}^{(R_i)}\pis_{LA}\ds^A(c^{-2}\mu)\big]L\Psi_\ga^j\\
   &+\big[cv_i\ds_A(c^{-2}\mu)\trgs\chi-\frac{1}{2}L(c^{-2}\mu){}^{(R_i)}\pis_{LA}+2{}^{(R_i)}\pis_{T}^B\check\chi_{AB}-{}^{(R_i)}\pis_{TA}\trgs\check\chi\big]\ds^A\Psi_\ga^j.
  \end{split}
 \end{equation}

Set
\begin{equation*}
\begin{split}
\mathscr M^{\al}=&\sum_{k=1}^{|\al|}\big(Z_{|\al|+1}+\leftidx{^{(Z_{|\al|+1})}}\lambda\big)\dots\big(Z_{|\al|+2-k}+\leftidx{^{(Z_{|\al|+2-k})}}\lambda\big)\big({}^{(Z_{|\al|+1-k})}\mathscr{N}_2^{|\al|-k}+{}^{(Z_{|\al|+1-k})}\mathscr{N}_3^{|\al|-k}\big)\\
&+{}^{(Z_{|\al|+1})}\mathscr{N}_2^{|\al|}+{}^{(Z_{|\al|+1})}\mathscr{N}_3^{|\al|}.
\end{split}
\end{equation*}
Then using \eqref{N2 pi Psi rho L}-\eqref{N3 pi Psi Ri}, Proposition \ref{Proposition higher order L^infty estimates} and \ref{Proposition higher order L^2 estimates} to derive
\begin{equation*}
\begin{split}
\delta^l\|\mathscr M^\al\|_{L^2(\Sigma_\tau^u)}\lesssim M^{p+1}\delta^{(1-\ve0)(p+1)-\f12}\tau^{-2}+\tau^{-1-s}\sqrt{\Et_{1,\leq|\al|+2}(\tau,u)}+\delta\tau^{-2}\sqrt{\Et_{2,\leq|\al|+2}(\tau,u)}.
\end{split}
\end{equation*}
Therefore,
\begin{equation}\label{error estimate 5}
\begin{split}
&\delta^{2l}\int_{D^{\mathfrak t,u}}|\mathscr M^\al|(\rho^{2s}|L\Psi_\ga^{|\al|+1}|+\delta|\underline L\Psi_\ga^{|\al|+1}|)\\
\lesssim&\delta^{2l+1}\int_{t_0}^\mathfrak t\rho^{2s}\|\mathscr M^\al\|_{L^2(\Sigma_\tau^u)}^2d\tau+\delta^{-1}\int_{D^{\mathfrak t,u}}\rho^{2s}\delta^{2l}|L\Psi_\ga^{|\al|+1}|^2+\delta^{2l+1}\int_{t_0}^\mathfrak t\|\mathscr M^\al\|_{L^2(\Sigma_\tau^u)}\|\underline L\Psi_\ga^{|\al|+1}\|_{L^2(\Sigma_\tau^u)}d\tau\\
\lesssim&M^{2p+2}\delta^{(1-\ve0)(2p+2)}+\de^{-1}\int_0^u\tilde F_{1,\leq|\al|+2}(\mathfrak t,u')\d u'+\delta\int_{t_0}^\mathfrak t\tau^{-1-s}\big\{\Et_{1,\leq|\al|+2}(\tau,u)+\Et_{2,\leq|\al|+2}(\tau,u)\big\}\d\tau.
\end{split}
\end{equation}

\vskip 0.2 true cm
\textbf{$\bullet$ Error estimates on $\big(Z_{|\al|+1}+\leftidx{^{(Z_{|\al|+1})}}\lambda\big)\dots\big(Z_{1}
	+\leftidx{^{(Z_1)}}\lambda\big)\Phi_{\ga}^0$}

\vskip 0.1 true cm

By \eqref{covariant wave equation}, \eqref{F la} and Proposition \ref{Proposition higher order L^2 estimates}, one has
\begin{equation*}
\begin{split}
&\delta^l\|\big(Z_{|\al|+1}+\leftidx{^{(Z_{|\al|+1})}}\lambda\big)\dots\big(Z_{1}
+\leftidx{^{(Z_1)}}\lambda\big)\Phi_{\ga}^0\|_{L^2(\Sigma_\tau^u)}\\
\lesssim&M^{2p+1}\delta^{(1-\ve0)(2p+1)}\tau^{-(p+2)}+M^p\delta^{(1-\ve0)p-1}\tau^{-p}\big\{\tau^{-s}\sqrt{\Et_{1,\leq|\al|+2}(\tau,u)}
+\delta\tau^{-1}\sqrt{\Et_{2,\leq|\al|+2}}\big\}.
\end{split}
\end{equation*}
Then
 \begin{equation}\label{error estimate 9}
  \begin{split}
   &\quad\de^{2l}\int_{D^{\mathfrak t,u}}|\big(Z_{|\al|+1}+\leftidx{^{(Z_{|\al|+1})}}\lambda\big)\dots\big(Z_{1}
   +\leftidx{^{(Z_1)}}\lambda\big)\Phi_{\ga}^0|\big(\rho^{2s}|L\Psi_{\ga}^{|\al|+1}|+\delta|\underline L\Psi_\ga^{|\al|+1}|\big)\\
\lesssim&M^{4p+2}\delta^{(1-\ve0)(4p+2)}+\de^{-1}\int_0^u\tilde F_{1,\leq|\al|+2}(\mathfrak t,u')\d u'+\delta\int_{t_0}^\mathfrak t\tau^{-p-s}\big\{\Et_{1,\leq|\al|+2}(\tau,u)+\Et_{2,\leq|\al|+2}(\tau,u)\big\}\d\tau.
  \end{split}
 \end{equation}

\section{Global existence}\label{Section 8}

\subsection{Global existence of $\phi$ near $C_0$}\label{Section 8.1}

By substituting the estimates \eqref{top error estimate 1}-\eqref{top error estimate 4}
and \eqref{error estimate 5}-\eqref{error estimate 9} into \eqref{energy inequality},
then it follows from Gronwall's inequality that by $p>p_c$,
 \begin{equation}\label{main energy estimate}
 \begin{split}
  &\Et_{1,\leq 2N-4}(\mathfrak t,u)+\tilde F_{1,\leq 2N-4}(\mathfrak t,u)+\de\Et_{2,\leq 2N-4}(\mathfrak t,u)
  +\de \tilde F_{2,\leq 2N-4}(\mathfrak t,u)\\
  \lesssim& M^{2p+2}\de^{(1-\ve0)(2p+2)-2}+\delta^{2-2\ve0}\\
  \lesssim&\delta^{2-2\ve0}.
  \end{split}
 \end{equation}

Together with the following Sobolev-type embedding formula on $S_{\mathfrak t,u}$
 \begin{equation}\label{Sobolev embedding on S_tu}
  \|f\|_{L^{\infty}(S_{\mathfrak t,u})}\lesssim\frac{1}{\mathfrak t}\sum_{|\be|\leq 2}\|R_i^{\be}f\|_{L^2(S_{\mathfrak t,u})},
 \end{equation}
we have that for $|\al|\leq 2N-7$,
 \begin{equation}\label{BA closed}
  \begin{split}
   \de^l|Z^{\al}\vp_{\ga}|
   &\lesssim\frac{1}{\mathfrak t}\sum_{|\be|\leq 2}\de^l\|R_i^{\be}Z^{\al}\vp_{\ga}\|_{L^2(S_{\mathfrak t,u})}\\
   &\lesssim\de^{\frac{1}{2}}\big(\sqrt{E_{1,\leq 2N-4}}+\sqrt{E_{2,\leq 2N-4}}\big)\mathfrak t^{-1}\\
   &\lesssim\de^{1-\ve0}\mathfrak t^{-1},
  \end{split}
 \end{equation}
where the related bounds are independent of $M$. This improves the bootstrap assumptions \eqref{BA}, and
hence we have proved the global existence of the solution $\phi$ to the equation \eqref{main equation} with the
initial data \eqref{initial data}-\eqref{condition on data} in the domain $D^{\mathfrak t,4\de}$. Moreover, since \eqref{BA closed} holds,
with the same argument as in the proof of Proposition \ref{Proposition higher order L^infty estimates},
the constants in \eqref{P6.1} can be refined to be independent of $M$ when $|\al|\leq 2N-9$.

When $N$ is large enough, we now claim that on $\Ct_{2\de}=\{(t,x):t\geq1+2\delta,t-r=2\delta\}$,
 \begin{equation}\label{estimate of phi on Ct_2de}
  |\Ga^{\al}\phi|\lesssim\de^{2-\ve0}t^{-1},\ |\al|\leq 2N-8,
 \end{equation}
here and below $\Ga\in\{(t+r)\Lt,\Lbt,\Om_1,\Om_2,\Om_3\}$.

We next focus on the proof of \eqref{estimate of phi on Ct_2de}.

It follows from $T=c^{-1}\mu \Tr$ and $\check{T}^i=\Tr^i+\frac{x^i}{\rho}$ that
 \begin{equation*}
  \om_iT^i=\om_ic^{-1}\mu\Tr^i=c^{-1}\mu(\om_i\check{T}^i-\frac{r}{\rho}).
 \end{equation*}
This, together with $\p_i=c^2\mu^{-2}T^iT+\ds^Ax^iX_A$, derives
 \begin{equation*}
  \p_r=\sum_{i=1}^3\om_i\p_i=c\mu^{-1}(\om_i\check{T}^i-\frac{r}{\rho})T+\om_i\ds^Ax^iX_A.
 \end{equation*}
In addition, by \eqref{vi} and \eqref{Ri}, one has
 \begin{equation*}
  \Om_i=R_i+\ep_{ijk}x^jc\mu^{-1}\check{T}^kT.
 \end{equation*}
Then by $\p_t=L+c^2\mu^{-1}T$, we have
 \begin{equation}\label{relation of Lt and L}
  \left\{
   \begin{aligned}
      \Lt&=L+\big[c^2\mu^{-1}+c\mu^{-1}(\om_i\check{T}^i-\frac{r}{\rho})\big]T+\om_i\ds^Ax^iX_A,\\
     \Lbt&=L+\big[c^2\mu^{-1}-c\mu^{-1}(\om_i\check{T}^i-\frac{r}{\rho})\big]T-\om_i\ds^Ax^iX_A,\\
    \Om_i&=R_i+\ep_{ijk}x^jc\mu^{-1}\check{T}^kT.
   \end{aligned}
  \right.
 \end{equation}
Due to $g(\Tr,\Tr)=1$, then $\f{r^2}{\rho^2}-2(\check T^i\omega_i)\f{r}\rho+g_{ij}\check T^i\check T^j-1=0$
and further
 \begin{equation*}
 \f{r}\rho-1=\f{2\check T^i\omega_i-g_{ij}\check T^i\check T^j}{\sqrt{(\check T^i\omega_i)^2
 -g_{ij}\check T^i\check T^j+1}+1-\check T^i\omega_i}.
 \end{equation*}
This yields that in $D^{\mathfrak t,4\de}$,
 \begin{equation}\label{rrho-1}
 |Z^\al(\f r\rho-1)|\lesssim\delta^{(1-\ve0)p-l}\mathfrak t^{-1},
 \end{equation}
where $|\al|\leq 2N-8$ and $l$ is the number of $T$ in $Z^\al$. Note that $c^2\mu^{-1}+c\mu^{-1}(\omega_i\check T^i-\f r\rho)=c\mu^{-1}(c-1)+c\mu^{-1}+c\mu^{-1}(1-\f r\rho)+c\mu^{-1}\omega_i\check T^i$. Then by \eqref{BA closed}, \eqref{rrho-1} and \eqref{P6.1} (neglecting the unimportant constant $M$), we obtain that in $D^{\mathfrak t,4\de}$,
 \begin{equation}\label{estimate of the distance between Lt and L}
  \begin{split}
  &\de^l|Z^{\be}\big[c^2\mu^{-1}+c\mu^{-1}(\om_i\check{T}^i-\frac{r}{\rho})\big]|\lesssim\delta^{(1-\ve0)p}\mathfrak t^{-1},\\
   &\de^l|Z^{\be}\big[c^2\mu^{-1}- c\mu^{-1}(\om_i\check{T}^i-\frac{r}{\rho})\big]|\lesssim 1,\\
   &\de^l|Z^{\be}(\om_i\ds^Ax^i)|\lesssim 1,\\
   &\de^l|Z^{\be}(\ep_{ijk}x^jc\mu^{-1}\check{T}^k)|\lesssim\de^{(1-\ve0)p},
  \end{split}
 \end{equation}
where $|\be|\leq 2N-8$ and $l$ is the number of $T$ in $Z^{\al}$. Therefore, in $D^{\mathfrak t,4\de}$, combining \eqref{estimate of the distance between Lt and L}, \eqref{relation of Lt and L} with \eqref{BA closed} derives
 \begin{equation}\label{preliminary estimate of Ga vp}
  \de^l|\Ga^{\al}\vp_{\ga}|\lesssim\de^{1-\ve0}t^{-1},\ |\al|\leq 2N-7,
 \end{equation}
 where $l$ is the number of $\tilde{\underline L}$ in $\Ga^\al$.
Recalling that $\phi$ is the solution to \eqref{main equation} and $\vp_{\ga}=\p_{\ga}\phi$,
then by \eqref{preliminary estimate of Ga vp} we have
 \begin{equation}\label{preliminary estimate of Lbt Ga vp}
  \de^l|\Lbt\Ga^{\al}\phi|\lesssim\de^{1-\ve0}t^{-1},\ |\al|\leq 2N-7.
 \end{equation}
Integrating \eqref{preliminary estimate of Lbt Ga vp} along the integral curves of $\Lbt$,
and choosing the zero boundary value on $C_0$, one has that  in $D^{\mathfrak t,4\de}$,
 \begin{equation}\label{preliminary estimate of Ga phi}
  \de^l|\Ga^{\al}\phi|\lesssim\de^{2-\ve0}t^{-1},\ |\al|\leq 2N-7.
 \end{equation}

In addition, it follows from \eqref{rrho-1} that the distance between $C_0$ and $C_{4\delta}$
on the hypersurface $\Sigma_{\mathfrak t}$ is $4\delta
 +O(\delta^{(1-\varepsilon_0)p})$ and the characteristic
 surface $C_u$ $(0\leq u\leq 4\delta)$ is almost straight with the error $O(\delta^{(1-\varepsilon_0)p})$ from
 the corresponding  outgoing light conic surface. Next we improve the estimate \eqref{preliminary estimate of Ga phi}.

 Rewriting the equation \eqref{main equation} as
 \begin{equation}\label{local L Lb phi}
  \t L\t \Lb\phi=\frac{1}{\Theta}\big[\frac{1}{r}(\t L\phi-\t \Lb\phi)+\frac{1}{r^2}\Des\phi-\frac{1}{2^{p+2}}\sum_{i=0}^p(\t L\phi)^i(\t \Lb\phi)^{p-i}(\t L^2\phi+\t \Lb^2\phi)\big],
 \end{equation}
where
 \begin{equation*}
  \Theta=1+\frac{1}{2^{p+1}}\sum_{i=0}^p(\t L\phi)^i(\t \Lb\phi)^{p-i}.
 \end{equation*}

It follows from  \eqref{local L Lb phi} and direct computation that
 \begin{equation}\label{transport equation of rLbt phi}
  \begin{split}
   \Lt(r\Lbt\phi)
   &=\Lbt\phi+r\Lt\Lbt\phi=\frac{1}{\Theta}(\Theta-1)\Lbt\phi+\frac{1}{\Theta}\big[\Lt\phi-\Lbt\phi+\frac{1}{r}\Des\phi+r(\Lt\phi)^p(\Lt^2\phi+\Lbt^2\phi)\big]+\frac{1}{\Theta}G,
  \end{split}
 \end{equation}
where
 \begin{equation*}
  \begin{split}
  G&=r\Lbt\phi\sum_{i=1}^p(\Lbt\phi)^{i-1}(\Lt\phi)^{p-i}(\Lt^2\phi+\Lbt^2\phi).
  \end{split}
 \end{equation*}
This, together with the estimates in \eqref{BA}, yields
 \begin{equation}\label{estimate of Lt(rLbt phi)}
  \begin{split}
   |\Lt(r\Lbt\phi)|
   \lesssim\de^{(1-\ve0)p-1}t^{-p}|r\Lbt\phi|+\de^{2-\ve0}t^{-2}.
  \end{split}
 \end{equation}
Integrating \eqref{estimate of Lt(rLbt phi)} along the integral curves of $\Lt$ from
$t=t_0$ and applying the estimates in Theorem \ref{Theorem local existence}, one then has that on $\Ct_{2\de}$,
 \begin{equation*}
  |\Lbt\phi|\lesssim\de^{2-\ve0}t^{-1}.
 \end{equation*}
Hence, $|\Lt\Lbt\phi|\lesssim\de^{2-\ve0}t^{-2}$ holds by \eqref{local L Lb phi}.
By the induction argument and \eqref{local L Lb phi}, we complete the proof of \eqref{estimate of phi on Ct_2de}.

Finally, we point out that \eqref{estimate of phi on Ct_2de} will play an
important role in solving the global Goursat problem of
\eqref{main equation} in $B_{2\de}$.

\subsection{Global existence of $\phi$ in $B_{2\de}$}\label{Section 8.2}
In this subsection, we establish the global  existence of solution $\phi$ to the equation \eqref{main equation}
in $B_{2\de}$. Set
 \begin{equation}
  D_T=\{(\tau,x): \tau-r\geqslant 2\delta, t_0\leqslant\tau\leqslant T\}\subset B_{2\de}.
 \end{equation}

\begin{lemma}\label{modified Klainerman-Sobolev inequality}
For any function $f(t,x)\in C^{\infty}(\mathbb{R}^{1+3})$, $t\geq 1$, $(t,x)\in D_T$, we have the following inequalities:
\begin{itemize}
  \item For $|x|\leq\frac{1}{4}t$ and any $s\ge 0$, then
        \begin{equation}\label{leq}
         |f(t,x)|\lesssim\sum_{i=0}^2t^{-\frac{3}{2}}\de^{(\frac{1}{2}-i)s}\|\bar{\Ga}^{2-i}f(t,\cdot)\|_{L^2(r\leq\frac{t}{2})}.
        \end{equation}
  \item For $|x|\geq\frac{1}{4}t$, then
        \begin{equation}\label{geq}
         |f(t,x)|\lesssim |f(t,B_t^x)|+t^{-1}\|\Om^{\leq 2}\p^{\leq 1}f(t,\cdot)\|_{L^2(\frac{t}{4}\leq r\leq t-2\de)},
        \end{equation}
\end{itemize}
where $\bar{\Ga}\in\{S,H_i,\Om_j\}$, $(t,B_t^x)$ is the intersection point of the boundary $\Ct_{2\de}$ and the ray crossing $(t,x)$ which emanates from $(t,0)$.
\end{lemma}

\begin{remark}
It should be remarked that although this Lemma and its proof are similar to those of Proposition  3.1 in \cite{MPY},
the refined inner estimate, \eqref{leq} is new (the appearance of factor $\delta^{(\f12-i)s}$) and is crucial for
 the treatment of the short pulse initial data as in \cite{Ding3}.
\end{remark}

\begin{proof}
	Let $\chi$ be a non-negative smooth cut off function on $\Bbb R_+$ such that
	\begin{equation*}
	\chi(t)=
	\begin{cases}
	1,&\ 0\leq t\leq\f14,\\
	0,&\ t\geq\f12.
	\end{cases}
	\end{equation*}
	Set $f_1(t,x)=\chi(\f{|x|}{t})f(t,x)$ and $f_2(t,x)=\big(1-\chi(\f{|x|}{t})\big)f(t,x)$, then $f(t,x)=f_1(t,x)+f_2(t,x)$ and $\textrm{supp}f_1\subset\{(t,x)\mid |x|\leq\f12t\}$, $\textrm{supp}f_2\subset\{(t,x)\mid |x|\geq\f14t\}$.
	\begin{enumerate}
		\item For any point $(t,x)$ satisfying $|x|\leq\f14t$, $f_1(t,x)=f(t,x)$. We perform the change of variable $x=t\delta^sy$,
and make use of the Sobolev embedding theorem for $y$ to have
		\begin{equation}\label{ins}
		\begin{split}
		|f(t,x)|&=|f_1(t,t\delta^sy)|\lesssim\sum_{|\al|\leq 2}\Big(\int_{\Bbb R^3}|\p_y^\al\big(f(t,t\delta^sy)\chi(\delta^sy)\big)|^2dy\Big)^{1/2}\\
		&\lesssim\sum_{|\al|\leq 2}\Big(\int_{|\delta^sy|\leq 1/2}(t\delta^s)^{2|\al|}|(\p_x^\al f)(t,t\delta^sy)|^2dy\Big)^{\f12}\\
		&\lesssim\sum_{|\al|\leq 2}\Big(\int_{|z|\leq\f12t}(t\delta^s)^{2|\al|-3}|\p_z^\al f(t,z)|^2dz\Big)^{1/2}.
		\end{split}
		\end{equation}
Note that
	\[
	\p_{i}=-\f{1}{t-r}\big(\f{x^i}{t+r}S-\f{t}{t+r}H_i+\f{x^j}{t+r}{\epsilon_{ji}}^k\O_k\big)
	\]
	and $t\sim t-|z|$ in the domain $\{(t,z): |z|\leq\f12t\}$. Then we have
	\begin{equation}\label{tf}
	\begin{split}
	&|t\p_zf(t,z)|\lesssim|(t-|z|)\p_zf(t,z)|\lesssim|\bar{\Gamma} f(t,z)|,\\
	&|t^2\p_z^2f(t,z)|\lesssim\sum_{|\al|\leq 2}|\bar\Gamma^2 f(t,z)|.
	\end{split}
	\end{equation}
	Therefore, substituting \eqref{tf} into \eqref{ins} yields
    \begin{equation*}
    \mid f(t,x)\mid\lesssim\sum_{i=0}^2 t^{-3/2}\delta^{(\f12-i)s}\|\bar{\Gamma}^{2-i}f(t,\cdot)\|_{L^2(r\leq\f12t)}.
    \end{equation*}
    \item
    When $(t,x)$ satisfies $|x|\geq\f14t$, by Newton-Leibnitz formula and the Sobolev embedding theorem on the spheres $S_t^\rho$
    with  radius $\rho$ and center at the origin on $\Sigma_t$, we have
    \begin{equation*}
    \begin{split}
    f^2(t,x)&=f^2(t,B_t^x)-\int_{|x|}^{t-2\delta}\p_\rho\big(f^2(t,\rho\omega)\big)d\rho\\
    &\lesssim f^2(t,B_t^x)+\int_{|x|}^{t-2\delta}\f{1}{\rho^2}\sum_{|\al|,|\beta|\leq 2}\|\Omega^\al f\|_{L^2(S_t^\rho)}\|\Omega^{{\beta}} \p f\|_{L^2(S_t^\rho)}d\rho\\
    &\lesssim f^2(t,B_t^x)+\sum_{|\al|\leq 2,|\beta|\leq 1}t^{-2}\|\Omega^\al\p^\beta f(t,\cdot)\|^2_{L^2(t/4\leq r\leq t-2\delta)}.
    \end{split}
    \end{equation*}
  This completes the proof of \eqref{geq}.
    	\end{enumerate}
\end{proof}

Define the energy
 \begin{equation}\label{Definition energy}
  E_{k,l}(t)=\int_{\Si_t\cap D_T}\big[(\p_tv)^2+|\na v|^2\big],\ k+l\leq 6,
 \end{equation}
where $v=\widetilde{\Ga}^k\Om^l\phi$, $\widetilde{\Ga}\in\{\p,S,H_i\}$. Thanks to the estimates on $\Si_{t_0}$
in Theorem \ref{Theorem local existence},
we make the following bootstrap assumptions:

For $t\geq t_0$, there exists a uniform constant $M_0$ such that
 \begin{equation}\label{global BA}
  \begin{split}
   E_{k,l}(t)&\leq M_0^2\de^{4-2\ve0},\ k=0,1,\\
   E_{k,l}(t)&\leq M_0^2\de^{7-2k-2\ve0},\ 2\leq k\leq 6.
  \end{split}
 \end{equation}

We next establish the following $L^{\infty}$ estimates:
\begin{proposition}\label{global L infty estimate}
Under the assumptions \eqref{global BA}, for sufficiently small $\de>0$, it holds that in $D_T$,
 \begin{equation}\label{YHC-3}
  \begin{split}
   |\p\Om^{\leq 3}\phi|&\lesssim M_0\de^{\frac{13}{8}-\ve0}t^{-1},\\
   |\p\widetilde{\Ga}\Om^{\leq 2}\phi|&\lesssim M_0\de^{\frac{7}{8}-\ve0}t^{-1},\\
   |\p\widetilde{\Ga}^2\Om^{\leq 1}\phi|&\lesssim M_0\de^{-\ve0}t^{-1},\\
   |\p\widetilde{\Ga}^3\phi|&\lesssim M_0\de^{-1-\ve0}t^{-1}.
  \end{split}
 \end{equation}
\end{proposition}

\begin{proof}
\begin{itemize}
  \item When $|x|\leq\frac{t}{4}$, it follows from Lemma \ref{modified Klainerman-Sobolev inequality} that
        \begin{equation*}
         \begin{split}
          |\p\Om^{\leq 3}\phi|
          &\lesssim t^{-\frac{3}{2}}\{\de^{-\frac{3}{2}s}\|\p\Om^{\leq 3}\phi\|_{L^2(\Si_t\cap D_T)}+\de^{-\frac{1}{2}s}\|\bar{\Ga}\p\Om^{\leq 3}\phi\|_{L^2(\Si_t\cap D_T)}+\de^{\frac{1}{2}s}\|\bar{\Ga}^2\p\Om^{\leq 3}\phi\|_{L^2(\Si_t\cap D_T)}\}\\
          &\lesssim t^{-\frac{3}{2}}\{\de^{-\frac{3}{2}s}\sqrt{E_{0,l}}+\de^{-\frac{1}{2}s}\sqrt{E_{1,l}}+\de^{\frac{1}{2}s}\sqrt{E_{2,l}}\}\\
          &\lesssim t^{-\frac{3}{2}}M_0\{\de^{-\frac{3}{2}s+2-\ve0}+\de^{-\frac{1}{2}s+2-\ve0}+\de^{\frac{1}{2}s+\frac{3}{2}-\ve0}\}.
         \end{split}
        \end{equation*}
        Choosing $s=\frac{1}{4}$, then
        \begin{equation}\label{case1}
         |\p\Om^{\leq 3}\phi|\lesssim M_0\de^{\frac{13}{8}-\ve0}t^{-\frac{3}{2}}.
        \end{equation}
  \item When $\frac{t}{4}\leq|x|\leq t-2\de$, by Lemma \ref{modified Klainerman-Sobolev inequality} and \eqref{estimate of phi on Ct_2de}, one has
        \begin{equation}\label{case2}
        \begin{split}
         |\p\Om^{\leq 3}\phi|\lesssim &|\p\Om^{\leq 3}\phi(t,B_t^x)|+t^{-1}\|\Om^{\leq 2}\p^{\leq 1}\p\Om^{\leq 3}\phi(t,\cdot)\|_{L^2(\frac{t}{4}\leq r\leq t-2\de)}\\
         \lesssim&\delta^{2-\ve0}t^{-1}+t^{-1}\big(\sqrt{E_{0,\leq 6}}+\sqrt{E_{1,\leq 5}}\big)\\
         \lesssim &M_0\de^{2-\ve0}t^{-1}.
         \end{split}
        \end{equation}
\end{itemize}
Combining \eqref{case1} and \eqref{case2} yields
 \begin{equation*}
  |\p\Om^{\leq 3}\phi|\lesssim M_0\de^{\frac{13}{8}-\ve0}t^{-1}.
 \end{equation*}

For the other cases, the procedures of proof are analogous. Indeed, in  Lemma \ref{modified Klainerman-Sobolev inequality},
\begin{itemize}
  \item for $\p\widetilde{\Ga}^1\Om^{\leq 2}\phi$, let $s=\frac{3}{4}$;
  \item for $\p\widetilde{\Ga}^2\Om^{\leq 1}\phi$, let $s=1$;
  \item for $\p\widetilde{\Ga}^3\phi$, let $s=1$.
\end{itemize}
Then Proposition \ref{global L infty estimate} is proved.
\end{proof}


Next we take the energy estimates of $\phi$ in $D_T$.
It follows from the integration $\p_tvg^{\al\be}\p_{\al\be}^2v$ over $D_t$ by parts and direct computation that
 \begin{equation}\label{YH-14}
  \begin{split}
   &\int_{\Si_t\cap D_t}\big[\frac{1}{c^2}(\p_tv)^2+|\na v|^2\big]\\
   \lesssim&\int_{\Si_{t_0}\cap D_t}\big[\frac{1}{c^2}(\p_tv)^2+|\na v|^2\big]+\int_{\Ct_{2\de}\cap D_t}|\frac{1}{2}\big[\frac{1}{c^2}(\p_tv)^2+|\na v|^2\big]+\om^i\p_iv\p_tv|\\
   &+\int_{D_t}|\p_tvg^{\al\be}\p_{\al\be}^2v|+\int_{D_t}|(\p_t\phi)^{p-1}\p_t^2\phi|(\p_tv)^2.
  \end{split}
 \end{equation}
On $\Ct_{2\de}$, by the estimate \eqref{estimate of phi on Ct_2de}, we have
 \begin{equation}\label{YH-15}
  \begin{split}
   &\int_{\Ct_{2\de}\cap D_t}|\frac{1}{2}\big[\frac{1}{c^2}(\p_tv)^2+|\na v|^2\big]+\om^i\p_iv\p_tv|\\
   \lesssim&\int_{\Ct_{2\de}\cap D_t}\big\{\Sigma_{i=1}^3(\om^i\p_tv+\p_iv)^2+|\p_t\phi|^p(\p_tv)^2\big\}\\
   \lesssim&\int_{\Ct_{2\de}\cap D_t}\left\{(Lv)^2+\sum_{i=1}^3(\frac{\om^j}{r}\ep_{ji}{}^m\Om_mv)^2+|\p_t\phi|^p(\p_tv)^2\right\}\\
   \lesssim&\de^{4-2\ve0}.
  \end{split}
 \end{equation}
On the initial hypersurface $\Si_{t_0}\cap D_t$, with the help of Theorem \ref{Theorem local existence}, one has
 \begin{equation}\label{YH-16}
 \int_{\Si_{t_0}\cap D_t}\big[\frac{1}{c^2}(\p_tv)^2+|\na v|^2\big]\lesssim
  \left\{
  \begin{aligned}
   \de\cdot(\de^{2-\ve0})^2&=\de^{5-2\ve0},\ k\leq\min\{1,6-l\},\\
   \de\cdot(\de^{3-k-\ve0})^2&=\de^{7-2k-2\ve0},\ 2\leq k\leq 6-l.
  \end{aligned}
  \right.
 \end{equation}
Therefore, it follows from \eqref{YH-14}-\eqref{YH-16} and the Gronwall's inequality that
 \begin{align}
  &E_{k,l}(t)\lesssim\delta^{4-2\varepsilon_0}+\int_{D_t}\mid (\partial_t\tilde\Gamma^k\Omega^l\phi)(g^{\alpha\beta}\partial_{\alpha\beta}^2\tilde\Gamma^k\Omega^l\phi)\mid,\quad k\leqslant \min\{1, 6-l\},\label{E01}\\
  &E_{k,l}(t)\lesssim\delta^{7-2k-2\varepsilon_0}+\int_{D_t}\mid (\partial_t\tilde\Gamma^k\Omega^l\phi)(g^{\alpha\beta}\partial_{\alpha\beta}^2\tilde\Gamma^k\Omega^l\phi)\mid,\quad 2\leqslant k\leqslant 6-l.\label{E2-6}
 \end{align}

The remaining task is to estimate the term $\int_{D_t}\mid(\partial_t\tilde\Gamma^k\Omega^l\phi)(g^{\alpha\beta}\partial_{\alpha\beta}^2\tilde\Gamma^k\Omega^l\phi)\mid$
in the right hand sides of \eqref{E01} and \eqref{E2-6}.
Acting the operator $\tilde{\Gamma}^k\Omega^l$ on $\eqref{main equation}$ and commuting it with $g^{\alpha\beta}\partial_{\alpha\beta}^2$ yield
 \begin{equation}\label{YHC-2}
  |g^{\alpha\beta}\partial_{\alpha\beta}^2\tilde{\Gamma}^k\Omega^l\phi|\lesssim\sum_{\substack{k_1+\cdots k_{p+1}\leqslant k\\l_1+\cdots l_{p+1}\leqslant l\\k_{p+1}+l_{p+1}<k+l}}|\partial\widetilde{\Gamma}^{k_1}\Omega^{l_1}\phi|\cdots|\partial\widetilde{\Gamma}^{k_p}\Omega^{l_p}\phi|\cdot|\partial^2\widetilde{\Gamma}^{k_{p+1}}\Omega^{l_{p+1}}\phi|.
 \end{equation}

\textbf{1. The treatment for the cases of $k=0$ and $l\leq 6$}
\begin{itemize}
  \item If $l_i\leq l_{p+1}$ holds for all $1\leq i\leq p$, then $l_i\leq 3\ (1\leq i\leq p)$,
 and it follows from Proposition \ref{global L infty estimate} and \eqref{global BA} that
        \begin{equation*}
         \begin{split}
          &\int_{D_t}|\p\Om^{\leq 6}\phi|\cdot|\p\Om^{l_1}\phi|\cdots|\p\Om^{l_p}\phi|\cdot|\p^2\Om^{l_{p+1}}\phi|\\
         \lesssim&\int_{D_t}|\p\Om^{\leq 6}\phi|\cdot|\p\Om^{\leq 3}\phi|^p\cdot|\p^2\Om^{\leq 5}\phi|\\
          \lesssim&\int_{t_0}^tM_0^p\de^{(\frac{13}{8}-\ve0)p}\tau^{-p}E_{0,l}(\tau)\d\tau
          +\int_{t_0}^tM_0^p\de^{(\frac{13}{8}-\ve0)p}\tau^{-p}E_{1,l}(\tau)\d\tau\\
          \lesssim&\de^{4-2\ve0}.
         \end{split}
        \end{equation*}
  \item If there exists one $i$ $(1\leq i\leq p)$ such that $l_i>l_{p+1}$, then $l_{p+1}\leq 2$, and
        \begin{equation*}
         \begin{split}
          &\int_{D_t}|\p\Om^{\leq 6}\phi|\cdot|\p\Om^{l_1}\phi|\cdots|\p\Om^{l_p}\phi|\cdot|\p^2\Om^{l_{p+1}}\phi|\\
          \lesssim&\int_{D_t}|\p\Om^{\leq 6}\phi|\cdot|\p\Om^{\leq 3}\phi|^{p-1}\cdot|\p\Om^{\leq 6}\phi|\cdot|\p^2\Om^{\leq 2}\phi|\\
          \lesssim&\int_{t_0}^tM_0^p\de^{(\frac{13}{8}-\ve0)(p-1)+\frac{7}{8}-\ve0}\tau^{-p}E_{0,l}(\tau)\d\tau\\
          \lesssim&\delta^{4-2\ve0},
         \end{split}
        \end{equation*}
        where $(\frac{13}{8}-\ve0)(p-1)+\frac{7}{8}-\ve0>(1-\ve0)p-1>0$.
\end{itemize}
Therefore,
 \begin{equation*}
  E_{0,l}(t)\lesssim\de^{4-2\ve0}.
 \end{equation*}

\textbf{2. The treatment for the cases of $k=1$ and $l\leq 5$}
\begin{itemize}
  \item If $k_i+l_i\leq k_{p+1}+l_{p+1}$ for all $1\leq i\leq p$, then $k_i+l_i\leq 3$, and
        \begin{equation*}
         \begin{split}
          &\quad\int_{D_t}|\p\widetilde{\Ga}\Om^{\leq 5}\phi|\cdot|\p\widetilde{\Ga}^{k_1}\Om^{l_1}\phi|\cdots|\p\widetilde{\Ga}^{k_p}\Om^{l_p}\phi|\cdot|\p^2\widetilde{\Ga}^{k_{p+1}}\Om^{l_{p+1}}\phi|\\
          &\lesssim\int_{D_t}\{|\p\widetilde{\Ga}\Om^{\leq 5}\phi|\cdot|\p\Om^{\leq 3}\phi|^p\cdot|\p^2\widetilde{\Ga}\Om^{\leq 4}\phi|+|\p\widetilde{\Ga}\Om^{\leq 5}\phi|\cdot|\p\Om^{\leq 3}\phi|^{p-1}\cdot|\p\widetilde{\Ga}\Om^{\leq 2}\phi|\cdot|\p^2\Om^{\leq 5}\phi|\}\\
          &\lesssim\int_{t_0}^tM_0^p\de^{(\frac{13}{8}-\ve0)p}\tau^{-p}E_{1,l}(\tau)\d\tau
          +\int_{t_0}^tM_0^p\de^{(\frac{13}{8}-\ve0)p}\tau^{-p}E_{2,l}(\tau)\d\tau\\
          &\qquad +\int_{t_0}^tM_0^p\de^{(\frac{13}{8}-\ve0)(p-1)+\frac{7}{8}-\ve0}\tau^{-p}E_{1,l}(\tau)\d\tau\\
          &\lesssim\de^{4-2\ve0}.
         \end{split}
        \end{equation*}
  \item If there exists one $i$ $(1\leq i\leq p)$ such that $k_i+l_i>k_{p+1}+l_{p+1}$, then $k_{p+1}+l_{p+1}\leq 2$, and
        \begin{equation*}
         \begin{split}
          &\int_{D_t}|\p\widetilde{\Ga}\Om^{\leq 5}\phi|\cdot|\p\widetilde{\Ga}^{k_1}\Om^{l_1}\phi|\cdots|\p\widetilde{\Ga}^{k_p}\Om^{l_p}\phi|\cdot|\p^2\widetilde{\Ga}^{k_{p+1}}\Om^{l_{p+1}}\phi|\\
          \lesssim&\int_{D_t}|\p\widetilde{\Ga}\Om^{\leq 5}\phi|\cdot\{|\p\Om^{\leq 3}\phi|^{p-1}\cdot|\p\widetilde{\Ga}\Om^{\leq 5}\phi|\cdot|\p^2\Om^{\leq 2}\phi|+|\p\Om^{\leq 3}\phi|^{p-1}\cdot|\p\Om^{\leq 6}\phi|\cdot|\p^2\widetilde{\Ga}\Om^{\leq 1}\phi|\\
          &\qquad+|\p\Omega^{\leq 3}\phi|^{p-2}\cdot|\p\widetilde{\Ga}\Omega^{\leq 2}\phi|\cdot|\p\Omega^{\leq 6}\phi|\cdot|\p^2\Omega^{\leq 2}\phi|\}\\
          \lesssim&\de^{4-2\ve0}.
         \end{split}
        \end{equation*}
\end{itemize}
Therefore,
 \begin{equation*}
  E_{1,l}(t)\lesssim\de^{4-2\ve0}.
 \end{equation*}

\textbf{3. The treatment for the other left cases in \eqref{YHC-2}}

For the cases of $k+l\leq 6$ $(k\geq 2)$, the treatment
procedure is exactly similar to that for the cases of $k=0$ and $k=1$,
the details are omitted. Then we can get
$$
E_{k,l}(t)\lesssim\delta^{7-2k-2\ve0},
$$
which is independent of $M_0$.

Note that all the bounded constants above are independent of $M_0$.
Then the bootstrap assumptions \eqref{global BA} can be closed.
By combining with the local existence of solution $\phi$ to \eqref{main equation} and the continuous argument,
the global existence of $\phi$ in $B_{2\de}$ is established.

\vskip 0.3 true cm

Finally, we prove Theorem \ref{main theorem}.
\begin{proof}
By Theorem \ref{Theorem local existence}, we have got the local existence of smooth solution $\phi$ to
equation \eqref{main equation} with \eqref{initial data}- \eqref{condition on data}.
On the other hand,  the global existence of $\phi$ near  $C_0$
and in $B_{2\de}$ has been established in Section \ref{Section 8}.
Then it follows from the uniqueness of smooth solution to \eqref{main equation}
that the proof of $\phi\in C^\infty([1,+\infty)\times\Bbb R^3)$  is finished.
In addition, $|\p\phi|\lesssim\delta^{1-\varepsilon_0}t^{-1}$ comes from Theorem \ref{Theorem local existence},
\eqref{BA} and
the first inequality in \eqref{YHC-3}. Thus Theorem \ref{main theorem}
is proved.

\end{proof}

\section*{Appendix}

\appendix

\section{The existence of short pulse initial data with \eqref{condition on data}}\label{Section 3.1}
In this section, we give the existence of short pulse initial data which satisfy \eqref{condition on data}.

Due to $\p_t^2\phi=c^2(\p_r^2\phi+\frac{2}{r}\p_r\phi+\frac{1}{r^2}\Des\phi)$, then
\begin{align*}
{\t L}^2\phi&=c^2(\p_r^2\phi+\frac{2}{r}\p_r\phi+\frac{1}{r^2}\Des\phi)+\p_r^2\phi+2\p_r(\p_t\phi)\\
&=(c^2+1)\p_r^2\phi+\frac{2c^2}{r}\p_r\phi+\frac{c^2}{r^2}\Des\phi+2\p_r(\p_t\phi).
\end{align*}
Assume that $\phi_0(s,\om)\in C_0^{\infty}\left((-1,0)\times\mathbb{S}^2\right)$ of \eqref{initial data} is chosen as the fixed smooth function.
By virtue of \eqref{initial data}, the
derivatives of $\phi$ (up to the second order) can be computed on $\Si_1$ as follows
\begin{equation*}
\left\{
\begin{aligned}
&\p_t\phi=\de^{1-\ve0}\phi_1, \p_r\phi=\de^{1-\ve0}\p_s\phi_0,\\
&\p_r(\p_t\phi)=\de^{-\ve0}\p_s\phi_1, \p_r^2\phi=\de^{-\ve0}\p_s^2\phi_0, \Des\phi=\de^{2-\ve0}\Des\phi_0.
\end{aligned}
\right.
\end{equation*}
Then ${\t L}^2\phi|_{t=1}$ can be expressed as
\begin{equation*}
{\t L}^2\phi|_{t=1}=(c^2+1)\de^{-\ve0}\p_s^2\phi_0+\frac{2c^2}{r}\de^{1-\ve0}\p_s\phi_0+\frac{c^2}{r^2}\de^{2-\ve0}\Des\phi_0
+2\de^{-\ve0}\p_s\phi_1.
\end{equation*}
It is claimed that  $\phi_1$ in \eqref{initial data} can be chosen such that
\begin{equation}\label{YHC-4}
{\t L}^2\phi|_{t=1}=O(\de^{2-\ve0}),
\end{equation}
where $\phi_1(s,\om)\in C_0^{\infty}\left((-1,0)\times\mathbb{S}^2\right)$.
We now make the following assumption for $\phi_1$ with
\begin{equation}\label{ansatz for phi0}
|\p_s\phi_1|\leq C,
\end{equation}
where the constant $C>0$ depends only on $\phi_0$ but is independent of $\de$.

By \eqref{ansatz for phi0},
in order to show \eqref{YHC-4}, it suffices to prove
\begin{equation}\label{Y-4}
(1+\frac{1}{c^2})\p_s^2\phi_0+\frac{2}{r}\de\p_s\phi_0+2c^{-2}\p_s\phi_1=O(\de^2).
\end{equation}
Since $(1+\frac{1}{c^2})^{-1}=\frac{1}{2}-\frac{1}{4}\de^{(1-\ve0)p}\phi_1^p+O(\de^{(1-\ve0)2p})$,
it then follows from \eqref{Y-4} that
\begin{equation*}
\p_s^2\phi_0+(\frac{2}{r}\de\p_s\phi_0+2c^{-2}\p_s\phi_1)\big(\frac{1}{2}-\frac{1}{4}\de^{(1-\ve0)p}\phi_1^p+O(\de^{(1-\ve0)2p})\big)=O(\de^2).
\end{equation*}
This, together with $p>p_c=\f{1}{1-\ve0}$, yields
\begin{equation}\label{Y-5}
\p_s^2\phi_0+\frac{\de}{r}(1-\frac{1}{2}\de^{(1-\ve0)p}\phi_1^p)\p_s\phi_0
+(1-\frac{1}{2}\de^{(1-\ve0)p}\phi_1^p)c^{-2}\p_s\phi_1=O(\de^2).
\end{equation}
Due to $r=1+s\de$ and $p>p_c$, in order to let \eqref{Y-5} hold,
then one can set
\begin{align*}
&\p_s^2\phi_0+\de\p_s\phi_0+(1+\frac{1}{2}\de^{(1-\ve0)p}\phi_1^p)\p_s\phi_1=0,
\end{align*}
which derives $F(\phi_1,\de)\equiv \p_s\phi_0+\de\phi_0+(\phi_1+\frac{\de^{(1-\ve0)p}}{2(p+1)}\phi_1^{p+1})=0$.
Note that $F(-\p_s\phi_0,0)=0$ and $\p_{\phi_1}F(\phi_1,\de)=1+\frac{1}{2}\de^{(1-\ve0)p}\phi_1^p>0$ for small $\de>0$.
Then it follows from the implicit function theorem and $F(\phi_1,\de)=0$ that $\phi_1(s,\om)
\in C_0^{\infty}\left((-1,0)\times\mathbb{S}^2\right)$ and \eqref{YHC-4} are obtained.
On the other hand, by $\p_t=\f12(\t L-\t{\underline L})$ and $\t{\underline L}\t L\phi=(c^2-1)\triangle \phi$, we have
\begin{equation*}
\p_r{\t L}\phi|_{t=1}=\frac{1}{2}[{\t L}^2\phi+(1-c^2)\p_r^2\phi-\frac{2c^2}{r}\p_r\phi-\frac{c^2}{r^2}\Des\phi]|_{t=1}=O(\de^{1-\ve0}).
\end{equation*}
Then it follows from the integration with respect to $r$ that
\begin{equation*}
|{\t L}\phi|_{t=1}|\leq\int_{1-\de}^1|\p_r{\t L}\phi|_{t=1}|\d r=O(\de^{2-\ve0}).
\end{equation*}

\vskip 0.5 true cm

\end{document}